\documentclass[11pt]{amsart}
\allowdisplaybreaks
\usepackage[english]{babel}
\usepackage[pdftex,textwidth=400pt,marginratio=1:1]{geometry}
\usepackage{amsfonts}
\usepackage{amsmath}
\usepackage{amsthm}
\usepackage{amssymb}
\usepackage{bbm}
\usepackage{cancel}
\usepackage{color}
\usepackage[curve]{xypic}
\usepackage{graphicx}
\usepackage{mathabx}
\usepackage{tikz-cd}
\usepackage{epigraph} 
\usepackage{hyperref}
\usepackage{stmaryrd}
\usepackage{enumitem}
\usepackage[mathscr]{eucal}
\usepackage{soul}
\usepackage{xfrac}

\newtheorem{theorem}{Theorem}[subsection]
\newtheorem{theoremX}{Theorem}
\newtheorem*{theorem*}{Theorem}
\newtheorem{corollary}[theorem]{Corollary}
\newtheorem{corollaryX}{Corollary}
\newtheorem*{corollary*}{Corollary}
\newtheorem{proposition}[theorem]{Proposition}

\newtheorem*{proposition*}{Proposition}
\newtheorem{lemma}[theorem]{Lemma}

\newtheorem*{conjecture*}{Conjecture}

\newtheorem{assumption}[theorem]{Assumption}
\theoremstyle{definition}
\newtheorem{definition}[theorem]{Definition}
\newtheorem{definitionX}{Definition}
\newtheorem*{definition*}{Definition}

\newtheorem*{variant*}{Variant}
\newtheorem{example}[theorem]{Example}
\newtheorem{exampleX}{Example}
\newtheorem*{example*}{Example}
\newtheorem{remark}[theorem]{Remark}
\newtheorem*{remark*}{Remark}

\newtheorem{notation}[theorem]{Notation}
\newtheorem*{notation*}{Notation}

\newtheorem*{question*}{Question}
\newtheorem*{acknowledgements}{Acknowledgements}

\usepackage[colorinlistoftodos]{todonotes}

\let\oldtocsection=\tocsection
\let\oldtocsubsection=\tocsubsection
\let\oldtocsubsubsection=\tocsubsubsection
\renewcommand{\tocsection}[2]{\hspace{0em}\oldtocsection{#1}{#2}}
\renewcommand{\tocsubsection}[2]{\hspace{1em}\oldtocsubsection{#1}{#2}}
\renewcommand{\tocsubsubsection}[2]{\hspace{2em}\oldtocsubsubsection{#1}{#2}}

\def\CAlg{\mathrm{CAlg}}

\def\Pic{\mathrm{Pic}}

\def\fC{\mathfrak{C}}
\def\BM{\mathrm{BM}}
\def\bbD{\mathbb{D}}
\def\X{\mathcal{X}}

\def\red{\mathrm{red}}

\def\ch{\mathrm{ch}}
\def\Kz{\mathrm{Kz}}
\def\Hd{\mathrm{Hd}}
\def\triv{\mathrm{triv}}

\def\ev{\mathrm{ev}}

\def\Ker{\mathrm{Ker}}
\def\Image{\mathrm{Im}}

\def\Perf{\mathrm{Perf}}

\def\loc{\mathrm{loc}}
\def\sp{\mathrm{sp}}
\def\SZ{\Zero^{\symp}}

\def\lag{\mathrm{lag}}

\def\rank{\mathrm{rank}}
\def\ori{\mathrm{or}}

\def\scD{\mathscr{D}}
\def\diag{\mathrm{diag}}
\def\doub{\mathrm{doub}}

\def\Def{\mathrm{Def}}
\def\Sp{\mathrm{Sp}}

\def\symp{\mathrm{symp}}

\def\E{\mathsf{E}}
\def\F{\mathsf{F}}
\def\sfD{\mathsf{D}}

\def\q{\mathfrak{q}}

\def\scS{\mathscr{S}}
\def\scZ{\mathscr{Z}}

\def\pr{\mathrm{pr}}
\def\id{\mathrm{id}}

\def\GG{\mathbb{G}}

\def\cart{\ar@{}[rd]|{\Box}}
\def\Spec{\mathrm{Spec}}
\def\tot{\mathrm{Tot}}
\def\Tot{\mathrm{Tot}}
\def\Cech{\mathrm{Cech}}
\def\coCech{\mathrm{coCech}}

\def\Span{\mathrm{Span}}
\def\pSymp{\mathrm{pSymp}}
\def\Symp{\mathrm{Symp}}
\def\Lag{\mathrm{Lag}}

\def\Crit{\mathrm{Crit}}

\def\dual{^{\vee}}

\def\isot{\mathrm{isot}}

\def\AA{\mathbb{A}}
\def\PP{\mathbb{P}}

\def\Z{\mathbb{Z}}
\def\Q{\mathbb{Q}}
\def\C{\mathbb{C}}

\def\O{\mathcal{O}} 
\def\I{\mathcal{I}} 

\def\LL{\mathbb{L}}
\def\TT{\mathbb{T}}
\def\T{\mathrm{T}}

\def\dSt{\mathrm{dSt}}
\def\dAff{\mathrm{dAff}}
\def\dPSt{\mathrm{dPSt}}

\def\QCAlg{\mathrm{QCAlg}}
\def\QCoh{\mathrm{QCoh}}

\def\Grpd{\mathrm{Grpd}}
\def\Cat{\mathrm{Cat}}
\def\Fun{\mathrm{Fun}}
\def\Adj{\mathrm{Adj}}

\def\Set{\mathrm{Set}}
\def\sSet{\mathrm{sSet}}

\def\map{\mathrm{Map}}
\def\Map{\mathrm{Map}}
\def\uMap{\underline{\Map}}
\def\fib{\mathrm{fib}}
\def\cof{\mathrm{cof}}
\def\Path{\mathrm{Path}}

\def\op{\mathrm{op}}
\def\fil{\mathrm{fil}}
\def\gr{\mathrm{gr}}
\def\forget{\mathrm{forget}}

\def\DR{\mathrm{DR}}
\def\hDR{\widehat{\DR}}
\def\Rees{\mathrm{Rees}}
\def\Gr{\mathrm{Gr}}
\def\Fil{\mathrm{Fil}}
\def\Sym{\mathrm{Sym}}

\def\textin{\quad\textup{in}\quad}
\def\and{\quad\textup{and}\quad}
\def\where{\quad\text{where}\quad }
\def\under{\quad\text{under}\quad}
\def\for{\quad\text{for }}

\def\sA{\mathscr{A}}
\def\lc{\mathrm{lc}}
\def\cl{\mathrm{cl}}
\def\ex{\mathrm{ex}}

\def\D{\mathrm{D}}
\def\DD{\mathrm{DD}}
\def\rmD{\mathrm{D}}
\def\hatD{\widehat{\rmD}}

\def\Zero{\mathrm{Z}}

\def\OT{\mathrm{OT}}

\begin{document}

\title{Shifted symplectic pushforwards}

\author[H.~Park]{Hyeonjun Park}
\address{June E Huh Center for Mathematical Challenges, Korea Institute for Advanced Study, 85 Hoegiro, Dongdaemun-gu, Seoul 02455, Republic of Korea}
\email{hyeonjunpark@kias.re.kr}

\date{June 27th, 2024}

\maketitle

\begin{abstract}
We introduce how to pushforward shifted symplectic fibrations along base changes.
This is achieved by considering symplectic forms that are closed in a stronger sense.
Examples include: symplectic zero loci and symplectic quotients.
Observing that twisted cotangent bundles are symplectic pushforwards, we obtain an equivalence between symplectic fibrations and Lagrangians to critical loci.

We provide two local structure theorems for symplectic fibrations:
a smooth local structure theorem for higher stacks via symplectic zero loci and twisted cotangents,
and an {\'e}tale local structure theorem for $1$-stacks with reductive stabilizers via symplectic quotients of the smooth local models.

We resolve deformation invariance issue in Donaldson-Thomas theory of Calabi-Yau $4$-folds.
Abstractly, we associate virtual Lagrangian cycles for oriented $(-2)$-symplectic fibrations as unique functorial bivariant classes over the exact loci.
For moduli of perfect complexes, we show that the exact loci consist of deformations for which the $(0,4)$-Hodge pieces of the second Chern characters remain zero.
\end{abstract}

\addtocontents{toc}{\protect\setcounter{tocdepth}{1}}

\section*{Introduction}

This paper aims to study {\em shifted symplectic fibrations}, that is, families of shifted symplectic derived Artin stacks, introduced in \cite{PTVV}.
We consider the symplectic categories $\Symp_{B,d}$ consisting of $d$-shifted symplectic fibrations $g:M \to B$ and their Lagrangian correspondences, constructed in \cite{Cal1,Hau}.

One fundamental question is the behavior of the symplectic categories under the base change.
Given a base change $p: U \to B$, there is an obvious pullback $p^*:\Symp_{B,d} \to \Symp_{U,d}$, but it is not obvious how to {\em pushforward} symplectic fibrations.

The main result in this paper is the existence of pushforwards in certain variants of symplectic categories.
We introduce {\em locked forms} (Definition \ref{Def_A:LF}) as stronger versions of closed forms and consider the $w$-locked versions of symplectic categories $\Symp_{B,d}^w$ for shifted functions $w:B \to \AA^1[d+2]$. 
(See \S\ref{ss:SympCat} for the precise definition.)

\begin{theoremX}[Symplectic pushforwards, Thm.~\ref{Thm:SP}]\label{Thm_A:SP}
Let $p:U \to B$ be a finitely presented morphism of derived stacks 
and $w:B \to \AA^1[d+2]$ be a $(d+2)$-shifted function.
Then there exists a right adjoint
\[p_* : \Symp_{U,d}^{w|_U} \to \Symp_{B,d}^w\]
of the pullback functor $p^*:\Symp_{B,d}^w \to  \Symp_{U,d}^{w|_U}$.
\end{theoremX}

The symplectic pushforwards can also be described explicitly via the zero loci of canonical {\em moment maps} (Proposition \ref{Prop:UnivMoment}).
Given a $w|_U$-locked symplectic fibration $h:N \to U$,
there exists a canonical Lagrangian $\mu_{N} : N \to \T^*_{U/B}[d+1]$ such that 
\[\xymatrix@R-.4pc{
p_*(N)\ar[r] \ar[d] \cart & U \ar[d]^{0} \\
N \ar[r]^-{\mu_N} & \T^*_{U/B}[d+1]
}\]
is a Lagrangian intersection diagram.

There are two main applications: 
\begin{enumerate}
\item (Symplectic) We provide an {\'e}tale local structure theorem for symplectic fibrations (Theorem \ref{Thm_B:LST}). This generalizes the Darboux theorem \cite{BBJ,BG} in the style of the local structure theorem for classical $1$-Artin stacks in \cite{AHR}.
\item (Enumerative) We construct virtual Lagrangian cycles for $(-2)$-symplectic fibrations (Theorem \ref{Thm_C:VLC}). This extends the virtual cycles for DT4 invariants in \cite{BJ,OT} to the relative setting and provides their characteristic properties.
\end{enumerate}

We can apply these results to moduli of perfect complexes on families of Calabi-Yau varieties (Theorem \ref{Thm_D:Perf}).
Especially for Calabi-Yau $4$-folds, we find interesting connections to the Hodge theory of surface classes.

\subsection*{Locked forms}

The concept of locked forms is motivated by a geometric description of closed forms.
Given a morphism of derived stacks $g:M \to B$, we can form a {\em deformation space}  \cite{HKR} as a family $\D_{M/B} \to \AA^1$ whose general fibers are the base $B$ and the special fiber is the normal bundle $\T_{M/B}[1]$.
Then closed forms are equivalent to $\GG_m$-equivariant formal functions on $\D_{M/B}$ (Proposition \ref{Prop:DRviaDefSp}),
\[\sA^{p,\cl }(M/B,d) \simeq \Map^{\GG_m}(\widehat{\D}_{M/B}, \AA^1(p)[p+d]),\]
where $\widehat{\D}_{M/B}$ is the formal completion of $\D_{M/B}$ at the special fiber $\T_{M/B}[1]$ and $\AA^1(p)$ is the weight $(-p)$ representation of $\GG_m$.
Thus we can view the closed forms as {\em formal} deformations of the ordinary differential forms (i.e. functions on $\T_{M/B}[1]$).

Our proposal is to consider the {\em global} functions on the deformation space.
\begin{definitionX}[Locked forms, Def.~\ref{Def:LF}, Prop.~\ref{Prop:DRviaDefSp}]\label{Def_A:LF}	
Let $g:M \to B$ be a morphism of derived stacks.
We define the space of $d$-shifted {\em locked $p$-forms} as:
\[\sA^{p,\lc}(M/B,d) := \Map^{\GG_m}(\D_{M/B}, \AA^1(p)[p+d]).\footnote{Alternatively, locked forms can be defined as sections of Hodge filtrations on {\em non-completed} de Rham complexes, 
while closed forms are given by their completions (see \S\ref{ss:DR}).
}\]
\end{definitionX}
Thus locked forms are {\em algebraic} deformations of differential forms to functions on the bases.
The {\em $w$-locked forms} are given by fixing the underlying functions, 
\[ \sA^{p,\lc }(M/B,d)^w:=\fib\left(\sA^{p,\lc }(M/B,d) \xrightarrow{\mathrm{gen}} \sA^0(B,p+d),w\right).\]
In this perspective, the {\em exact forms} (in \cite[\S5.1]{Toen}) can be viewed as $0$-locked forms.



\subsection*{Basic examples}

There are three basic examples of symplectic pushforwards.

\subsubsection*{Symplectic quotients}
Given a symplectic fibration $M\to B$ with a symplectic action of a group stack $G \to B$,
the quotient stack $M/G$ is usually not symplectic over the base $B$, but symplectic over the classifying stack $BG$.
To obtain a symplectic version of a quotient stack, we apply the pushforward along the projection $BG \to B$.

\begin{exampleX}[Symplectic quotients, Def.~\ref{Def:SQ}] \label{Ex_A}
Let $g:M \to B$ be a $w$-locked symplectic fibration for $w\in \sA^0(B,d+2)$ with a $w$-locked symplectic action of a smooth group stack $G \to B$.\footnote{Equivalently, $M/G \to BG$ is a $w|_{BG}$-locked symplectic fibration.}
We define the {\em  symplectic quotient} of $M$ by $G$ as:
\[M/\!/G:=(BG \to B)_*(M/G) \in \Symp^w_{B,d}.\]
\end{exampleX}

This is compatible with the Hamiltonian reduction in \cite{Cal1,Saf16}. 
There are two advantages of considering locked forms instead of closed forms:
\begin{enumerate}
\item (Existence/Uniqueness) 
The locked symplectic actions are already {\em Hamiltonian} in the sense that the symplectic quotients can be constructed without any additional data.
In the closed version, moment maps should be given as additional data whose existence or uniqueness is not guaranteed in general.
\item (Functoriality) Since the symplectic quotients are defined as symplectic pushforwards, various functorial properties follow immediately.
For instance, given an exact sequence of smooth group stacks $K \to G \to H$, we have
\[M/\!/G \simeq (M/\!/K)/\!/H \textin \Symp_{B,d}^w,\]
for an induced $w$-locked symplectic action of $H$ on $M{/\!/}K$ (Proposition \ref{Prop:Changeofgroups}).
\end{enumerate}

\subsubsection*{Symplectic zero loci}

The second example 
is the {\em symplectic zero locus} of a section of a symmetric complex.
It is a rather new example motivated to understand the local model in \cite{BBJ}.
Given a $(-2)$-symplectic scheme $M$, its {\em classical truncation} is locally the zero locus 
\[\xymatrix{
& E \ar[d] \\
M_\cl\simeq \Zero(s)_\cl \ar@{^{(}->}[r] & U, \ar@/_.3cm/[u]_s
}\]
of an isotropic section $s$ of an orthogonal bundle $E$ over a smooth scheme $U$.
However, the derived structure of $M$ is {\em not} given this way.
The usual derived zero locus $\Zero(s)$ is {\em not} $(-2)$-symplectic (unless the dimension is zero) since the cotangent complex $\LL_{\Zero(s)}\simeq \cof(E|_{\Zero(s)}\dual \to \Omega_U|_{\Zero(s)})$ is not $(-2)$-shifted symmetric.
It is natural to ask:
\[\textit{Is there a natural derived structure on $M_\cl$ that is $(-2)$-shifted symplectic?}	\]
In \cite[Ex.~5.12]{BBJ}, such derived structure is given for affine $U$ with explicit cdga representatives,
but our desire is to find an ``intrinsic'' construction without choosing cdga representatives.
A key observation is that $\Zero(s)$ is {\em relatively} $(-2)$-symplectic over $U$.
Our proposal is to apply the pushforward along the projection $U \to \Spec(\C)$.

\begin{exampleX}[Symplectic zero loci, Def.~\ref{Def:SympZero}] \label{Ex_B}
Let $p:U \to B$ be a finitely presented morphism of derived stacks, 
$E$ be a $(d+2)$-shifted symmetric complex on $U$,\footnote{Equivalently, $E$ is a perfect complex on $U$  with a symmetric $2$-form $\beta: \O_U \to \Sym^2(E\dual)[d+2]$ such that the induce map $\beta^\#:E \xrightarrow{} E\dual[d+2]$ is an equivalence.}
and $s : \O_U \to E$ be a section with $s^2 \simeq w|_U \in \sA^0(U,d+2)$ for some $w\in \sA^0(B,d+2)$.
We define the {\em symplectic zero locus} of $s$ in $U$ (over $B$) as:
\[\SZ_{U/B}(E,s):=(U \xrightarrow{p} B)_{*}(\Zero(s)) \in \Symp^{w}_{B,d},\]
where $\Zero(s) \in \Symp_{U,d}^{s^2}$ by Proposition \ref{Prop:LFZero}.
\end{exampleX}

The moment map description shows that the classical truncation remains the same, i.e. $\Zero^{\symp}_{U/B}(s)_\cl \simeq \Zero(s)_\cl$, if $p:U \to B$ is smooth and $d\leq -2$.
See \cite{AY} for the explicit comparison of $\SZ_{U/B}(E,s)$ with \cite[Ex.~5.12]{BBJ} for the smooth affine case.

\subsubsection*{Twisted cotangents}

The third example 
is a well-known one---the {\em twisted cotangent}.
It can be realized as the symplectic pushforward of the identity $\id_U:U \to U$.
Given a locked $1$-form $\alpha \in \sA^{1,\lc}(U/B,d+1)^w$, there is a canonical $w|_U$-locked symplectic form $0_\alpha$ on the identity map $\id_U$,\footnote{There is a unique symplectic form on $\id_U$, but it has several $w|_U$-locking structures, see \eqref{Eq:LSFidentity}.} such that the $\alpha$-twisted cotangent bundle is:
\[ \T^*_{U/B,\alpha}[d]:=U\times_{0,\T^*_{U/B}[d+1],\alpha}U \simeq p_*(U,0_\alpha) \textin \Symp_{B,d}^w.\]
In particular, the {\em critical locus} $\Crit_{U/B}(v):=\T^*_{U/B,d_{\DR }v}[d]$ of a shifted function $v:U \to \AA^1[d+1]$ can also be realized as a symplectic pushforward. This implies:

\begin{corollaryX}[Lagrangian factorizations, Cor.~\ref{Cor:LagFact}] \label{Cor_A:LagFact}
Let $p:U \to B$ be a finitely presented morphism of derived stacks.
Given a $v$-locked symplectic fibration $h: N \to U$ for $v:U \to \AA^1[d+2]$, there exists a canonical factorization
\[\xymatrix{
& \Crit_{U/B}(v) \ar[d] \\
N \ar[r]_-{h} \ar@{.>}[ru]^{\mu_N} & U,
}\]
by an exact Lagrangian $\mu_N$.
Moreover, this induces an equivalence of spaces
\begin{multline*}
\mu_{(-)}:\{\text{$d$-shifted $v$-locked symplectic fibrations over $U$}\} \\ 
\xrightarrow{\simeq} \{\text{$(d+1)$-shifted exact Lagrangians on $\Crit_{U/B}(v)$ (over $B$)}\}	.
\end{multline*}
\end{corollaryX}


\subsubsection*{Functoriality}
Since these examples are considered coherently as symplectic pushforwards, various functorial properties follow immediately. For instance, we have:
\[\SZ_{U/B}\left(E,s\right)/\!/G\simeq \SZ_{(U/G)/B}\left(E/G,s/G\right), \quad
\T^*_{V/B,\alpha}[d]/\!/G  \simeq \T^*_{(V/G)/B,\alpha/G}[d]\]
when  $(U,E,s)$ and $(V,\alpha)$ are given $G$-equivariant structures (see \S\ref{ss:SQ} for the definitions of equivariant structures).
In particular, this recovers \cite[Thm.~A]{AC}.

\subsection*{Local structures}

Our main application is an {\'e}tale local structure theorem for symplectic fibrations via the three basic examples of symplectic pushforwards.

\begin{theoremX}[{\'E}tale local structure, 
Cor.~\ref{Cor:ELS}]\label{Thm_B:LST}
Let $g:M \to B$ be a $d$-shifted $w$-locked symplectic fibration for $w\in \sA^0(B,d+2)$
such that $B_\cl$ is an algebraic space of finite type over $\C$,
$M_\cl$ is a quasi-separated $1$-Artin stack with affine stabilizers,
and $d< 0$. 
Let $m\in M(\C)$ be a point with linearly reductive stabilizer $\underline{\mathrm{Aut}}_{M}(m)$ and
$G:=\underline{\mathrm{Aut}}_{M}(m)\times B$.
\begin{enumerate}
\item 
If $d\equiv 2 \in \Z/4$ (resp. $d\equiv 0 \in \Z/4$), then there exist 
\begin{itemize}
\item a derived affine scheme $U$ of finite presentation over $B$ with a $G$-action 
such that $\LL_{U/B}$ is of tor-amplitude $\geq \frac d2+1$,
\item a $G$-equivariant orthogonal (resp. symplectic) bundle $E$ over $U$,
\item a $G$-invariant section $s : \O_U \to E[\tfrac d2 +1]$ with a $G$-equivalence 
$s^2 \simeq w|_{U}$, 
\item and a pointed {\'e}tale symplecto-morphism
\[ \left(\SZ_{U/B}\left(E\left[\tfrac d2+1\right],s\right)/\!/G,u\right) \longrightarrow \left(M,m\right).\]
\end{itemize}
\item If $d$ is odd, then there exist
\begin{itemize}
\item a derived affine scheme $V$ of finite presentation over $B$ with a $G$-action 
such that $\LL_{V/B}$ is of tor-amplitude $\geq \frac d2$,
\item a $G$-invariant $(d+1)$-shifted $w$-locked $1$-form $\alpha$ on $V$, 
\item  and a pointed {\'e}tale symplecto-morphism
\[\left(\T^*_{V/B,\alpha}[d]/\!/G,v\right) \longrightarrow \left(M,m\right).\]
\end{itemize}
\end{enumerate}
\end{theoremX}


Theorem \ref{Thm_B:LST} extends the derived Darboux theorem \cite{BBJ,BG} in three perspectives:
\begin{enumerate}
\item (Relative) The main new feature in the relative setting is the role of the underlying function $w$.
In the absolute case (i.e. $B=\Spec(\C)$), all negatively shifted closed forms are exact, that is, $0$-locked (Remark \ref{Rem:Absolute}). 
However, there are many non-exact closed forms in the relative setting; even the canonical symplectic forms on moduli spaces can be non-exact (Remark \ref{Rem:nonexactmoduli}).
Therefore, the local models become the symplectic zero loci of sections that are {\em not necessarily isotropic}, 
or the twisted cotangents instead of critical loci.
\item (Stacky) Theorem \ref{Thm_B:LST} can be viewed as a combination of the {\em {\'e}tale} local structure theorems for symplectic derived schemes in \cite{BBJ,BG} and classical $1$-Artin stacks in \cite{AHR}. In particular, this refines the {\em smooth} local structure theorem for symplectic $1$-Artin stacks in \cite{BBBJ}.
\item (Coordinate-independence) Even for symplectic schemes in the absolute case, the local models in Theorem \ref{Thm_B:LST} are {\em intrinsic} in the sense that choices of quasi-free cdga representatives are not required.
\end{enumerate}

Theorem \ref{Thm_B:LST} can be applied to derived $1$-Artin stacks with good moduli spaces \cite{Alp} (Remark \ref{Rem:goodmoduli}).
Theorem \ref{Thm_B:LST} recovers the Lagrangian neighborhood theorem \cite{JS} since Lagrangians are locally equivalent to symplectic fibrations by Corollary \ref{Cor_A:LagFact} (Remark \ref{Rem:LNT}).
There is also a $0$-shifted version of Theorem \ref{Thm_B:LST} (Remark \ref{Rem:0-symp}).

\subsubsection*{Symplectic pushforward towers}
We sketch how Theorem \ref{Thm_B:LST} is obtained.
Based on the classical local structure theorem \cite{AHR}, 
we can find a quotient stack presentation:
\[M \simeq L/G\quad\text{for a derived affine scheme $L$ with a $G$-action.}\]
The inductive description of derived affine schemes in \cite[Thm.~7.4.3.18]{LurHA} can be extended to the $G$-equivariant setting using the reductivity of $G$; we have a $G$-equivariant sequence of derived affine schemes
\[ L:=L_{(-d+1)} \hookrightarrow L_{(-d)} \hookrightarrow  \cdots \hookrightarrow L_{(1)} \hookrightarrow L_{(0)} \xrightarrow{\mathrm{sm}} L_{(-1)}:=B,\]
where $L_{(k+1)}\hookrightarrow L_{(k)}$ are the zero loci of sections of $(-k)$-shifted vector bundles for $k\geq0$ and $L_{(0)} \to L_{(-1)}$ is smooth.
Bending the sequence in the middle, we can form
\[\xymatrix@C-.7pc{
M\simeq M_{(-d+1)} \ar[r] \ar[d]^{}& M_{(-d)} \ar[r]\ar[d]^{} & M_{(-d-1)} \ar[r]\ar[d]^{} & \cdots \ar[r] & M_{\left(\lceil \frac {-d-1}{2} \rceil +1\right)} \ar[r] \ar[d]^{} & M_{\left(\lceil \frac {-d-1}{2} \rceil \right)} \ar[d]^{}  \\
B\simeq M_{(-2)} &  BG\simeq M_{(-1)} \ar[l] &  M_{(0)} \ar[l] & \cdots \ar[l] & M_{\left(\lfloor \frac {-d-1}{2} \rfloor -1 \right)} \ar[l] & M_{\left(\lfloor \frac {-d-1}{2} \rfloor \right) },\ar[l] 
}\]
where $M_{(k)}:=L_{(k)}/G$, $M_{(-2)}:=B$, and the last vertial arrow is identity for odd $d$.\footnote{Here $\lceil \tfrac {-d-1}{2} \rceil:=\min\{i \in \Z :  i\geq \tfrac {-d-1}{2}\}$ and $\lfloor\tfrac {-d-1}{2}\rfloor:=\max\{i\in \Z : i \leq \tfrac{-d-1}{2}\}$.}
By choosing $L_{(\bullet)}$ of minimal dimensions (as in \cite[Thm.~4.1]{BBJ}),
we can inductively lift the locked symplectic form on $M \to B$ to $M_{(k)} \to M_{(-d-1-k)}$ such that 
 \[M_{(k)} \cong \left(M_{(-d-k)} \to M_{(-d-1-k)}\right)_*\left(M_{(k-1)}\right) \textin \Symp_{M_{(-d-1-k)},d}\]
(Theorem \ref{Thm:SPT}).
Then the local structure theorem follows by analyzing the locked symplectic structure on the last vertical arrow $M_{\left(\lceil \frac {-d-1}{2} \rceil \right)} \to M_{\left(\lfloor \frac {-d-1}{2} \rfloor \right) }$, which is either the zero locus of a section of a shifted vector bundle or the identity map. 

\subsubsection*{Smooth local structures}

As a variant of Theorem \ref{Thm_B:LST}, we also provide a {\em smooth} local structure theorem. 
Since there is no smooth symplecto-morphism (unless it is {\'e}tale or the shift is positive), we will use certain forms of Lagrangian correspondences as symplectic charts.
We say that a morphism $C:W \dashrightarrow M$ in $\Symp_{B,d}^w$ is a {\em smooth symplectic cover} if the corresponding Lagrangian correspondence
\[ \xymatrix@R-0.1pc{
& C \ar@{_{(}->}[ld]_{\mathrm{cl.eq}}^{} \ar@{->>}[rd]^{\mathrm{sm.surj}}_{} & \\
W && M,
}\]
consists of a smooth surjective map $C \to M$ and a map $C\to W$ whose classical truncation is an equivalence.

\begin{theorem*}[Smooth local structure, Cor.~\ref{Cor:SLS}]
Let $g:M \to B$ be a $d$-shifted $w$-locked symplectic fibration for $w\in \sA^0(B,d+2)$. 
Assume that $M$ and $B$ are derived Artin stacks whose classical truncations are of finite type over $\C$ and $d<0$.
\begin{enumerate}
\item If $d\equiv 2 \in \Z/4$ (resp. $d\equiv 0 \in \Z/4$), 
then there exist 
\begin{itemize}
\item a finitely presented morphism $p:U \to B$ from a derived scheme $U$ such that $\LL_{U/B}$ is of tor-amplitude $\geq \frac d2 +1$,
\item an orthogonal (resp. symplectic) bundle $E$ on $U$, 
\item a section $s: \O_U \to E [\tfrac d2+1]$ with $s^2 \cong w|_U$, 
\item and a smooth symplectic cover
\[C:\SZ_{U/B}(E\left[\tfrac d2+1\right],s) \dashrightarrow M.\]
\end{itemize}

\item If $d\not\equiv2 \in \Z/4$ or ($d\equiv 2 \in \Z/4$ and $\rank(\TT_{M/B})$ is even), then there exist 
\begin{itemize}
\item  a finitely presented morphism $q:V \to B$ from a derived scheme $V$ such that $\LL_{V/B}$ is of tor-amplitude $\geq \frac d2$, 
\item a locked $1$-form $\alpha \in \sA^{1,\lc }(V/B,d+1)^w$, 
\item and a smooth symplectic cover
\[C: \T^*_{V/B,\alpha}[d] \dashrightarrow M.\]
\end{itemize}
\end{enumerate}
\end{theorem*}

This is a generalization of the smooth local structure theorem for symplectic $1$-Artin stacks in \cite{BBBJ} to symplectic fibrations of (higher) Artin stacks.
Unlike Theorem \ref{Thm_B:LST}, here the twisted cotangent bundles can also be used for even $d$.
Since orthogonal (resp. symplectic) bundles over {\em schemes}, {\'e}tale locally have maximal isotropic (resp. Lagrangian) subbundles,
the symplectic zero loci become twisted cotangent bundles (by Proposition \ref{Prop:ChangeofSymmCplx}), except the case when $d\equiv 2$ and  $\rank(\TT_{M/B})$ is odd.

\subsection*{Virtual Lagrangian cycles}

Our main enumerative application is the existence of unique functorial bivariant classes for $(-2)$-symplectic fibrations over the exact loci.

\begin{theoremX}[Virtual Lagrangian cycles, Thm.~\ref{Thm:VLC}]\label{Thm_C:VLC}
Let $g:M \to B$ be an oriented $(-2)$-shifted $w$-locked symplectic fibration.
Assume that $M_\cl$ is a quasi-projective scheme and $B_\cl$ is an $1$-Artin stack with affine stabilizers.
Then there exists a map
\[[M/B]^\lag : A_*(\Zero_B(w)) \to A_{*+\frac12\rank(\TT_{M/B})}(M),\] 
where $\Zero_B(w)$ is the zero locus of $w:B \to \AA^1$ in $B$, satisfying the following properties:
\begin{enumerate}
\item (Bivariance) $[M/B]^\lag$ commutes with projective pushforwards and quasi-smooth (quasi-projective) pullbacks \cite{BF,Man}.
\item (Functoriality) Consider an oriented Lagrangian correspondence 
\[\xymatrix{
& L\ar[ld]_g \ar[rd]^f \\
M && N
}\]
of oriented $w$-locked $(-2)$-symplectic fibrations over $B$ such that $M_\cl$ and $N_\cl$ are quasi-projective schemes.
Then we have:
\[\qquad\text{$f:$ quasi-smooth \& $g_\cl:$  isomorphism} \implies [M/B]^\lag = f^! \circ [N/B]^\lag.\]
\end{enumerate}
Moreover, the maps $[M/B]^\lag$ are uniquely determined by the above two properties.
\end{theoremX}

Theorem \ref{Thm_C:VLC} extends the virtual Lagrangian cycles \cite{BJ,OT} and their functoriality \cite{Park} in three perspectives:

\begin{enumerate}
\item (Relative) 
Theorem \ref{Thm_C:VLC} provides a cycle-theoretical {\em deformation invariance} of the virtual Lagrangian cycles.
A crucial feature is that it is {\em necessary} to restrict the bases to the loci where the symplectic forms are {\em exact} (see Remark \ref{Rem:ExactNecessary} for a counterexample without the exactness).
In terms of the local structures (Theorem \ref{Thm_B:LST}), the exactness of symplectic forms is equivalent to the sections of orthogonal bundles being isotropic.
This feature is overlooked in  \cite{BJ,OT} (see Remark \ref{Rem:BJ/OT}); their constructions are valid in the absolute case, 
but not in the relative case where symplectic forms can be non-exact.

\item (Derived) 
Theorem \ref{Thm_C:VLC} gives a derived interpretation of its classical shadow studied via the language of obstructions theories in \cite{Park}.
The necessary technical conditions in the classical perspective become natural in the derived perspective (Remark \ref{Rem:SOT}).
Especially, an artificial compatibility condition \cite[Def.~2.1]{Park} to obtain the functoriality is automatic for Lagrangian correspondences---nothing but the morphisms in the symplectic categories.\footnote{This is analogous to the fact that the compatibility condition in \cite{Man} to obtain the functoriality of the virtual fundamental cycles in \cite{LT,BF} holds automatically for quasi-smooth morphisms of quasi-smooth derived schemes.}
\item (Uniqueness)
Extending the virtual Lagrangian cycles to the relative setting allows us to characterize them uniquely by the bivariance and functoriality.
\end{enumerate}

We sketch how the uniqueness part of Theorem \ref{Thm_C:VLC} is obtained.
The key ingredient is {\em symplectic deformations} to normal bundles (Corollary \ref{Cor:SD}).
More precisely, given a locked symplectic fibration $g:M \to B$, there exists a  locked symplectic form on 
\[M\times \AA^1 \to \D_{M/B}.\]
This can be achieved by observing that the double deformation space $\D_{M\times\AA^1/\D_{M/B}}$ is the ``doubling'' of the ordinary deformation space $D_{M/B}$ (Lemma \ref{Lem:DD}).
By the bivariance, we can replace $g:M \to B$ with the zero section
$0:M \to \T_{M/B}[1]$.
By the functoriality, we can replace the symmetric complex $\T_{M/B}[1]$ with an orthogonal bundle $E$. 
The zero section can be described as the symplectic zero locus
\[\xymatrix{
& E|_E \ar[d] \\
M\simeq \SZ_{E/E}(E|_E,\tau) \ar@{^{(}->}[r]^-{0_E} & E, \ar@/_.3cm/[u]_\tau
}\]
of the tautological section $\tau$.
In this simplest case, we can consider the map
\[[M/E]^\lag:=\sqrt{e}(E|_{E},\tau) : A_*(QE:=\Zero(\tau^2)) \to A_*(M\simeq \Zero(\tau)),\]
constructed in \cite{OT}, as a localization of the characteristic class $\sqrt{e}\in A^*(BSO(2r))$ in \cite{EG}.
The bivariance, funtoriality, and uniqueness of this map is shown in \cite{KP}.

\subsection*{Moduli of perfect complexes}

Our main example of a locked symplectic fibration is the moduli of perfect complexes for a family of Calabi-Yau varieties.

\begin{theoremX}[Moduli of perfect complexes, Cor.~\ref{Cor:Map}, Ex.~\ref{Ex:Perf}]\label{Thm_D:Perf}
Let $f: X \to B$ be a smooth projective Calabi-Yau morphism of classical schemes of dimension $n \geq 4$ 
and $\ch_* \in H^{2*}_{\DR}(X/B)^{\nabla}$ be a horizontal section with respect to the Gauss-Manin connection $\nabla$.
Let $\Perf(X/B,\ch_*)$ be the moduli stack of perfect complexes on the fibers of $f$ whose topological Chern characters are the fibers of $\ch_*$.
Then we have
\[\Perf(X/B,\ch_*) \in \Symp_{B,2-n}^{W}, \quad\text{where}\quad W:=\int_{X/B} \ch_2\cup \Omega,\]
and $\Omega  \in H^0(X, \Omega^n_{X/B})=\Fil^n_{\Hd}H^n_{\DR}(X/B) \subseteq H^n_{\DR}(X/B)$ is the Calabi-Yau $n$-form. 
\end{theoremX}

There is a general result for locking $d$-shifted closed $p$-forms for $d\leq -p$ over classical bases (Proposition \ref{Prop:Locking}):
\begin{itemize}
\item If $d<-p$, then all closed forms have unique exact ($0$-locking) structures.
\item If $d=-p$, then closed forms have unique locking structures formally locally on the bases (on each connected components).
Moreover, global locking structures exist if and only if the underlying (formal) functions of the unique locking structures on the formal neighborhoods converge to global functions.
\end{itemize}
These are consequences of the equivalence of the derived de Rham cohomology and the classical algebraic de Rham cohomology \cite{Har}, shown in \cite{Bhatt}.
Then Theorem \ref{Thm_D:Perf} follows by computing the underlying formal functions, which is $\int_{X/B} \ch_2\cup \Omega$.

If the base $B$ is reduced, then $\ch_2$ is a Hodge class by the global invariant cycle theorem \cite[Thm.~4.1.1]{Del71} (see \cite[Prop.~11.3.5]{CS}),
and the underlying function $W$ is zero. 
This happens because having a global horizontal class over an algebraic base is quite a strong condition;
if we do not fix the Chern characters by a horizontal class, 
then the symplectic form can be non-exact even for reduced bases (Remark \ref{Rem:nonexactmoduli}).

\subsubsection*{Vanishing cycles on Hodge loci}

We finally combine our results to Calabi-Yau $4$-folds and discuss a connection to the Hodge theory.
Here we will assume that everything in this paper extends to the analytic setting,
which provides a simpler heuristic picture.

Let $X$ be a smooth projective Calabi-Yau $4$-fold. Then $X$ has a local universal deformation $f:\X \to \bbD$ over a polydisc $\bbD \subseteq H^1(X,T_X)$ by \cite{Bog,Tian,Tod}.
Consider any moduli space $M$ of perfect complexes on the fibers of $f$ whose Chern characters are pullbacks of a horizontal section $\ch_* \in H^{2*}_{\DR}(\X/\bbD)^\nabla \simeq H^{2*}_{\DR}(X)$.

An analytic version of Theorem \ref{Thm_D:Perf} will give us a $(-2)$-shifted locked symplectic form on $M \to \bbD$ whose underlying function is: 
$W:= \int_{\X/\bbD} \ch_2 \cup \Omega : \bbD \to \C.$\footnote{This function $W$ is called the Gukov-Vafa-Witten superpotential \cite{GVW} in the physics literature.}
The zero/critical locus of this function $W$ has a Hodge-theoretical description:
\begin{align*}
\mathrm{Zero}(W) &= \left\{d\in \bbD : (\ch_2)_d \in \Fil^1_{\Hd}H_{\DR}^4(\X_d)\right\}	,\\
\Crit(W) &= \left\{d\in \bbD : (\ch_2)_d \in \Fil^2_{\Hd}H_{\DR}^4(\X_d)\right\}. 	
\end{align*}
In particular, $\Crit(W)$ is the Hodge locus of $\ch_2$ and expected to be $(-1)$-symplectic.

An analytic version of Theorem \ref{Thm_C:VLC} will give us a canonical map 
\[[M/\bbD]^\lag : H^{\BM}_*(\mathrm{Zero}(W) ,\Q) \to H^{\BM}_*(M,\Q)\]
in the Borel-Moore homology.
This means that the DT4 invariants are invariant under deformations of $X$ for which the $(0,4)$-Hodge pieces of $\ch_2$ remain zero.

An analytic version of Corollary \ref{Cor_A:LagFact} will imply that the  factor $M \to \Crit(W)$ is a $(-1)$-shifted Lagrangian.  
Then an analytic version of the Joyce conjecture \cite[Conj.~1.1]{JS} will give us a canonical map
\[[M/\bbD]^{\mathrm{crit}} : H^*(\Crit(W),\phi) \to H^{\BM}_*(M,\Q),\]
where $\phi:=\phi_W(\Q_{\bbD}[\dim \bbD] )$ is the perverse sheaf of vanishing cycles.
This map can be viewed as a refined DT4 invariant in terms of the singularities of the Hodge loci.

\subsection*{Further applications}

The results in this paper can be used in various contexts:

\begin{enumerate}
\item (Variational Hodge conjecture) Theorem \ref{Thm_D:Perf} ensures that the reduced virtual cycles for counting surfaces (constructed in \cite[Thm.~1.6]{BKP}) detect the variational Hodge conjecture (as stated in \cite[Thm.~1.13]{BKP}).
\item (Cohomological Hall algebras) The {\'e}tale local structure theorem (Theorem \ref{Thm_B:LST}) for $(-1)$-symplectic Artin stacks is used in \cite{KPS} to construct cohomological Hall algebras for $3$-Calabi-Yau categories.
\item (Cosection localization) In a forthcoming paper with Young-Hoon Kiem \cite{KP2}, we use the locked forms (Definition \ref{Def_A:LF}) and symplectic deformations (Corollary \ref{Cor:SD}) to provide an intrinsic description of the cosection-localized virtual cycles in \cite{KL} via the virtual Lagrangian cycles (Theorem \ref{Thm_C:VLC}).
\item (Symplectic rigidifications) In a forthcoming paper with Jemin You \cite{PY}, we use symplectic pushforwards (Theorem \ref{Thm_A:SP}) to construct a symplectic version of rigidification of $\Perf(X)$ by the action of $B\GG_m$ (for a Calabi-Yau $X$).
\end{enumerate}
Theorem \ref{Thm_B:LST}, Theorem \ref{Thm_C:VLC}, and Theorem \ref{Thm_D:Perf} are also used for studying quasi-maps to critical loci \cite{CZ} and the degeneration formula for local Calabi-Yau $4$-folds \cite{CZZ}.

\addtocontents{toc}{\protect\setcounter{tocdepth}{2}}

\begin{acknowledgements}
Special thanks to Richard Thomas for crucial contributions to this project.
He shared his expectations on relative moduli of sheaves that the $(0,4)$-Hodge pieces of second Chern characters are the obstructions to exactness and those moduli spaces are Lagrangians of the Hodge loci.
Most of the main results were achieved to realize these expectations.

The author is thankful to Younghan Bae and Martijn Kool for the previous collaboration on counting surfaces which was the main motivation of this project. Much of the contents in this paper was discussed together.

The author particularly thanks 
Dominic Joyce and Jeongseok Oh for discussions to resolve the deformation invariance issue in DT4 theory;
Tasuki Kinjo for a suggestion to extend the Darboux theorem to stacks;
Chris Brav for a suggestion to consider de Rham complexes as filtered complexes;
Benjamin Hennion for a suggestion to prove the uniqueness of virtual Lagrangian cycles;
Adeel Khan for teaching him the derived deformation spaces;
Pavel Safronov for an explanation of exact structures on $(-1)$-shifted closed $2$-forms;
Yalong Cao for careful reading of preliminary versions of the paper and various suggestions.

The author is  grateful to
Qingyuan Jiang, Joost Nuiten, Marco Robalo, Michail Savvas, Bertrand To{\"e}n, and Jemin You for various helpful discussions on derived algebraic geometry and shifted symplectic structures;
Young-Hoon Kiem, Woonam Lim, Davesh Maulik, Yukinobu Toda, and Gufang Zhao for discussions on possible applications to enumerative geometry;
Arkadij Bojko and Nachiketa Adhikari for sharing their drafts on related works.

The author was supported by Korea Institute for Advanced Study (SG089201).
\end{acknowledgements}

\subsection*{Notation and conventions}

We use the language of $\infty$-categories in \cite{LurHTT,LurHA}.
Denote by $\Grpd$ (resp. $\Cat$) the $\infty$-category of $\infty$-groupoids (resp. $\infty$-categories).
By abuse of notation, we usually suppress the symbol $\infty$; 
all categories, functors, limits, etc. are considered in the $\infty$-categorial sense, unless stated otherwise.

We work over the field of complex numbers $\C$.
Denote by 
\begin{itemize}
\item $\CAlg$ the category of commutative algebra spectra over $\C$,
\item $\CAlg^{\leq0}\subseteq \CAlg$ the full subcategory of connective objects,
\item $\dAff:=(\CAlg^{\leq0})^\op$ the category of {\em derived affine schemes},
\item $\dPSt:=\Fun(\dAff^\op,\Grpd)$ the category of {\em derived prestacks},
\item $\dSt\subseteq \dPSt$ the full subcategory of {\em derived stacks}, i.e., {\'e}tale sheaves on $\dAff$.
\end{itemize}
A {\em derived Artin stack} is a derived stack that is $n$-geometric for some $n$ in the sense of \cite{TV-HAG2} or \cite{LurDAG}.
A {\em derived scheme} (resp. {\em derived algebraic space}, 
{\em derived $1$-Artin stack}) $M$ is a derived Artin stack whose classical truncation $M_\cl$ is a scheme (resp. algebraic space, 
$1$-Artin stack).

We use the following conventions on a morphism of derived stacks $g:M \to B$:
\begin{itemize}
\item $g$ is {\em geometric} if $g$ is relatively representable by derived Artin stacks,
\item $g$ is {\em of finite type} if $g$ is geometric and $g_\cl$ is locally of finite type,
\item $g$ is {\em of finite presentation} if $g$ is geometric and locally of finite presentation.
\end{itemize}
For derived affine schemes, we follow the definition of finite type/presentation in \cite[Chap.~4]{LurSAG}.
For derived stacks, we drop the term ``locally'' for the simplicity of notation;
we are not assuming the quasi-compact/quasi-separated conditions.

Given a derived stack $B$, denote by:
\begin{itemize}
\item $\QCoh_B$ the symmetric monoidal category of quasi-coherent sheaves, 
\item $\QCAlg_B:=\CAlg(\QCoh_B)$ the category of quasi-coherent algebras,
\item $\QCAlg_B^\fil:= \QCAlg_{B\times \AA^1/\GG_m}$ the category of filtered algebras,
\item $\QCAlg_B^\gr:= \QCAlg_{B\times B\GG_m}$ the category of graded algebras,
\item $\Gr^*:\QCAlg_B^\fil \to \QCAlg_B^\gr$ the functor of associated graded complexes, 
\item $\Fil^* :\QCAlg_B^\fil \to \QCAlg_B^\gr$ the functor of forgetting filtrations, 
\item $\QCAlg_B^{\fil,\fil}:= \QCAlg_{B\times \AA^1/\GG_m\times \AA^1/\GG_m}$ the category of doubly filtered algebras.
\end{itemize}
We always use the cohomological degrees for complexes 
and decreasing filtrations.
Note that $\QCoh_{B\times \AA^1/\GG_m}\simeq\Fun((\Z,\geq),\QCoh_B)$ and $\QCoh_{B\times B\GG_m}\simeq\Fun(\Z,\QCoh_B)$
by \cite{Mou} (and \cite[Cor.~5.7]{BP}).
We use the basic facts on filtered complexes in \cite{GP}.
The {\em completion} $\widehat{C} \in \QCoh_B^\fil$ of  $C\in \QCoh_B^\fil$ is given by $\Fil^p \widehat{C}:=\varprojlim_{q\geq p}\Fil^pC/\Fil^q C$.

The {\em total space} of a perfect complex $E$ is $\Tot(E):=\Spec(\Sym(E\dual))$.
By abuse of notation, we sometimes use the same letter $E$ to denote the associated total space.

Given a morphism of derived stacks $p:U \to B$, denote by:
\begin{itemize}
\item $p^*:=(-)\times_B U :\dSt_B:=\dSt_{/B} \to \dSt_U:=\dSt_{/U}$ the pullback functor,
\item $p_!:\dSt_U \to \dSt_B$ the forgetful functor (which is a left adjoint of $p^*$),
\item $p_*:\dSt_U \to \dSt_B$ the {\em Weil restriction} functor (i.e. right adjoint of $p^*$).
\end{itemize}

Given morphisms of derived stacks $M \to U \to B$, denote by
\begin{itemize}
\item $(-)|_M:\sA^{p,\lc}(U/B,d) \to \sA^{p,\lc}(M/B,d)$ the pullback map,
\item $(-)_{/U} :\sA^{p,\lc}(M/B,d) \to \sA^{p,\lc}(M/U,d)$ the restriction map.
\end{itemize}
We use the same symbols for pullbacks/restrictions of closed, exact, ordinary forms.


\tableofcontents
\addtocontents{toc}{\protect\setcounter{tocdepth}{2}}

\section{Locked forms}\label{Sec:LF}

In this preliminary section, we introduce our main objects (Definition \ref{Def_A:LF}): the {\em locked forms} on derived stacks.
We provide two equivalent descriptions of them:
\begin{enumerate}
\item sections of Hodge filtrations on {\em non-completed} de Rham complexes (\S\ref{ss:DR});
\item $\GG_m$-equivariant functions on deformations to normal bundles (\S\ref{ss:DefSp}).
\end{enumerate}

Throughout this section, 
let $M$ be a derived stack over a base derived stack $B$. 

\subsection{De Rham complexes}\label{ss:DR}


In this subsection, we define the locked forms via the de Rham complexes.
The usual closed forms in \cite{PTVV,CPTVV} can be obtained with the de Rham complexes replaced by their completions.

The functor of {\em de Rham complexes} (together with Hodge filtrations)
\[\DR_B : \dSt_B^\op \to \QCAlg_B^\fil\]
can be constructed through the following steps:
\begin{enumerate}
\item [(S1)] Let $\Rees_B : \QCAlg_B \to \QCAlg_B^\fil$ be the left adjoint of $\Gr^0$.
\item [(S2)] When $B$ is affine, define $\DR_B$ as the right Kan extension of the restriction of $\Rees_B$ to the connective objects.
\item [(S3)] In general, define $\DR_B:=\varprojlim	_{b:T \to B} b_* \circ \DR_T \circ (-\times_BT)$, where the limit is taken over all morphisms $b:T \to B$ from derived affine schemes $T$.
\end{enumerate}
When $M \to B$ is geometric with cotangent complex $\LL_{M/B}$,
the de Rham complex $\DR(M/B):=\DR_B(M)$ has the following associated graded/underlying objects:
\begin{equation}\label{Eq:DRun/gr}
\Gr^p\DR(M/B) \simeq (M \to B)_*(\Lambda^p\LL_{M/B}[-p]),\quad \Fil^0\DR(M/B)\simeq \O_B.
\end{equation}
The left equivalence follows from the universal property and descent of cotangent complexes\footnote{Since $\Gr\circ\Rees_B:\QCAlg_B \to \QCAlg_B^\gr$ is the left adjoint of the square-zero extension functor $C\mapsto \Gr^0C \oplus \Gr^1C[1]$, it is equivalent to $\Sym(\LL_{-/B}[-1])$ for connective objects, see \cite[Thm.~5.3.6]{Rak}. For the descent of (the wedge powers of) cotangent complexes, see \cite[Lem.~1.15]{PTVV}.}
and the right equivalence follows from the descent of structure sheaves.

We define the locked forms via sections of the Hodge filtrations.

\begin{definition}[Locked forms]\label{Def:LF}
The space of $d$-shifted {\em locked $p$-forms} is:
\[\sA^{p,\lc}(M/B,d):= \map_{\QCoh_B}(\O_B, \Fil^p \DR(M/B)[p+d]).\]
\end{definition}

The locked forms induce {\em underlying forms/functions};
we have canonical maps
\[\xymatrix{
 & \sA^{p,\lc}(M/B,d) \ar[ld]_{(-)^p} \ar[rd]^-{[-]} & \\
 \sA^{p}(M/B,d) && \sA^{0}(B,p+d)
}\]
induced by $\Gr^p \leftarrow \Fil^p \rightarrow \Fil^0$, 
where 
\begin{itemize}
\item $\sA^{p}(M/B,d):=\map(\O_B,\Gr^p\DR(M/B)[p+d])$ consists of $d$-shifted {\em $p$-forms},
\item $\sA^0(B,p+d):=\map(\O_B,\O_B[p+d])$ consists of $(p+d)$-shifted {\em functions}.
\end{itemize}
Since the underlying functions will play the crucial role, we use the following notation:

\begin{notation}[$w$-locked forms] \label{Notation:w-locked}
Let $w\in \sA^0(B,p+d)$.
Denote by
\[\sA^{p,\lc}(M/B,d)^w:= \fib\left([-]:\sA^{p,\lc}(M/B,d) \xrightarrow{}\sA^0(B,p+d),w\right).\]
\end{notation}

We compare the locked forms with the usual {\em closed} forms (in \cite{PTVV,CPTVV}).
Let $\hDR(M/B) \in \QCAlg_B^\fil$ be the completion of $\DR(M/B)$.\footnote{The pushforward of the completed de Rham complex $\hDR(M/B)$ along $B \to \Spec(\C)$ is equivalent to the de Rham complex defined in \cite[Def.~1.3.9, 2.3.1, 2.4.2]{CPTVV}, under the canonical equivalence of completed filtered complexes and graded-mixed complexes (e.g. \cite[Prop.~1.3.1]{TV20}).
}
Denote by
\begin{itemize}
\item $\sA^{p,\cl}(M/B,d):=\map(\O_B,\Fil^p\hDR(M/B)[p+d])$,
\item $\sA^{\DR}(M/B,d):=\map(\O_B,\Fil^0\hDR(M/B)[d])$, 
\item $\sA^{p,\ex}(M/B,d):=\fib(\sA^{p,\cl}(M/B,d) \to \sA^{\DR}(M/B,p+d))$,
\end{itemize}
the spaces of {\em closed}, {\em de Rham}, {\em exact} forms.
From the usual perspective, the locked forms can be viewed as intermediate notions between closed forms and exact forms;
we have a cartesian diagram
\begin{equation}\label{Eq:ex/cl/DR}
\xymatrix{
\sA^{p,\ex}(M/B,d) \ar[r] \ar[d] \cart & \sA^{p,\lc}(M/B,d) \ar[r]^-{\widehat{(-)}} \ar[d]^{[-]} \cart & \sA^{p,\cl}(M/B,d) \ar[d]^{} \\
\star \ar[r]^-{0} & \sA^{0}(B,p+d) \ar[r] & \sA^{\DR}(M/B,p+d)
}
\end{equation}
since $\Gr^*\hDR(M/B)\simeq \Gr^*\DR(M/B)$, 
where $\star$ is a contractible space.
\begin{itemize}
\item At the one hand, the locked forms are stronger versions of closed forms in the sense that $\hDR$ is replaced with $\DR$;
the locked forms are closed forms whose associated de Rham classes come from functions on the bases.
\item On the other hand, the locked forms are twisted versions of exact forms;
\begin{equation}\label{Eq:TwistedExact}
\quad\sA^{p,\lc}(M/B,d)^w \simeq \Path_{0,0_w|_M} \sA^{p,\ex}(M/B,d+1) \quad \text{for $w\in \sA^0(B,p+d)$},
\end{equation}
where $0_w \in \sA^{p,\ex}(B/B,d+1) \simeq \sA^0(B,p+d)$ is the induced exact form.
\end{itemize}



\subsection{Deformation spaces}\label{ss:DefSp}


In this subsection, we provide geometric descriptions of locked forms (and closed forms) via deformations to normal bundles.

The {\em deformation space} \cite{KhanRydh,Hekking,HKR} of $M\in \dSt_B$ is the mapping stack
\begin{equation}\label{Eq:DefSp}
\D_{M/B}:=\uMap_{B\times\AA^1}(B\times \{0\},M \times \AA^1).
\end{equation}
When $M \to B$ is geometric, then the special/generic fiber can be computed as:
\[\xymatrix{
\T_{M/B}[1] \ar@{^{(}->}[r] \ar[d] \cart & \rmD_{M/B} \ar[d]^{} \cart & B\times\GG_m \ar@{_{(}->}[l] \ar@{=}[d]\\
B \ar@{^{(}->}[r]^-{}  & B\times  \AA^1 & B\times\GG_m, \ar@{_{(}->}[l]_-{ }
}\]	
since $\T_{M/B}[1]:=\Spec(\Sym(\LL_{M/B}[-1])) \simeq \uMap_{B}(B\times\AA^1[-1],M)$ by the universal property of cotantent complexes.
Moreover, the deformation space $\rmD_{M/B}$ has a canonical $\GG_m$-action induced by the weight $1$ action on $\AA^1$ whose quotient stack is:
\[\D_{M/B}/\GG_m\simeq \uMap_{B\times \AA^1/\GG_m}(B\times B\GG_m,M \times \AA^1/\GG_m).\]
In particular, the equivariant functions on $\D_{M/B}$ form a filtered algebra
\[\Gamma^\fil_B(\rmD_{M/B}/\GG_m):=(\rmD_{M/B}/\GG_m \to B\times \AA^1/\GG_m)_*\O_{\rmD_{M/B}/\GG_m} \in \QCAlg_{B}^\fil.\]

We can recover the de Rham complex from the deformation space.

\begin{proposition}[De Rham via deformation]\label{Prop:DRviaDefSp}
We have a canonical equivalence 
\[\DR(M/B) \simeq \Gamma^\fil_B(\rmD_{M/B}/\GG_m) \textin  \QCAlg_{B}^\fil.\]
\end{proposition}

Proposition \ref{Prop:DRviaDefSp} says that the locked $p$-forms can be viewed as weight $p$ functions on the deformation space and their underlying $p$-forms/functions are the restrictions to the special/generic fiber;
we have a canonical commutative diagram
\[\xymatrix{
\sA^{p}(M/B,d) \ar[d] & \ar[l]_-{(-)^p}  \sA^{p,\lc}(M/B,d) \ar[r]^-{[-]} \ar[d] & \sA^0(B,p+d)\ar[d] \\
\sA^0(\T_{M/B}[1],p+d) & \ar[l]_-{\mathrm{sp}} \sA^0(\rmD_{M/B},p+d) \ar[r]^-{\mathrm{gen}} & \sA^0(B\times \GG_m,p+d),
}\]
where the upper row is the weight $p$ version of the lower row (see Lemma \ref{Lem:FtnsLinear} for the weight $p$ functions on $\T_{M/B}[1]$).
Moreover, by taking completions, we also have
\[\hDR(M/B) \simeq \Gamma^\fil_B(\hatD_{M/B}/\GG_m) \textin  \QCAlg_{B}^\fil, \]
where $\hatD_{M/B}:=\rmD_{M/B}\times_{\AA^1}\widehat{\AA}^1\simeq  \varinjlim_{n\to \infty} \rmD_{M/B}\times_{\AA^1}\Zero_{\AA^1}(T^n)$ is the formal completion of the deformation space along the special fiber.\footnote{For any $D \in \dSt_{B\times\AA^1/\GG_m}$, we have $\Gamma^\fil_B(\widehat{D}/\GG_m) \simeq \widehat{\Gamma_B^\fil(D/\GG_m)} \in \QCAlg_B^\fil$.
Indeed, if $D$ and $B$ are affine, the equivalence follows from the base change \cite[Prop.~3.10]{BFN} along $\Zero(T^n)/\GG_m \hookrightarrow \AA^1/\GG_m$.
In general, the equivalence follows from descent (as in Step 1 of Proposition \ref{Prop:DRviaDefSp} below).
}
Thus closed forms can be viewed as formal functions on the deformation space near the special fiber.


\begin{proof}[Proof of Proposition \ref{Prop:DRviaDefSp}]
Observe that there is an analogy between the construction of the de Rham complexes and the deformation spaces:
\begin{enumerate}
\item (De Rham) The functor $\Rees_B$ in \S\ref{ss:DR} is the left adjoint of
\[\QCAlg_B^\fil:= \QCAlg_{B\times\AA^1/\GG_m} \xrightarrow{\Gr\simeq 0^*} \QCAlg_B^\gr := \QCAlg_{B\times B\GG_m} \xrightarrow{\Gr^0\simeq (\pr_1)_*} \QCAlg_B.\]
\item (Deformation) The functor $\D_{-/B}/\GG_m$ in \eqref{Eq:DefSp} is the right adjoint of
\[\dSt_B^\fil:=\dSt_{B\times\AA^1/\GG_m} \xrightarrow{(-)^\gr:=0^*} \dSt_B^\gr:=\dSt_{B\times B\GG_m} \xrightarrow{\forget:=(\pr_1)_!} \dSt_B.\]
\end{enumerate}
Moreover the two functors are connected by the canonical adjunction
\[\xymatrix{
\Gamma_B :\dSt_B\ar@<.4ex>[r] & \ar@<.4ex>[l] \QCAlg_B^\op : \Spec_B,\quad\text{where $\Gamma_B:M \mapsto(M\to B)_*\O_M$.}
}\]	
Let $\Gamma_B^{\fil}:=\Gamma_{B\times\AA^1/\GG_m}$, $\Spec_B^{\fil}:=\Spec_{B\times\AA^1/\GG_m}$, $\Gamma_B^{\gr}:=\Gamma_{B\times B\GG_m}$, $\Spec_B^{\gr}:=\Spec_{B\times B\GG_m}$.

\medskip

\noindent{\em Step 1: Deformation spaces for affine schemes.}
We claim that there is an equivalence
\begin{equation}\label{Eq:DRdef1}
(\D_{-/B}/\GG_m) \circ \Spec_B \cong \Spec_B^\fil \circ \Rees_B: \QCAlg_B^\op \to \dSt^\fil_B
\end{equation}
if $B$ is affine.
By considering the left adjoints of the both sides, 
it suffices to show
\[\Gamma_B \circ (-)^\gr \simeq \Gr^0 \circ \Gamma_B^\fil : \dSt^\fil_B \to \QCAlg_B.\]
For $M \in \dSt_B^\fil$, this is the weight $0$ part of the base change for the fiber square
\[\xymatrix{
M^\gr \ar@{^{(}->}[r] \ar[d] \cart & M \ar[d] \\
B\times B\GG_m \ar@{^{(}->}[r]^-{0} & B\times \AA^1/\GG_m.
}\]
Since $0^* : \QCoh_B^\fil \to \QCoh_B^\gr$ preserves limits, we may assume that $M$ is affine.
Since $M \to B\times\AA^1/\GG_m$ is affine, the base change follows from \cite[Prop.~3.10]{BFN}.

\medskip

\noindent{\em Step 2: Proposition \ref{Prop:DRviaDefSp} for affine schemes.}
Take the global sections of \eqref{Eq:DRdef1}, then 
\[\Gamma_B^\fil \circ (\D_{-/B}/\GG_m) \circ \Spec_B \cong \Gamma_B^\fil\circ \Spec_B^\fil \circ 
\Rees_B: \QCAlg_B \to \QCAlg_B^\fil.\]
It remains to show that the unit map of $\Gamma_B^\fil\dashv \Spec_B^\fil$ for connective objects,
\[\Rees_B \to \Gamma_B^\fil\circ \Spec_B^\fil \circ \Rees_B : \QCAlg_B^{\leq 0}\subseteq \QCAlg_B \to \QCAlg_B^\fil\]
is an equivalence.\footnote{The Rees algebra $\Rees_B(A) \in\QCAlg_B^\fil$ of a connective algebra $A\in\QCAlg_B^{\leq 0}$ is usually non-connective and thus it is a priori not obvious that the unit map is an equivalence.}
Since
$(\Gr,F^{-\infty}) : \QCAlg^\fil_B \to \QCAlg_B^\gr \times \QCAlg_B$
is conservative and we have \eqref{Eq:DRun/gr}, it suffices to show that the two unit maps
\begin{align*}
\Sym_B \to \Gamma_B^\gr\circ \Spec_B^\gr \circ \Sym_B &: \QCoh_B^{\leq 1} \to \QCAlg_B^\gr\\
\O_B \to \Gamma_B\circ \Spec_B(\O_B) & \textin \QCAlg_B,
\end{align*}
are equivalences.
Here we used that an affine scheme $M$ has a connective cotangent complex $\LL_{M/B}$ and the projection $\D_{M/B} \to B\times\AA^1/\GG_m$ has the base change property by \cite[Cor.~1.4.5(i)]{DG}.
Since the unit maps for abelian cones are equivalences (by Lemma \ref{Lem:FtnsLinear} below), we have Proposition \ref{Prop:DRviaDefSp} for affine $M$ and $B$.

\medskip

\noindent{\em Step 3: Proposition \ref{Prop:DRviaDefSp} for general derived stacks.}
Firstly, assume that $B$ is affine.
Any derived stack $M \in \dSt_B$ can be written as the colimit
$\varinjlim_i \Spec(A_i) \simeq M$
in the category of derived {\em prestacks}, where $A_i \in \QCAlg_B^{\leq0}$.
Then the induced map
\[\varinjlim_i \D_{\Spec(A_i)/B} \xrightarrow{\simeq} \D_{M/B}\]
is an equivalence of derived stacks (with $\GG_m$-actions) by the definition (in \eqref{Eq:DefSp}).
Since $\Gamma_B^\fil$ preserves limits, we obtain the desired equivalence since $\DR_B$ is the right Kan extension of its restriction to affine objects.

Secondly, consider a general derived stack $B$.
It is clear that the deformation spaces are compatible with a base change $T \to B$; we have:
$\D_{M/B}\times_B T\simeq \D_{M\times_B T/T}.$
In particular, the colimit over $\dAff_{/B}$ recovers the given deformation space;
we have
\[\D_{M/B} \simeq \varinjlim_{b:T \to B} \D_{M/B}\times_B T\simeq \varinjlim_{b:T \to B} \D_{M\times_BT/T}\]
since base changes preserve colimits in topoi.
Since $\Gamma^\fil_B$ preserves limits, we have the desired equivalence.
\end{proof}

We need the following lemma to complete the proof of Proposition \ref{Prop:DRviaDefSp}.
Let
\[C: \QCoh_B^{\op} \to \dSt_B^\gr : L \mapsto \Spec^\gr (\Sym(L))\]
be the functor of {\em abelian cones} (together with the scaling $\GG_m$-actions).

\begin{lemma}[Equivariant functions on abelian cones]\label{Lem:FtnsLinear}
Let $L \in \QCoh_B$ be a $n$-connective complex for some $n \in \Z$.
Then we have a canonical equivalence
\[\Gamma^\gr_B  C(L) \simeq \Sym(L) \textin \QCAlg_B^\gr,\]
induced by the adjunction $\Gamma_B^\gr \dashv \Spec_B^\gr$.
\end{lemma}

Lemma \ref{Lem:FtnsLinear} is shown in \cite{Mon} for perfect complexes (or the weight $p=1$ part).
These assumptions are required in the arguments in \cite{Mon} since the totalizations (i.e. limits of cosimplicial diagrams) may not commute with symmetric powers.
However this can be resolved by considering the {\em cosimplicial homotopies} \cite{Meyer}:
\begin{itemize}
\item For any cosimplicial complex $N \in \QCoh_B^\Delta$ and a simplicial set $K \in \sSet:=\Set^{\Delta^{\op}}$, denote by $N^K \in \QCoh_B^\Delta$ the composition
\[N^K : \Delta \xrightarrow{\mathrm{diag}}\Delta \times\Delta \xrightarrow{K\times N} \mathrm{Set}^{\op} \times \QCoh_B \xrightarrow{(S,E) \mapsto \prod_S E}\QCoh_B.\]
Then we have $\tot(N^K) \simeq \prod_{\pi_0(K)}\tot(N) $ since $\Delta$ is sifted \cite[Lem.~5.5.8.4]{LurHTT}.
\item A {\em cosimplicial homotopy} between two morphisms $f,g: N \to M$ in $\QCoh_B^\Delta$ is a map
$h : N \to M^{\Delta^1} \textin \QCoh_B^\Delta$
together with equivalences $h \circ s \simeq f$ and $h \circ t \simeq g$ where $s,t :M^{\Delta^1} \to M$ are induced by the two vertices $0,1\in \Delta^1$.
\item A {\em cosimplicial homotopy equivalence} $f: N \to M$ is a morphism in $\QCoh_B^{\Delta}$ such that there is another map $f' : M \to N$ where $f \circ f'$ and $f' \circ f$ have cosimplicial homotopies to the identity maps. 
\item A cosimplicial homotopy equivalence $f: N \to M$ induces an equivalence between the totalizations
$\tot(f) : \tot(N) \xrightarrow{\simeq} \tot(M).$
\end{itemize}

\begin{proof}[Proof of Lemma \ref{Lem:FtnsLinear}]
By descent, we may assume that $B$ is affine.
Indeed, if we write $B\simeq \varinjlim_i B_i$ with affine $B_i$, then $C(L) \simeq \varinjlim_i C(L)\times_B B_i \simeq \varinjlim_i C(L|_{B_i})$,
and hence $\Gamma_B^\gr C(L) \simeq \varprojlim_i (b_i)_* \Gamma^\gr_{B_i} C(L|_{B_i})$ for $b_i:B_i \to B$.
On the other hand, $\Sym(L) \simeq \varprojlim_i (b_i)_*\Sym(L)|_{B_i}\simeq \varprojlim_i (b_i)_*\Sym(L|_{B_i})$ since $\QCoh^\gr_B \simeq \varprojlim_i \QCoh^\gr_{B_i}$.

We will use an induction on $n$.
Since the statement is obvious for the connective case, we may assume that $n<0$.
Then we may write
\[L \simeq \fib (K \to P) \in \QCoh_B\]
for some $K,P\in \QCoh_B$ such that $P[-n-1]$ is free and and $K$ is $(n+1)$-connective.
Then we can form a commutative square
\begin{equation}\label{Eq:AbCone1}
\xymatrix{ 
\Sym(L) \ar[r] \ar[d] & \Gamma^\gr_{B} C(L)\ar[d] \\
\Sym (\coCech(L \to K) ) \ar[r] & \Gamma^\gr_B (\Cech(C(K) \to C(L))),
}	
\end{equation}
in $(\QCoh_B^\gr)^{\Delta}$, 
where $\Cech(-)$ (resp. $\coCech(-)$) denotes the Cech (resp. coCech) nerve, 
and the upper two objects are regarded as constant cosimplicial objects.
To show that the top horizontal arrow in \eqref{Eq:AbCone1} is an equivalence, it suffices to show that the totalizations 
of the other three arrows are equivalences.

Firstly, the lower horizontal arrow in \eqref{Eq:AbCone1} is an equivalence by induction since
$\coCech_k(L \to K)\simeq K \oplus P^{\oplus k}$
is $(n+1)$-connective and $C$ preserves limits.
Secondly, the totalization of the right vertical arrow in \eqref{Eq:AbCone1} is an equivalence since
$C(K) \to C(L)$
is an effective epimorphism (in the sense of \cite[\S6.2.3]{LurHTT}) by \cite[Lem.~2.7]{Mon}.


Finally, we consider the left vertical arrow in \eqref{Eq:AbCone1}. 
Note that $\coCech(0 \to P)$ has a cosimplicial homotopy equivalence to zero since $0 \to P$ has a retract, see \cite[Prop.~2.2]{Bhatt} and \cite[14.~28.~5]{Sta}.
Hence $Q\otimes \Sym^w(\coCech(0 \to P))$ for any $Q\in \QCoh_B$ also has a cosimplicial homotopy equivalence to zero. In particular, the totalization vanishes,
\begin{equation}\label{Eq:AbCone2}
\varprojlim_{k\in\Delta} Q\otimes \Sym^w(\coCech_k(0 \to P))\simeq 0.
\end{equation}
From the canonical cofiber sequence
\[\xymatrix{
L \ar[r]& \coCech(L \to K) \ar[r] & \coCech (0\to P )
} \textin  \QCoh_B^{\Delta},\]
we can form an induced bounded filtration on $\Sym^w(\coCech(L \to K))$ such that
\[\Gr^i\Sym^w(\coCech(L \to K)) \simeq \Sym^i(L)\otimes \Sym^{w-i}(\coCech (0\to P )).\]
By \eqref{Eq:AbCone2}, we have the desired equivalence.
\end{proof}

There is also a classical analog of Proposition \ref{Prop:DRviaDefSp}.

\begin{remark}[Classical deformation spaces]\label{Rem:ClDefSp}
Let $g:M \to B$ be a morphism of derived Artin stacks.
The classical deformation space $M^\circ_{M/B}:=M^\circ_{M_\cl/B_\cl}$ in \cite{Ful,AP}
 is a flat deformation of the base $B_\cl$ to the intrinsic normal cone $\fC_{M/B}:=\fC_{M_\cl/B_\cl}$. 

The induced complete filtered algebra $\hDR^{\inf}(M/B) :=\widehat{\Gamma^\fil_B(M^\circ_{M/B}/\GG_m)}$ is equivalent to the de Rham complex with the infinitesimal Hodge filtration \cite{Har} (see also \cite[Const.~4.7]{Bhatt}).
This can be shown by the descent of $\hDR^{\inf}$ as \eqref{Eq:Descent} in \S\ref{ss:Obs}.

\end{remark}

\section{Symplectic pushforwards}\label{Sec:SP}

In this section, we establish our main theorem (Theorem \ref{Thm_A:SP}): the existence of {\em pushforwards} in the locked version of symplectic categories.
Our strategy is as follows:
\begin{enumerate}
\item We first observe that the {\em presymplectic} version of the main theorem follows from general properties of categories of correspondences (\S\ref{ss:preSP}).
\item Then the main theorem can be reduced to showing that the presymplectic pushforwards factor through the {\em symplectic categories} (\S\ref{ss:SympCat}).
\item We describe the presymplectic pushforwards as the zero loci of {\em moment maps}, which implies that they preserve the geometricity of derived stacks (\S\ref{ss:Moment}).
\item We finally compare the non-degeneracy in the presymplectic adjunction using the notion of {\em Lagrangian correspondence fibrations} (\S\ref{ss:LCF}).
\end{enumerate}

Throughout this section, let $p: U \to B$ be a finitely presented morphism of derived stacks and  $w\in \sA^0(B,d+2)$ be a $(d+2)$-shifted function.

\subsection{Presymplectic pushforwards}\label{ss:preSP}

In this subsection, we provide pushforwards in the presymplectic categories consisting of derived stacks with locked $2$-forms, without the geometricity nor the non-degeneracy.

The $w$-locked {\em presymplectic category} is the category of spans \cite{Hau} 
\[\pSymp^w_{B,d}:=\Span\left(\dSt_{\sA^{2,\lc}_B[d]^w}\right),\]
where $\sA^{2,\lc}_B[d]^w:=\sA^{2,\lc}(-/B,d)^w  \in \dSt_B$ is the derived stack\footnote{The prestack $\sA^{2,\lc}(-/B,d)^w$ satisfies the {\'e}tale descent since we have \eqref{Eq:TwistedExact}. On the other hand, the prestack $\sA^{2,\lc}(-/B,d)$ is not a derived stack since constant prestacks are not derived stacks.} of $w$-locked forms (Notation \ref{Notation:w-locked}).
More explicitly, the presymplectic category $\pSymp^w_{B,d}$ consists of:
\begin{itemize}
\item [(D1)] The objects are derived stacks $M\in \dSt_B$ together with $\theta_M\in \sA^{2,\lc}(M/B,d)^w$.
\item [(D2)] The morphisms, denoted by $C:(M,\theta_M) \dashrightarrow (N,\theta_N)$, are correspondences 
\[ \xymatrix@-1pc{
& C \ar[ld]_{} \ar[rd]^{} & \\
M && N,
}\]
in $\dSt_B$, together with equivalences $\gamma_C:\theta_M|_C \xrightarrow{\simeq} \theta_N|_C$ in $\sA^{2,\lc }(C/B,d)^w$.
\item [(D3)]The composition of $C: M \dashrightarrow N$ and $D:N \dashrightarrow L$ is the fiber product
\[ {\xymatrix@-1pc{
&& E\ar[ld]\ar[rd] \ar@{}[dd]|-{\Box} \\
& C \ar[ld]_{} \ar[rd]^{} & &D \ar[ld] \ar[rd] \\
M && N && L,
}}\]
together with $\gamma_E:=\gamma_D|_E \circ \gamma_C|_E : \theta_M|_E \xrightarrow{\simeq} \theta_N|_E \xrightarrow{\simeq} \theta_L|_E$ in $\sA^{2,\lc  }(E/B,d)^w.$
\end{itemize}

\begin{proposition}[Presymplectic adjunctions]\label{Prop:pSympPush}
There is a canonical adjunction
\[\xymatrix{
p^* :\pSymp_{B,d}^w\ar@<.4ex>[r] & \ar@<.4ex>[l] \pSymp_{U,d}^{w|_U}:p_*.
}\]	
\end{proposition}

Objectwise, there is a simple description of the above presymplectic adjunction.
Consider the canonical correspondence of derived stacks
\begin{equation}\label{Eq:Corr}
\xymatrix{
&\sA^{2,\lc }_B[d]^w|_U \ar[ld]_{s:=\pr_1} \ar[rd]^{t:=(-)_{/U}} & \\
\sA^{2,\lc }_B[d]^w && \sA^{2,\lc  }_U[d]^{w|_U}
}
\end{equation}
where $\sA^{2,\lc }_B[d]^w|_U:=\sA^{2,\lc }_B[d]^w\times_B U$ and $t$ is the restriction map.\footnote{For the precise construction of the map $t: \sA^{2,\lc}_B[d]^w|_U \to \sA^{2,\lc}_U[d]^{w|_U}$, see \cite[Rem.~B.12.6]{CHS}.}
We may consider:
\begin{itemize}
\item $M \in \pSymp_{B,d}^w \mapsto p^*M:=t_!s^*M \in \pSymp_{U,d}^{w|_U}$,
\item $N \in \pSymp_{U,d}^{w|_U} \mapsto p_*N:=s_!t^*N \in \pSymp_{B,d}^w $,
\end{itemize}
where $s_!,t_!$ denote the forgetful functors and $s^*,t^*$ denote the pullback functors.
Then for any $M \in \pSymp_{B,d}^w$ and $N \in \pSymp_{U,d}^{w|_U}$, the desired equivalence of spaces
\[\map_{\pSymp_{U,d}^{w|_U}}(p^* M,N) \simeq \map_{\pSymp_{B,d}^w}(M,p_*N) \]
can be induced from the canonical equivalences of derived stacks
\[t_!s^*M \times_{\sA^{2,\lc }_U[d]^{w|_U}} N \simeq s^*M \times_{\sA^{2,\lc }_B[d]^w|_U} t^*N \simeq M\times_{\sA^{2,\lc }_B[d]^w} s_!t^*N.\]

For the precise proof of Proposition \ref{Prop:pSympPush}, we use the $(\infty,2)$-category of spans.

\begin{proof}[Proof of Proposition \ref{Prop:pSympPush}]
Denote by
$\Span_2(\dSt)$ the $(\infty,2)$-category of spans in \cite[\S5]{Hau}.
Applying the $(\infty,2)$-categorial Yoneda lemma \cite{Hin} 
to the point $\bullet:=\Spec(\C) \in \Span_2(\dSt)^\op$,  
we have an $(\infty,2)$-functor
\[h_\bullet : \Span_2(\dSt) \to \Cat_2: A\in \dSt \mapsto \Span(\dSt_{\bullet\times A})\simeq \Span(\dSt_A)\]
where $\Cat_2$ is the $(\infty,2)$-category of $(\infty,1)$-categories and the mapping categories are given by \cite[Thm.~1.2(iii)]{Hau}.\footnote{The $(\infty,2)$-categories used in \cite{Hau}---complete $2$-fold Segal spaces \cite{Bar1}---are equivalent to the $(\infty,2)$-categories used in \cite{Hin}---categories enriched over $\Cat$ \cite{GH}---by \cite{Hau1,Mac}.}
The correspondence \eqref{Eq:Corr}, considered as an $1$-morphism in $ \Span_2(\dSt)$, has an adjoint by \cite[Lem.~12.3]{Hau}.
In other words, we have an $(\infty,2)$-functor
$c:\Adj \to \Span_2(\dSt),$
where $\Adj$ denotes the $2$-category of adjunctions \cite{RV}.
Then the composition
$h_\bullet \circ c:\Adj \to \Span_2(\dSt) \to \Cat_2$
gives us the desired adjunction.\footnote{An $(\infty,2)$-functor $\Adj \to \Cat_2$ is equivalent to an adjunction in \cite[\S5.2]{LurHTT} by \cite{RV1,RV}.} 
\end{proof}

\subsection{Symplectic categories}\label{ss:SympCat}

In this subsection, we present our main theorem.

We first define the locked versions of symplectic categories.
This is completely analogous to the usual symplectic categories \cite{Cal1,Hau} with $\sA^{2,\cl}$ replaced by $\sA^{2,\lc}$.
The $w$-locked {\em symplectic category} is the subcategory
\[\Symp_{B,d}^w \subseteq \pSymp_{B,d}^w,\]
 consisting of the following data:
\begin{enumerate}
\item [(D1)] (Objects) $w$-locked {\em symplectic fibrations}, i.e. $(M,\theta_M)\in \pSymp_{B,d}^w$ such that 
\begin{enumerate}
\item (Geometricity) $M \to B$ are (geometric and) finitely presented,
\item (Non-degeneracy) $\TT_{M/B} \xrightarrow{\theta_M} \LL_{M/B}[d]$ induced by  $\theta_M$ are equivalences.
\end{enumerate}
\item [(D2)] (Morphisms) $w$-locked {\em Lagrangian correspondences}, i.e. morphisms $(C,\gamma_C):(M,\theta_M) \dashrightarrow (N,\theta_N)$ in $\pSymp_{B,d}^w$ such that
\begin{enumerate}
\item (Geometricity) $C\to B$ are (geometric and) finitely presented,
\item (Non-degeneracy) the commutative squares
\begin{equation}\label{Eq:LagCorr}
\xymatrix{
\TT_{C/B} \ar[r] \ar[d] \cart & \TT_{N/B}|_C \cong \LL_{N/B}[d]|_C \ar[d] \\
\TT_{M/B}|_C \cong \LL_{M/B}[d]|_C \ar[r] & \LL_{C/B}[d].
}
\end{equation}
induced by $\gamma_C$ are pullback squares.
\end{enumerate}
\end{enumerate}
It is straightforward to show that the compositions in $\pSymp_{B,d}^w$ descends to $\Symp_{B,d}^w$, see \cite[Thm.~4.4]{Cal1} or \cite[Prop.~14.12]{Hau}.


\begin{theorem}[Symplectic adjunctions]\label{Thm:SP}
The presymplectic adjunction $p^*\dashv p_*$ (Proposition \ref{Prop:pSympPush}) factors through an adjunction between the symplectic categories
\[\xymatrix{
p^* :\Symp_{B,d}^w\ar@<.4ex>[r] & \ar@<.4ex>[l] \Symp_{U,d}^{w|_U}:p_*.
}\]	
\end{theorem}

More explicitly, we will show the following three statements:
\begin{enumerate}
\item [(P1)] If $N \to U$ is a $w|_U$-locked symplectic fibration, then $p_*(N) \to B$ is a $w$-locked symplectic fibration.
\item [(P2)]If $C: N \dashrightarrow L$ is a $w|_U$-locked Lagrangian correspondence over $U$, then \[p_*(C) : p_*(N) \dashrightarrow p_*(L)\] is a $w$-locked Lagrangian correspondence over $B$.
\item [(P3)] If $M\to B$ is a $w$-locked symplectic fibration, $N\to U$ is a $w|_U$-locked symplectic fibration, and $D: p^*(M) \dashrightarrow N$ is a morphism in $\pSymp_{U,d}^{w|_U}$, then
\begin{align*}
\qquad \qquad & D : p^*(M) \dashrightarrow N \text{ is a $w|_U$-locked Lagrangian correspondence over $U$}\\
&\iff D : M \dashrightarrow p_*(N)\text{ is a $w$-locked Lagrangian correspondence over $B$}.
\end{align*}
\end{enumerate}

Before proving Theorem \ref{Thm:SP},
we state basic functorial properties of the symplectic pushforwards that will be useful in the subsequent sections.

\begin{lemma}[Base change/Functoriality]\label{Lem:SP_Functoriality}\
\begin{enumerate}
\item Given a fiber square of derived stacks
\[\xymatrix{
U' \ar[r]^{p'} \ar[d]^{q'} \cart & B'\ar[d]^{q}\\
U \ar[r]^p & B,
}\] 
we have
\[\qquad q^*\circ p_* \simeq (p')_*\circ (q')^* : \Symp_{U,d}^{w|_U} \to \Symp_{B',d}^{w|_{B'}}.\]
\item Given a finitely presented morphism of derived stacks $r: V \to U$, we have
\[p_* \circ r_* \simeq (p \circ r)_* : \Symp_{V,d}^{w|_V} \to \Symp_{B,d}^w.\]
\end{enumerate}
\end{lemma}

Lemma \ref{Lem:SP_Functoriality}(1) will follow from the moment map description of symplectic pushforwards (in \S\ref{ss:Moment})
and Lemma \ref{Lem:SP_Functoriality}(2) follows from the uniqueness of right adjoints.




\subsection{Universal moment maps}\label{ss:Moment}

In this subsection, we construct the {\em universal moment maps} by studying the symplectic geometry of cotangent bundles
and use them to describe the presymplectic pushforwards.
In particular, this ensures that the presymplectic pushforwards preserve the geometricity.

Denote by $\T^*_{U/B}[d]:=\tot(\LL_{U/B}[d])$ the $d$-shifted {\em cotangent bundle} of $p:U \to B$.

\begin{proposition}[Universal moment maps]\label{Prop:UnivMoment}
There exists a canonical map
\[\mu : \sA^{2,\lc }_U[d]^{w|_U} \to \T^*_{U/B}[d+1] \]
that fits into the fiber diagram of derived stacks
\[\xymatrix{
\sA^{2,\lc}_B[d]^w|_U \ar[r]^-{(-)_{/U}} \ar[d] \cart & \sA^{2,\lc }_U[d]^{w|_U} \ar@{.>}[d]^{\mu} \ar[r] \cart & U \ar[d]^-{0_w|_U}\\
U \ar[r]^-{0} & \T^*_{U/B}[d+1] \ar[r]
& \sA^{2,\ex}_B[d+1]|_U,
}\]
where 
$0_w \in \sA^{2,\ex}(B/B,d+1) \xleftarrow{\simeq} \sA^0(B,d+2)$ is the exact $2$-form induced by $w$.
\end{proposition}

Consequently, the presymplectic pushforward of $N\in \pSymp_{U,d}^{w|_U}$ is
\[p_*(N) \simeq \mu_N^{-1}(0):=N \times_{\mu_N,\T_{U/B}[d+1],0}U,\] 
the zero locus of the {\em moment map} $\mu_N : N \to \sA^{2,\lc}_U[d]^{w|_U} \xrightarrow{\mu} \T^*_{U/B}[d+1]$.
In particular, if $N \to U$ is geometric (of finite presentation), then so is $p_*(N) \to B$.
This proves the geometricity part of Theorem \ref{Thm:SP} and Lemma \ref{Lem:SP_Functoriality}(1).

Proposition \ref{Prop:UnivMoment} will follow from the {\em exact} symplectic geometry of cotangent bundles.
Without the exact structures, the symplectic geometry is well-known;
recall from \cite{PTVV,Cal} that we have the following structures.
\begin{itemize}
\item $\T^*_{U/B}[d+1] \to B$ is a $(d+1)$-shifted symplectic fibration;
\item $\T^*_{U/B}[d+1] \to U$ is a Lagrangian fibration (see Example \ref{Ex:LagFib});
\item $\alpha \in \sA^{1,\cl}(U/B,d+1)$ induces a Lagrangian $\Gamma_\alpha : U \to \T^*_{U/B}[d+1]$.
\end{itemize}
If we additionally consider the exact structures,
not only the extensions of the above three results (see \S\ref{ss:TwCot}),
but also an alternative description of the cotangent bundles exist; 
the cotangent bundles are the derived stacks of {\em exact isotropic fibrations}.

\begin{lemma}[Exact isotropic fibrations]\label{Lem:ExIsotFib}
There exists a canonical fiber square
\begin{equation}\label{Eq:104}
\xymatrix{
\T^*_{U/B}[d+1] \ar[r]
\ar[d]^{} \cart & \sA^{2,\ex}_B[d+1]|_U \ar[d]^-{(-)_{/U}}\\
U \ar[r]^-{0} & \sA_U^{2,\ex}[d+1].
}	
\end{equation}
\end{lemma}

\begin{proof}
The desired fiber square follows by combining the two fiber squares:
\begin{equation}\label{Eq:107}
\xymatrix{
\T^*_{U/B}[d+1] \ar[r]
\ar[d] \cart & \sA^1_B[d+1]|_U  \ar[d]^-{(-)_{/U}} 
\\
U \ar[r]^-{0}  & \sA^1_U[d+1]
}\qquad
\xymatrix{
\sA^1_B[d+1]|_U \ar[r]^-{} \ar[d]^-{(-)_{/U}} \cart & \sA^{2,\ex}_B[d+1]|_U \ar[d]^-{(-)_{/U}}
\\
\sA^1_U[d+1] \ar[r]^-{} & \sA^{2,\ex}_U[d+1].
}	
\end{equation}
The left fiber square follows from 
the canonical fiber sequence of cotangent complexes
\[\xymatrix{\LL_{U/B}|_M \ar[r] & \LL_{M/B} \ar[r] & \LL_{M/U}}\]
for any geometric morphism $M\to U$ (we can use the descent of $\sA^1(M/U,d+1)$ and $\sA^1(M/B,d+1)$ along $B$ to replace $B$ with a derived affine scheme). 
The right fiber square follows from the canonical fiber sequence 
\[ \xymatrix{\Gr^1\DR(M/B) \ar[r] & \Fil^0/\Fil^2\DR(M/B) \ar[r] & \Gr^0\DR(M/B)}\]
together with the equivalence $(-)_{/U}:\Gr^0\DR(M/B) \xrightarrow{\cong} p_*\Gr^0\DR(M/U)$.
\end{proof}

\begin{proof}[Proof of Proposition \ref{Prop:UnivMoment}]
Take the fibers of the fiber square \eqref{Eq:104} over the exact $2$-form $0_w|_U \in \sA^{2,\ex}(U/B,d+1)$, then we have a fiber diagram
\[\xymatrix{
\sA^{2,\lc}_B[d]^w|_U \ar[r]^-{(-)_{/U}} \ar[d] \cart & \sA^{2,\lc}_U[d]^{w|_U} \ar[d]^{\mu} \ar[r] \cart & U \ar[d]^-{0_w|_U}\\
U \ar[r]^-{0} & \T^*_{U/B}[d+1] \ar[r]
\ar[d] \cart & \sA^{2,\ex}_B[d+1]|_U \ar[d]^-{(-)_{/U}}\\
& U \ar[r]^-{0} & \sA_U^{2,\ex}[d+1],
}\]
since locked forms are equivalent to twisted exact forms; see \eqref{Eq:TwistedExact}.
\end{proof}

\begin{remark}[Closed version]\label{Rem:ClosedSympPush}
Given a derived stack $N$ over $U$ and a symplectic $\theta_N\in \sA^{2,\cl}(N/U,d)$,
we can still define {\em a moment map} as a Lagrangian
\[\mu: N \to \T^*_{U/B}[d+1]\]
together with an equivalence between $\theta_N$ and the induced symplectic form on  $N \to \T^*_{U/B}[d+1] \xrightarrow{\pi} U$ via the Lagrangian fibration structure on $\pi$.\footnote{The composition of a Lagrangian and a Lagrangian fibration is a symplectic fibration, see \cite[Prop.~1.10]{Saf} or Lemma \ref{Lem:nd}.}
Then we can also define the {\em symplectic pushforward with respect to $\mu$} as the Lagrangian intersection
\[p_*^\mu(N) := \mu^{-1}(0):=N\times_{\mu,\T^*_{U/B}[d+1],0}U.\]
However the existence of a moment map is not guaranteed since $ \T^*_{U/B}[d+1]$ is {\em not} the stack of isotropic fibrations;
the canonical commutative square
\[\xymatrix{
\T^*_{U/B}[d+1] \ar[r]
\ar[d]^{}  & \sA^{2,\cl}_B[d+1]|_U \ar[d]^-{(-)_{/U}}\\
U \ar[r]^-{0} & \sA_U^{2,\cl}[d+1]
}	\]
is not a fiber square.
Put differently, we have a canonical map
\[\T^*_{U/B}[d+1]^{\isot}:=\fib (\T^*_{U/B}[d+1] \xrightarrow{d_{\DR}\lambda} \sA^{2,\cl}_B[d+1]|_U) \to \sA^{2,\cl}_U[d]\]
that is not an equivalence.
Finding a moment map is equivalent to find a lift of $N$ under the above map.
Moreover even when the moment map exists, it is not unique and the (closed) symplectic pushforward {\em depends} on the choice of a moment map.
See Remark \ref{Rem:ClosedTwCot} for an example.
\end{remark}

\subsection{Lagrangian correspondence fibrations}\label{ss:LCF}
In this subsection, we introduce {\em Lagrangian correspondence fibrations} as generalizations of Lagrangian fibrations.
This notion is designed to compare the non-degeneracy of isotropic morphisms to symplectic fibrations over {\em different} bases.
We provide three canonical Lagrangian correspondence fibrations which immediately show that the presymplectic adjunctions induce the desired symplectic adjunctions.
Since we are focusing on the non-degeneracy here, we will work with {\em closed} forms, instead of locked forms.

\begin{definition}[Lagrangian correspondence fibrations]\label{Def:LagCorrFib}
A commutative diagram
\begin{equation}\label{Eq:LCF}
\xymatrix{
M \ar[r]^{r} \ar[d]^g & N \ar[d]^h \\
B & U , \ar[l]_{p}
}	
\end{equation}
of derived stacks and finitely presented morphisms is called a $d$-shifted {\em Lagrangian correspondence fibration} if it is equipped with the following data:
\begin{itemize}
\item [(D1)] $d$-shifted symplectic structures $\theta_{g} \in \sA^{2,\cl}(M/B,d)$, $\theta_h \in \sA^{2,\cl}(N/U,d)$, and
\item [(D2)] an equivalence $(\theta_{g})_{/U} \cong (\theta_h)|_M$ in $\sA^{2,\cl }(M/U,d)$ that induces a fiber square
\[\xymatrix{
\TT_{M/U} \ar[r] \ar[d] \cart &\TT_{N/U}|_M \cong \LL_{N/U}[d]|_M \ar[d] \\
\TT_{M/B} \cong \LL_{M/B}[d] \ar[r] & \LL_{M/U}[d].
}\]
\end{itemize}
In short, we say that $M \xrightarrow{r} N$ is a Lagrangian correspondence fibration over $U \xrightarrow{p} B$.
\end{definition}

We note that
a commutative diagram \eqref{Eq:LCF} 
for symplectic fibrations $g:M \to B$ and $h:N \to U$
is a Lagrangian correspondence fibration if and only if 
\[\xymatrix{
 & M \ar[ld]_{(\id,h\circ r)} \ar[rd]^{r} & \\
 M\times_B U && N
}\]	
is a Lagrangian correspondence (over $U$).
In particular, if $B=\Spec(\C)$, then all fibers of $N \to U$ have Lagrangian correspondences to $M$.	

We present two basic examples.
The simplest example is a Lagrangian fibration.

\begin{example}[Lagrangian fibration]\label{Ex:LagFib}
A {\em Lagrangian fibration} $r:E \to U$ for a symplectic fibration $E \to B$ is a Lagrangian correspondence fibration of the form
\[\xymatrix{
E \ar[r]^{r} \ar[d] & U \ar@{=}[d] \\
B & U. \ar[l]_{}
}	\]
In short, we say that $E \xrightarrow{r} U$ is a Lagrangian fibration over $B$.
\end{example}

Another example is a Lagrangian intersection on a Lagrangian fibration. 

\begin{lemma}[Lagrangian intersection]\label{Lem:LagIntLCF}
Let $r:E \to U$ be a $d$-shifted Lagrangian fibration over $B$.
Given Lagrangians $n:N \to E$ and $l:L \to E$,
then
\[N\times_{E} L \to N\times_U L\]
is a $(d-1)$-shifted Lagrangian correspondence fibration (over $U \to B$).

\end{lemma}


\begin{proof} 
Note that we have three Lagrangians (over $U$)
\[\xymatrix@C+1pc{
&  L_2:=L\times_B U \ar[d]^{l\times \id}\\
L_1:=N\times_B U \ar[r]^-{n\times \id} & M:=E \times_B U & E=:L_3 . \ar[l]_-{(\id,r)}
}\]
By the triple Lagrangian intersection theorem \cite[Thm.~3.1]{Ben}, the canonical map
\[L_1\times_M L_2\times_M L_3 \to (L_1\times_M L_2) \times_U (L_2\times_M L_3)\times_U (L_3\times_M L_1)\]
is Lagrangian.
This induces a canonical Lagrangian correspondence
\[N\times_E L:(N\times_E L)\times_B U  \dashrightarrow  N \times_U L,\]
equivalent to the desired Lagrangian correspondence fibration.
\end{proof}

The key property 
is the equivalence of the non-degeneracy of isotropic morphisms. 

\begin{lemma}[Equivalence of non-degeneracy]\label{Lem:nd}
Let $r:M \to N$ 
be a Lagrangian correspondence fibration over $p:U \to B$.
For an isotropic morphism $c:C \to M$,
\[c:C \to M\text{ is a Lagrangian (over $B$)} \iff r\circ c:C \to N \text{ is a Lagrangian (over $U$)} .\]	
\end{lemma}

\begin{proof}
Form a commutative diagram 
\[\xymatrix{
\TT_{C/U} \ar[r] \ar[d] \cart & \TT_{M/U}|_C \ar[r] \ar[d] \cart &\TT_{N/U}|_C \cong \LL_{N/U}[d]|_C \ar[d] \\
\TT_{C/B} \ar[r] \ar[d] \ar@{}[rd]|{\gamma} & \TT_{M/B}|_C \cong \LL_{M/B}[d]|_C \ar[r] \ar[d] \cart & \LL_{M/U}[d]|_C \ar[d]\\
0 \ar[r]  & \LL_{C/B}[d] \ar[r] & \LL_{C/U}[d],
}\]
where $\gamma$ is the isotropic structure of $C \to M$.
Then the total  square is cartesian if and only if the left lower square is cartesian, since the three other squares are cartesian.
Equivalently, $C \to M$ is Lagrangian if and only if $C \to N$ is Lagrangian.
\end{proof}

We are now ready to prove our main theorem.

\begin{proof}[Proof of Theorem  \ref{Thm:SP}]
We will prove the three statements (P1), (P2), (P3) in \S\ref{ss:SympCat}.
Since we already have all the structures (by Proposition \ref{Prop:pSympPush}) and the geometricity (by Proposition \ref{Prop:UnivMoment}), it suffices to prove the non-degeneracy of the given maps.

(P1) The shifted cotangent bundle \[E:=\T^*_{U/B}[d+1] \to U\] is a Lagrangian fibration (over $B$) by \cite[Thm.~2.4(3)]{Cal} and hence is a Lagrangian correspondence fibration (over $U \to B$) as explained in Example \ref{Ex:LagFib}.
Therefore, the moment map $\mu_N : N \to E$ is Lagrangian by the equivalence of non-degeneracy (Lemma \ref{Lem:nd}) 
and hence the presymplectic pushforward $p_*(N):=\mu^{-1}_N(0)$ is symplectic by the Lagrangian intersection theorem \cite[Thm.~2.9]{PTVV}.

(P2) Since the moment maps $\mu_N :N \to E$, $\mu_L :L \to E$ are Lagrangian by (P1), 
\[N\times_E L:=N \times_{\mu_N,E,\mu_L} L \to N \times_U L\]
is a Lagrangian correspondence fibration by Lemma \ref{Lem:LagIntLCF}.
Hence the canonical map
\[C \to N \times_{E} L\]
is Lagrangian by Lemma \ref{Lem:nd}.
The pushforward $p_*(C)$ is a Lagrangian correspondence since it is the composition of Lagrangian correspondence \cite[Thm.~4.4]{Cal1}
\[\xymatrix@C+3pc{
p_*(C): p_*(N)\times_B p_*(L) \ar@{-->}[r]^-{p_*(N\times_E L)} 
&  N \times_{E}L \ar@{-->}[r]^-{C} & B,
}\]
where the first one is Lagrangian by the triple Lagrangian intersection theorem \cite[Thm.~3.1]{Ben} for the three Lagrangians $\mu_N:N\to E$,  $\mu_L:L \to E$, $0_E:U \to E$.\footnote{Alternatively, $p_*(C)$  is the ``horizontal composition'' of the $2$-fold Lagrangian correspondence $C: N \leadsto L : B \dashrightarrow E $ with $\id_{U} : U \leadsto U : E \dashrightarrow B$, given by the zero section $0_E:U \to E$, in the $(\infty,2)$-category of symplectic fibrations \cite{CHS}. Here the curly arrows $\leadsto$ are $2$-morphisms.}

(P3) We have a canonical Lagrangian correspondence fibration 
\[M\times_Bp_*N \simeq p^*M \times_{\mu_{p^*M},E,\mu_N} N \to p^*M \times_U N\]
by Lemma \ref{Lem:LagIntLCF}, where the equivalence is given by the fiber diagram
\[\xymatrix{
 \ar@{.>}[r] \ar@{.>}[d] \cart & p_*N \ar[r] \ar[d] \cart & N \ar[d]^{\mu_N} \\
p^*M \ar[r] \ar[d] \cart & U \ar[r]^{0} \ar[d]^{p} & E \\
M \ar[r] & B
}\]
since the (underlying morphism) of the moment map $\mu_{p^*M} : p^*M \to E$ factors through the zero section $0_E:U \to E$.
Moreover, this equivalence preserves the symplectic forms (over $B$).
Lemma \ref{Lem:nd} then completes the proof.
\end{proof}

\begin{remark}[Locked version: symplectic pushforward squares]\label{Rem:bexLagCorFib}
There is a straightforward generalization of Lagrangian correspondence fibrations (Definition \ref{Def:LagCorrFib}) to locked forms.
However it is not necessary to introduce this additional notion since we have better notions---the symplectic pushforwards---in the locked setting.
Indeed, we have canonical equivalences between the following structures:
\begin{itemize}
\item $w$-locked Lagrangian correspondence fibrations $M \to N$ over $U \to B$;
\item $w|_U$-locked Lagrangian correspondences $M : p^*M \dashrightarrow N$ over $U$;
\item $w$-locked Lagrangian correspondences $M: M \dashrightarrow p_*N$  over $B$;
\item $w$-locked {\'e}tale symplecto-morphisms $M \to p_*N$ over $B$.
\end{itemize}

We will say that a commutative square
\[\xymatrix{
M \ar[r]^{r} \ar[d]^g & N \ar[d]^h \\
B & U , \ar[l]_{p}
}\]
is a {\em symplectic pushforward square} if it is the $w$-locked version of the Lagrangian correspondence fibration such that the induced map $M \to p_*N$ is an equivalence.
\end{remark}

\section{Three basic examples}\label{Sec:Examples}

In this section, we describe three basic examples of symplectic pushforwards:
 twisted cotangent bundles (\S\ref{ss:TwCot}), symplectic zero loci (\S\ref{ss:SZ}), and symplectic quotients (\S\ref{ss:SQ}).
We will see in the subsequent section (\S\ref{Sec:LST}) that these basic examples are the local models for general locked symplectic fibrations.

\subsection{Twisted cotangents}\label{ss:TwCot}

In this subsection, we study the symplectic geometry of {\em twisted cotangent bundles}.
Realizing the twisted cotangent bundles as symplectic pushforwards, various properties and structures follow naturally.

Throughout this subsection, we fix a finitely presented morphism of derived stacks $p:U \to B$ and a $(d+2)$-shifted function $w: B \to \AA^1[d+2]$.

We first consider the most trivial case---the {\em identity map} $\id_U : U \to U$.
There is a unique symplectic form on the identity map (since $\sA^{2,\cl}(U/U,d)$ is contractible), 
but it has several $w$-locking structures.\footnote{Actually, there is also a unique locked symplectic form (without fixing the underlying functions).}
Giving a $w$-locked symplectic form is equivalent to giving a {\em $w$-locked $1$-form};
we have canonical equivalences
\begin{equation}\label{Eq:LSFidentity}
\sA^{1,\lc}(U/B,d+1)^w\xrightarrow[(-)_{/U}]{\simeq}\sA^{1,\lc}(U/U,d+1)^{w|_U} \xleftarrow[\Fil^2 \to \Fil^1]{\simeq} \sA^{2,\lc}(U/U,d)^{w|_U}.
\end{equation}
Given a locked $1$-form $\alpha\in \sA^{1,\lc}(U/B,d+1)^w$, denote by $0_\alpha \in \sA^{2,\lc}(U/U,d)^{w|_U}$ the image of the above equivalence and
\[U_\alpha:=\left(\id_U:U \to U, 0_\alpha \in \sA^{2,\lc}(U/U,d)\right) \in \Symp_{U,d}^{w|_U}\]
the induced locked symplectic fibration.

We observe that the symplectic pushforwards of the above almost trivial objects are important non-trivial objects---the {\em twisted cotangent bundles}.
\begin{definition}[Twisted cotangent bundles]\label{Def:TwCot}
For a $(d+1)$-shifted $w$-locked $1$-form $\alpha \in \sA^{1,\lc}(U/B,d+1)^w$, the $d$-shifted {\em $\alpha$-twisted cotangent bundle} is:
\[\T^*_{U/B,\alpha}[d]:=(p:U\to B)_*(U_\alpha) \in \Symp_{B,d}^{w}.\]
\end{definition}

The moment map description (Proposition \ref{Prop:UnivMoment}) shows that Definition \ref{Def:TwCot} is compatible with the usual definition.
Indeed, we have a Lagrangian intersection
\begin{equation}\label{Eq:1}
\xymatrix{
\T^*_{U/B,\alpha}[d] \ar[r] \ar[d] \cart & U \ar[d]^{\mu_{\alpha}} \\
U \ar[r]^-{0} & \T^*_{U/B}[d+1],
}
\end{equation}
where $\mu_\alpha$ is the moment map associated to $U_\alpha$.
It is straightforward to check that the underlying morphism of $\mu_\alpha$ corresponds to the underlying $1$-form of $\alpha$
and the Lagrangian structure of $\mu_\alpha$ corresponds to the closing structure of $\alpha$. 

In the {\em exact} case (i.e. $w\simeq 0$),  we obtain the {\em critical loci}.

\begin{example} [Critical loci]
The {\em critical locus} of a function  $v \in \sA^{0}(U,d+1)$ is:
\[\Crit_{U/B}(v):=\T^*_{U/B,d_{\DR}(v)}[d] \in \Symp_{B,d}^{0}.\]
\end{example}

From the universal property of symplectic pushforwards, we can observe that locked symplectic fibrations are equivalent to (exact) Lagrangians on critical loci.

\begin{corollary}[Lagrangian factorizations]\label{Cor:LagFact}
For any $v\in \sA^0(U,d+2)$, 
we have
\[
\Lag_{\Crit_{U/B}(v)/B,d+1}^0 \simeq \left(\Symp_{U,d}^v\right)_0
\]
where $\Lag_{\Crit_{U/B}(v)/B,d+1}^0 :=\Map_{\Symp_{B,d+1}^0}(B,\Crit_{U/B}(v))$ is the space of exact Lagrangians and $(-)_0:\Cat \to \Grpd$ is the functor of underlying spaces of objects.
\end{corollary}

\begin{proof}
Let $H:=\Crit_{U/B}(v)$.
By the adjunction $p^*\dashv p_*$, we have 
\[\Lag_{H/B,d+1}^0:=\Map_{\Symp_{B,d+1}^0}(B,H) \simeq \Map_{\Symp_{U,d+1}^0}(U,U_{d_{\DR}(v)})\simeq \left(\Symp_{U,d}^v\right)_0,\]
where the last equivalence comes from $\sA_U^{2,\lc}[d]^v \simeq \Path_{0,0_v}\sA^{2,\ex}_U[d+1]$ in \eqref{Eq:TwistedExact}.
\end{proof}

Corollary \ref{Cor:LagFact} can be rephrased as follows: for any locked symplectic fibration $h: N \to U$, 
there exists a canonical factorization
\[\xymatrix{
& \Crit_{U/B}(v) \ar[d] \\
N \ar[r]_-{h} \ar@{.>}[ru]^{\mu_N} & U,
}\]
by an exact Lagrangian $\mu_N$ and the Lagrangian fibration $\Crit_{U/B}(v) \to U$,
where $v:U \to \AA^1[d+2]$ is the underlying function of $h$.
Not only the underlying morphism of $h: N \to U$, but also its locked symplectic form can be recovered from $\mu_N$.
Moreover, such exact Lagrangian $\mu_N$ is uniquely determined by the above property.

\begin{remark}[Moment maps are Lagrangian factors]
The moments map (Proposition \ref{Prop:UnivMoment}) is a special case of the Lagrangian factor (Corollary \ref{Cor:LagFact}) when the function $v:U \to \AA^1[d+2]$ is the pullback of a function on the base $B \to \AA^1[d+2]$.	
\end{remark}

We collect various Lagrangian structures that follows immediately from the definition of twisted cotangents and the functoriality of symplectic pushforwards.

\begin{remark}[Functoriality]\label{Rem:TwCotFunct}
Let $\alpha\in \sA^{1,\lc}(U/B,d+1)^w$ be a $w$-locked $1$-form.
\begin{enumerate}
\item [(a)] (Lagrangian fibration, \cite[Thm.~3.5]{Gra}) The projection map \[\T^*_{U/B,\alpha}[d] \to U\] is a Lagrangian fibration (Example \ref{Ex:LagFib}) by Remark \ref{Rem:bexLagCorFib}.
	Indeed, the projection $\T^*_{U/B,\alpha}[d] \to U_\alpha$ is a Lagrangian correspondence fibration (Definition \ref{Def:LagCorrFib}) and $U_\alpha \simeq U$ without the locking structure.
\item [(b)] (Twisted cotangent correspondences I)
For a morphism $U \to V$ of finitely presented derived stacks over $B$, we have a canonical Lagrangian correspondence
\begin{equation}\label{Eq:TwCotCorr1}	
\T^*_{U/B,\alpha}[d] : \T^*_{U/B,\alpha}[d]\times_B V  \dashrightarrow \T^*_{U/V,\alpha_{/V}}[d].
\end{equation}
Indeed, we have $\T^*_{U/B,\alpha}[d] \simeq (V \xrightarrow{q} B)_*(\T^*_{U/V,\alpha_{/V}})$ by the functoriality of symplectic pushforwards (Lemma \ref{Lem:SP_Functoriality}) and thus \eqref{Eq:TwCotCorr1} follows from the unit map of the  adjunction $q^*\dashv q_*$.
\item [(c)] (Twisted cotangent correspondences II)
 For a morphism $s:W \to U$ of finitely presented derived stacks over $B$,
we have a canonical Lagrangian correspondence
\begin{equation}\label{Eq:TwCotCorr2}
 \T^*_{U/B,\alpha}[d]\times_U W :
\T^*_{U/B,\alpha}[d] \dashrightarrow \T^*_{W/B,\alpha|_W}[d].
\end{equation}
Indeed, we have a Lagrangian correspondence $W:U_\alpha \dashrightarrow \T^*_{W/U,\alpha|_{W/U}}[d]$ given by the unit map of $s^*\dashv s_*$.
Apply $p_*$, then we get \eqref{Eq:TwCotCorr2}.
\item [(d)] (Twisted conormal Lagrangian, cf.~\cite[Thm.~2.13]{Cal}) 
In the situation of Remark \ref{Rem:TwCotFunct}(b), we have a canonical Lagrangian
\[\mu : \T^*_{U/V,\alpha_{/V}}[d] \to \T^*_{V/B}[d+1]\]
together with a Lagrangian intersection diagram
\[\xymatrix{
\T^*_{U/B,\alpha}[d] \ar[r] \ar[d] \cart & \T^*_{U/V,\alpha_{/V}}[d] \ar[d]^{\mu} \\
V \ar[r]^-{0} & \T^*_{V/B}[d+1].
}\]
This follows from the moment map description of
$\T^*_{U/B,\alpha}[d] \simeq q_*(\T^*_{U/V,\alpha_{/V}}[d])$.
\end{enumerate}
\end{remark}

\begin{remark}[Closed version]\label{Rem:ClosedTwCot}
In the literatures, the twisted cotangent bundles are defined for closed forms, instead of locked forms, via the fiber square \eqref{Eq:1}.
They can be explained via the closed version of symplectic pushforwards (in Remark \ref{Rem:ClosedSympPush}).
Indeed, there is a unique symplectic form $0 \in \sA^{2,\cl}(U/U,d)$ on $\id_U:U \to U$.
The moment map (in the sense of Remark \ref{Rem:ClosedSympPush}) exists, but {\em not} unique.
Giving a moment map is equivalent to give a closed $1$-form; 
we have 
\[ \sA^{1,\cl }(U/B,d+1)\xrightarrow{\simeq} \left\{\text{moment maps } \mu : U \to \T^*_{U/B}[d+1]\right\}.\]
Given $\alpha \in \sA^{1,\cl }(U/B,d+1)$, 
the symplectic pushforward (in the sense of Remark \ref{Rem:ClosedSympPush}) with respect to the induced moment map $\mu_\alpha$ is:
\[(U\to B)_*^{\mu_\alpha}(U):= \mu_\alpha^{-1}(0) := U\times_{\mu_\alpha,\T^*_{U/B}[d+1],0}U,\]
the usual twisted cotangent bundle given by the Lagrangian intersection \eqref{Eq:1}.
\end{remark}

\subsection{Symplectic zero loci}\label{ss:SZ}
In this subsection, we introduce {\em symplectic zero loci} of sections of symmetric complexes.
They are local models for even-shifted symplectic fibrations (in \S\ref{Sec:LST}).

We first observe that symmetric forms on perfect complexes give rise to locked forms on the zero sections.
To be precise, let us fix some notations:
\begin{itemize}
\item Given a perfect complex $E$ on a derived stack $U$, we consider the zero section 
\[0_E: U \to \E:=\Tot(E).\]
\item The space of $d$-shifted {\em symmetric $p$-forms} on $\E$ is: 
\[\scS^p(\E,d):=\Map_{\QCoh_U}(\O_U,\Sym^p(E\dual)[d]).\]
\item The symmetric $p$-forms on $\E$ are equivalent to weight $p$ equivariant functions on $\E$ (by Lemma \ref{Lem:FtnsLinear}) and thus we have the forgetful map
\begin{equation}\label{Eq:SymForms1}
\scS^p(\E,d)\to \sA^{0}(\E,d).
\end{equation}
\item The symmetric $p$-forms on $\E$ are also equivalent to the $p$-forms on the zero section $0_E$ since $\LL_{U/\E} \simeq E\dual[1]$;
we have a canonical equivalence
\begin{equation}\label{Eq:SymForms2}
	\scS^p(\E,p+d)\xrightarrow{\simeq} \sA^{p}(U/\E,d).
\end{equation}
\end{itemize}

\begin{proposition}[Locked forms on zero sections]\label{Prop:LFZero}
Let $E$ be a perfect complex on a derived stack $U$.
Then there exists a canonical map of spaces
\[\scZ:\scS^p(\E,p+d) \to \sA^{p,\lc}(U/\E,d) \]
satisfying the following properties:
\begin{enumerate}
\item $(-)^p \circ \scZ :\scS^p(\E,p+d) \xrightarrow{} \sA^p(U/\E,d)$ is equivalent to \eqref{Eq:SymForms2}.
\item $[-]\circ \scZ :\scS^p(\E,p+d) \to \sA^0(\E,p+d)$ is equivalent to \eqref{Eq:SymForms1}.
\end{enumerate}
\end{proposition}

Recall Proposition \ref{Prop:DRviaDefSp} that the locked forms on the zero section $0_\E$ can be viewed as functions on the associated deformation space $\D_{U/\E}$.
Thus Proposition \ref{Prop:LFZero} can be shown by computing the deformation space $\D_{U/\E}$.

\begin{lemma}[Deformation spaces of zero sections]\label{Lem:D_U/E}
Let $E$ be a perfect complex on a derived stack $U$.
Then there exists a canonical equivalence of derived stacks
\[\D_{U/\E} \simeq \E \times \AA^1.\]
Moreover, the $\GG_m$-action on the deformation space $\D_{U/E}$ is equivalent to the diagonal $\GG_m$-action on $\E\times\AA^1$ with weights $(-1)$ on $\E$ and $1$ on $\AA^1$.
\end{lemma}

\begin{proof} 
Note that we have a canonical fiber square
\begin{equation}\label{Eq:D_U/E}
\xymatrix{
\D_{U/\E} \simeq \uMap_{\E \times \AA^1}(\E \times 0,U\times \AA^1) \ar[r] \ar[d] \cart & \uMap_{U \times \AA^1}(U \times 0,U\times \AA^1) \ar[d]^{0_E}\simeq  \D_{U/U}
\\
\E \times \AA^1 \ar[r]^-{\mathrm{constant}} & \uMap_{U \times \AA^1}(U \times 0,\E\times \AA^1) \simeq \D_{\E/U}.
}\end{equation}
Moreover, we also have canonical equivalences of derived stacks
\[\uMap_{U \times \AA^1}(U \times 0,U\times \AA^1) \simeq  U \times \AA^1, \quad \uMap_{U \times \AA^1}(U \times 0,\E\times \AA^1) \simeq \Tot(i_*i^*(E\boxtimes \O_{\AA^1}))\]
for $i:=(\id,0): U \to U\times \AA^1$.
Hence, the canonical cofiber sequence
\[\xymatrix{
E\boxtimes \O_{\AA^1} \ar[r]^-{T} & E\boxtimes \O_{\AA^1} \ar[r] & i_*i^*(E\boxtimes \O_{\AA^1})},\]
for the coordinate function $T\in \Gamma(\AA^1,\O_{\AA^1})$, gives us the desired equivalence.
\end{proof}


\begin{proof}[Proof of Proposition \ref{Prop:LFZero}]
We define the desired map $\scZ$  as the pullback of weight $p$ functions along the projection map
\[\D_{U/\E} \simeq \E \times \AA^1 \xrightarrow{\pr_1} \E.\]
Here equivariant functions on $\D_{U/\E}$ are identified to locked forms on $0_E$ via Proposition \ref{Prop:DRviaDefSp}
and equivariant functions on $\E$ are identified to symmetric forms on $E$ via Lemma \ref{Lem:FtnsLinear}.
\end{proof}


We are now ready to define the {\em symplectic zero loci} for sections of symmetric complexes.
We use the following notations:
\begin{itemize}
\item A $d$-shifted {\em symmetric complex} $E$ on $U$ is a perfect complex equipped with a symmetric $2$-form $\beta_E \in\scS^2(\E,d)$ which induces an equivalence $E\xrightarrow{\simeq}E\dual[d]$.
\item 
Given a $(d+2)$-shifted symmetric complex $E$, the {\em symplectic zero section} is:
\[0^\symp_E:=(0_E:U \to \E,\scZ(\beta_E)) \in \Symp_{\E,d}^{\q_E},\]
where $\q_E\in \sA^0(\E,d+2)$ is induced by $\beta_E\in \scS^2(E,d+2)$ under \eqref{Eq:SymForms1}.
\item Given a section $s:\O_U \to E$, the {\em zero locus} is the fiber product
\[\xymatrix{
\Zero(s) \ar[r]^{} \ar[d]^{} \cart & U \ar[d]^{s} \\
U \ar[r]^{0_E} & \E.
}\]
\end{itemize}

\begin{definition}[Symplectic zero loci]\label{Def:SympZero}
Let $p:U \to B$ be a finitely presented morphism of derived stacks and $w\in \sA^0(B,d+2)$.
Let $E$ be a $(d+2)$-shifted symmetric complex on $U$ and $s:\O_U \to E$ be a section equipped with an equivalence $s^2\simeq w|_U \in \sA^0(U,d+2)$.
The {\em symplectic zero locus} is:
\[\Zero^\symp_{U/B}(E,s):=(U \xrightarrow{p}B)_*(U \xrightarrow{s}\E)^*(0_E^\symp) \in \Symp_{B,d}^w.\]
\end{definition}

We now discuss the {\em functoriality} of the symplectic zero loci.
We say that
\[\xymatrix{
& D\ar[ld]_a \ar[rd]^b &\\
E && F
}\]
is a {\em maximal isotropic correspondence} 
of $d$-shifted symmetric complexes $E$, $F$ 
if it is equipped with an equivalence $\beta_E|_D\simeq \beta_F|_D$ in $\scS^2(D,d)$ which induces a fiber square
\begin{equation}\label{Eq:MaxIsoCorr}
\xymatrix{
D \ar[r]^-b \ar[d]_-a \cart & F \simeq F\dual[d] \ar[d]^{b\dual[d]}\\
E \simeq E\dual[d] \ar[r]^-{a\dual[d]} & D\dual[d].
}
\end{equation}
In short, we denote by $D:E \dashrightarrow F$ for a maximal isotropic correspondence.

The symplectic zero loci are stable under the change of symmetric complexes via maximal isotropic correspondences.

\begin{proposition}[Change of symmetric complexes]\label{Prop:ChangeofSymmCplx}
In the situation of Definition \ref{Def:SympZero}, 
if we are additionally given a maximal isotropic correspondence $D:E \dashrightarrow F$, then there exists an induced section $s_F : \O_V \to F|_V$ on $V:=U\times_{s,\E,a}\D$ such that
\[\Zero^\symp_{U/B}(E,s) \simeq \Zero^\symp_{V/B}(F|_V,s_F) \textin \Symp_{B,d}^w.\]
\end{proposition}

As a special case, if $E$ admits a {\em maximal isotropic complex} $M$ (i.e. a maximal isotropic correspondence of the form $M : E \dashrightarrow 0$), then the symplectic zero locus is the twisted cotangent bundle (Definition \ref{Def:TwCot}); we have
\begin{equation}\label{Eq:SZbecometwcot}
\Zero^\symp_{U/B}(E,s) \simeq \T^*_{V/B,\alpha_s}[d] \textin \Symp_{B,d}^w,
\end{equation}
for a canonically induced locked $1$-form $\alpha_s \in \sA^{1,\lc}(V/B,d+1)^{w}$.


Proposition \ref{Prop:ChangeofSymmCplx} follows from the functoriality of the symplectic zero sections. 

\begin{lemma}[Functoriality of symplectic zero sections]\label{Lem:FunctZeroSection}
Let $D:E \dashrightarrow F$ be a maximal isotropic correspondence of $(d+2)$-shifted symmetric complexes on a derived stack $U$.
Then we have a canonical equivalence
\[0_\E^\symp \simeq (\sfD \xrightarrow{a} \E)_* (\sfD \xrightarrow{b} \F)^*(0_\F^\symp) \textin \Symp_{\E,d}^{q_E},\]
where $\E:=\Tot(E)$, $\sfD:=\Tot(D)$, and $\F:=\Tot(F)$ are the associated total spaces.	
\end{lemma}

\begin{proof}
Note that the commutative diagram
\[\xymatrix{
U \ar[r]^{0_\F^\symp} \ar[rd]|{0_\D} \ar[d]_{0_\E^\symp} & \F \\
\E & \sfD \ar[u]_b \ar[l]_a
}\]
is a {\em dual} version of Lagrangian correspondences in the sense that
\[\xymatrix{
\TT_{U/\sfD} \ar[r] \ar[d] \cart & \TT_{U/\F} \simeq \LL_{U/\F}[d] \ar[d] \\
\TT_{U/\E} \simeq \LL_{U/\E}[d] \ar[r] & \LL_{U/\sfD}[d]
}\]
is cartesian.
Equivalently, we have a $\q_D$-locked Lagrangian correspondence (over $\sfD$)
\begin{equation}\label{Eq:LagCorrSympZeroSections}
\xymatrix{
&U \ar[ld] \ar[rd]& \\
a^* 0_\E^\symp && b^*0_\F^\symp.
}
\end{equation}
As observed in Remark \ref{Rem:bexLagCorFib}, this is also equivalent to an {\'e}tale symplecto-morphism
\[0_\E^\symp \to a_*b^* 0_\F^\symp \textin \Symp_{\E,d}^{q_E}.\]
By the moment map description of symplectic pushforwards (Proposition \ref{Prop:UnivMoment}), we can observe that the underlying derived stack of $a_*b^* 0_\F^\symp$ is a total space of a perfect complex over $U$.
Hence the above {\'e}tale map is an equivalence as desired.
\end{proof}

\begin{proof}[Proof of Proposition \ref{Prop:ChangeofSymmCplx}]
Form a fiber diagram
\[\xymatrix{
V \ar[r]^t \ar[d]^{i} \cart & \sfD \ar[r]^b \ar[d]^a & \F \\ 
U \ar[r]^s \ar[d]^p & \E \\
B
}\]
and let $s_F:\O_V \to F|_V$ be the section induced by $V\xrightarrow{t}\sfD \xrightarrow{b} \F$.
Then we have 
\begin{align*}
\Zero^{\symp}_{U/B}(E,s):=p_*s^*0_\E^\symp &\simeq p_*s^*a_*b^*0_{\F}^\symp\\
&\simeq p_* i_*t^*b^*0_\F^\symp   \simeq p_*i_*(s_F)^*(0_{F|_V}^\symp)=:\Zero^\symp_{V/B}(F|_V,s_F)
\end{align*}
in $\Symp_{B,d}^w$ by Lemma \ref{Lem:FunctZeroSection}.
\end{proof}

We end this subsection with a technical lifting lemma that will be used later in \S\ref{ss:SPT}.
In the affine case, we can conversely lift locked forms on zero loci to symmetric forms on perfect complexes.

\begin{lemma}[Lifting]\label{Lem:LiftingLockedformsZeroLoci}
Let $E$ be a perfect complex on a derived stack $U$ with a section $s:\O_U \to E$.
If $U$ is cohomologically affine\footnote{We say that a derived stack $U$ is {\em cohomologically affine} if $\Gamma(U,-) : \QCoh_U \to \QCoh_{\C}$ preserves connective objects.
If $U$ is a quasi-separated derived Artin stack with affine diagonal, this condition is equivalent to $U_\cl$ being cohomologically affine in the sense of \cite[Def.~3.1]{Alp}, see \cite[Rem.~3.5]{Alp}. \label{footnote:cohaffine}}
 and $E$ is of tor-amplitude $\geq a$ (for $a \geq0$), then for any $d\leq -pa-p$, the map
\begin{equation}\label{Eq:LiftingLemma}
\mathscr{S}^{p}(\E,p+d) \xrightarrow{\scZ} \sA^{p,\lc}(U/\E,d) \xrightarrow{s^*} \sA^{p,\lc}(\Zero(s)/U,d) 
\end{equation}
is surjective (on $\pi_0$).
\end{lemma}

We will use the {\em Koszul filtration} on the zero locus.\footnote{If $U$ is smooth affine and $E$ is a vector bundle, then the Koszul filtration on $\Gr^0\DR(\Zero(s)/U) \simeq (\Zero(s) \to U)_*\O_{\Zero(s)}$ is the filtration given by the stupid truncations of the Kuszul cdga representative.}
For any perfect complex $E$ on a derived stack $U$ with a section $s:\O_U \to E$, we consider the fiber square
\begin{equation}\label{Eq:LiftingLemma3}
\xymatrix{
\Zero(Ts) \ar[r] \ar[d] \cart & U\times \AA^1 \ar[d]^{Ts} \\
U\times \AA^1 \ar[r]^-{0} & \E\times \AA^1,
}\end{equation}
where $T\in \Gamma(\AA^1,\O_{\AA^1})$ is the coordinate function.
Note that the fiber square \eqref{Eq:LiftingLemma3} is $\GG_m$-equivariant with the weight $1$ actions on $\E$ and $\AA^1$.
Hence the Hodge-filtered de Rham complex $\DR(\Zero(s)/U)$ has an additional ($\Z_{\leq0}$-indexed) 
filtration; we have
\[\Fil_\Kz\DR(\Zero(s)/U):=\DR^{\GG_m}(\Zero(Ts)/U\times\AA^1) \in \QCAlg_{U\times\AA^1/\GG_m}^{\fil} \simeq \QCAlg_U^{\fil,\fil}\]
such that the underlying/associated graded algebra is:
\begin{equation}\label{Eq:LiftingLemma6}
\Fil^{-\infty}_{\Kz}\DR(\Zero(s)/U) \simeq \DR(\Zero(s)/U),\quad \Gr_\Kz \DR(\Zero(s)/U) \simeq \DR^{\GG_m}(\E[-1]/U),
\end{equation}
by the base change \cite[Prop.~3.10]{BFN} (when $E\dual$ is connective).

\begin{proof}[Proof of Lemma \ref{Lem:LiftingLockedformsZeroLoci}]	
We will use the following properties of the Koszul filtration: 
\begin{align}
\Fil_\Kz^w\Fil_\Hd^p \DR(\Zero(s)/U) &\simeq 0 \quad\text{for $p+w>0$}, \label{Eq:LL2}\\ 
\Gr^w_\Kz\Gr^p_\Hd \DR(\Zero(s)/U) &\simeq \Sym^w(E\dual)\otimes \Lambda^{-w-p}(E\dual)[-w-p], \label{Eq:LL3}
\end{align}
where $\Fil_{\Hd}$ is the Hodge filtration on $\DR$ (defined in \S\ref{ss:DR}).
Indeed, by the definition of deformation spaces in \eqref{Eq:DefSp},
we have a canonical equivalence
\[\D_{\E[-1]/U}/\GG_m \simeq \Tot(E[-1]^\triv) \textin \dSt_{U\times\AA^1/\GG_m} ,\]
where $(-)^\triv:=0_* :\QCoh_{U}^\gr \to \QCAlg_{U}^\fil$ is the pushforward along the zero section $0:U\times B\GG_m \xrightarrow{} U\times \AA^1/\GG_m$.
By the right equivalence in \eqref{Eq:LiftingLemma6}, Proposition \ref{Prop:DRviaDefSp}, and Lemma \ref{Lem:FtnsLinear}, 
we have a canonical equivalence
\begin{equation}\label{Eq:2}
\Gr_\Kz^w \DR(\Zero(s)/U) \simeq \Sym^{-w}(E\dual(-1)^\triv) \textin \QCoh_U^\fil.
\end{equation}
\begin{itemize}
\item The equivalence \eqref{Eq:LL3} follows from the associated graded parts of \eqref{Eq:2}.
\item The equivalence \eqref{Eq:LL2} follows from the filtration-amplitude of \eqref{Eq:2}---more precisely, $\Fil^p\Sym^{-w}(E\dual(-1)^\triv)\simeq 0$ for $p>-w$. 
\end{itemize}

Consider the commutative square
\begin{equation}\label{Eq:LiftingLemma2}
\xymatrix{
 \Fil^{-p}_{\Kz}\Fil^p_{\Hd}\DR_{\pi}(U/\E) \ar[r] \ar[d]^-{} \ar@{-->}[rd] & \Fil^p_{\Hd}\DR_{\pi}(U/\E) \ar[d]^-{}\\
 \Fil^{-p}_{\Kz}\Fil^p_{\Hd}\DR(\Zero(s)/U) \ar[r]^-{} & \Fil^p_{\Hd}\DR(\Zero(s)/U),
}\end{equation}
induced by \eqref{Eq:LiftingLemma3}, where $\Fil_\Kz\DR_\pi(U/\E):=\Gamma_U^{\fil,\fil}\DR^{\GG_m}(U\times\AA^1/\E\times\AA^1) \in \QCAlg_U^{\fil,\fil}$.
\begin{itemize}
\item We first observe that the upper horizontal arrow in \eqref{Eq:LiftingLemma2} induces the map $\scZ$ in \eqref{Eq:LiftingLemma}.
More specifically, the two filtrations on $\DR_\pi(U/\E)$ split 
(i.e. are pullbacks of graded complexes along $U\times \AA^1/\GG_m \to U\times B\GG_m$)
so that
\begin{equation}\label{Eq:LL1}
\Fil^w_\Kz\Fil^p_\Hd\DR_\pi(U/E)\simeq \bigoplus_{p\leq i \leq -w}\Sym^i(E\dual),
\end{equation}
and the map $\scZ$ is $\Map_{\QCoh_U}(\O_U,(-)[p+d])$ of the upper horizontal arrow.
\item We then claim that the left vertical arrow in \eqref{Eq:LiftingLemma2} is an equivalence.
Indeed, by \eqref{Eq:LL1} and \eqref{Eq:LL2}, the left vertical arrow is equivalent to
\[ \Gr^{-p}_{\Kz}\Gr^p_{\Hd}\left(\DR_{\pi}(U/\E) \to \DR(\Zero(s)/U)\right).\]
This is an equivalence since it is the weight $(-p)$ piece of the map
\begin{align*}
&\Gr^p_{\Hd}\left(\DR^{\GG_m}_{\pi}(U/\E) \to \DR^{\GG_m}(\E[-1]/U)\right)\\
&\simeq \Sym^p(E\dual(1))\otimes \Big(\O_U \to \Sym(E\dual(1)[1])\Big).
\end{align*}
\item We next show that the lower horizontal arrow in \eqref{Eq:LiftingLemma2} is $(pa)$-connective.
Indeed, by the left equivalence in \eqref{Eq:LiftingLemma6}, it suffices to show that
\[\Gr^{w}_{\Kz}\Fil^p_{\Hd}\DR(\Zero(s)/U) \text{ is $(pa+1)$-connective for $w<-p$}.\]
By \eqref{Eq:LL2}, it suffices to show that
\[\qquad \Gr^{w}_{\Kz}\Gr^q_{\Hd}\DR(\Zero(s)/U) \text{ is $(pa+1)$-connective for $w<-p$, $q\geq p$}.\]
This follows from \eqref{Eq:LL3} since $-qa+w+q\leq -pa-1$.
\end{itemize}
Combining the above results, the map \eqref{Eq:LiftingLemma} is $\Map_{\QCoh_U}(\O_U,(-)[p+d])$ of the diagonal arrow in \eqref{Eq:LiftingLemma2}, which is $(pa)$-connective.
This completes the proof.
\end{proof}

\subsection{Symplectic quotients}\label{ss:SQ}
In this subsection, we construct {\em symplectic quotients} via symplectic pushforwards.
The locked versions of symplectic actions are already {\em Hamiltonian} 
so that the symplectic quotients can be constructed without any additional structures.
Various basic properties of symplectic quotients follow immediately from the functoriality of symplectic pushforwards.

Throughout this subsection, fix the base derived stack $B$ and $w \in \sA^0(B,d+2)$.

We first fix the notion of (locked) {\em symplectic actions}.
We say that $G$ is a {\em group stack} if it is a group object in the overcategory $\dSt_B$ (in the sense of \cite[Def.~7.2.2.1]{LurHTT}).
Recall that giving a {\em group action} of $G$ on a derived stack $M\in \dSt_B$ is equivalent to give a fiber square of derived stacks
\[\xymatrix{
M \ar[r] \ar[d] \cart & B \ar[d]^{\sigma}\\
M/G \ar[r] & BG,
}\]
where $BG$ is the classifying stack. 
Given a $w$-locked symplectic fibration $M \to B$,
we say that a $G$-action on $M$ is a {\em $w$-locked symplectic action} if it is equipped with:
\[M/G\in \Symp_{BG,d}^{w|_{BG}},\and (BG \xrightarrow{\sigma} B)^*(M/G) \simeq M \in \Symp_{B,d}^w.\]


\begin{definition}[Symplectic quotients]\label{Def:SQ}
Let $M\to B$ be a $w$-locked symplectic fibration together with a $w$-locked symplectic action of a smooth 
group stack $G\to B$.
We define the {\em symplectic quotient} as:
\[M/\!/G:=(BG \xrightarrow{\pi} B)_*(M/G) \in \Symp_{B,d}^w.\]
\end{definition}

The moment map description of symplectic pushforwards (Proposition \ref{Prop:UnivMoment}) says that the symplectic quotient $M/\!/G$ is the Lagrangian intersection
\[\xymatrix{
M/\!/G \simeq 
\mu^{-1}(0)/G 
\ar[r] \ar[d] \ar@{}[rd]|-{\Box}  & M/G \ar[d]^{\mu/G} \\
BG \ar[r]^-{0}& \T^*_{BG/B}[d+1] 
\simeq
\mathfrak{g}\dual[d]/G,
}\]
where $\mathfrak{g}:=\T_{G/B}|_{\id}$ is the tangent complex of $G$ at the identity section,
and the moment map is identified to a $G$-equivariant map $\mu:M \to \mathfrak{g}\dual[d]$.
Moreover, 
we have a canonical Lagrangian correspondence
\[\xymatrix{
&\mu^{-1}(0)\ar[ld]\ar[rd]&\\
M/\!/G\simeq \mu^{-1}(0)/G && M,
}\]
given by the pullback $\sigma^*$ of the counit map $\pi^*\pi_*(M/G) \dashrightarrow M/G$. 

The symplectic quotients are stable under change of groups.

\begin{proposition}[Change of groups]\label{Prop:Changeofgroups}
Let $M$ be a $w$-locked symplectic fibration with a $w$-locked symplectic action of a smooth group stack $G$.
Let $G \to H$ be a group homomorphism of smooth group stacks over $B$.
Then there exists an induced $w$-locked symplectic fibration $M_{G/H}$ with a $w$-locked symplectic action of $H$ such that
\[M/\!/G \simeq (M_{G/H})/\!/H \textin \Symp_{B,d}^w.\]
Moreover, if $G \to H$ is a smooth, then
$M_{G/H} \simeq M/\!/K$
for $K:=\ker(G\to H)$.
\end{proposition}


\begin{proof} 
We define
$M_{G/H}:=(B\to BH)^*(BG \to BH)_*(M/G).$
Then 
we have
\[M_{G/H}/\!/H:=(BH \to B)_*(BG \to BH)_*(M/G) \simeq (BG \to B)_*(M/G)=:M/\!/G\]
by the functoriality of symplectic pushforwards (Lemma \ref{Lem:SP_Functoriality}).

If $G \to H$ is smooth, then $K$
is also smooth,
and the canonical fiber square
\[\xymatrix{
BK \ar[r] \ar[d] \cart & BG \ar[d]^{} \\
B \ar[r] & BH,
}\]
of the classifying stacks gives us the desired equivalence
\begin{align*}
M_{G/H}:=(B\to BH)^*(BG \to BH)_*(M/G) &\simeq (BK \to B)_*(BK \to BG)^*(M/G)\\
	&\simeq (BK \to B)_*(M/K) =:M/\!/K,
\end{align*}
since symplectic pushforwards commute with pullbacks (Lemma \ref{Lem:SP_Functoriality}).
\end{proof}

The symplectic quotients are compatible with twisted cotangent bundles/symplectic zero loci in the previous sections since they are all symplectic pushforwards.
We use the following notations when $U$ is equipped with an action of a group stack $G$:
\begin{itemize}
\item A locked $1$-form $\alpha \in \sA^{p,\lc}(U/B,d)$ is a {\em $G$-invariant} if it is equipped with $(\alpha/G) \in \sA^{p,\lc}((U/G)/BG,d)$ and an equivalence $(\alpha/G)|_{U/B} \simeq \alpha$.
\item A symmetric complex $E$ on $U$ is {\em $G$-equivariant} if it is equipped with a symmetric complex $E/G$ on $U/G$ and an equivalence $(E/G)|_{U} \simeq E$.
\item A section $s:\O_U \to E$ of a $G$-equivariant complex $E$ is {\em $G$-invariant} if it is equipped with a section $s/G : \O_{U/G} \to E/G$ and an equivalence $(s/G)|_U\simeq s$.
\item A {\em $G$-equivalence} $h: f \simeq g$ of $G$-invariant functions $f, g : U \to \AA^1[d+2]$ is an equivalence $h/G:f/G\simeq g/G : U/G \to \AA^1[d+2]$.
\end{itemize}

\begin{proposition}[Compatibility]\label{Prop:Quot_of_TwCot}
Let $U$ be a finitely presented derived stack over $B$ together with an action of a smooth group stack $G$.

\begin{enumerate}
\item Let $\alpha$ be a $G$-invariant $w$-locked $(d+1)$-shifted $1$-form on $U$.
Then 
$\T^*_{U/B,\alpha}[d]$ has a canonical $w$-locked symplectic $G$-action such that
\[\T^*_{U/B,\alpha}[d]/\!/G \simeq \T^*_{(U/G)/B,(\alpha/G)}[d] \in \Symp_{B,d}^w.\]
\item Let $E$ be a $G$-equivariant $(d+2)$-shifted symmetric complex on $U$ and $s$ be a $G$-invariant section with a $G$-invariant equivalence $s^2 \simeq w|_{U}$.
Then 
$\Zero^{\symp}_{U/B}(E,s)$ 
has a canonical $w$-locked symplectic $G$-action such that
\[\Zero^{\symp}_{U/B}(E,s)/\!/G \simeq \Zero^{\symp}_{(U/G)/B}((E/G),(s/G)) \in \Symp_{B,d}^w.\]
\end{enumerate}
\end{proposition}

\begin{proof}
Both (1) and (2) follow immediately from Lemma \ref{Lem:SP_Functoriality}. 
\end{proof}

In particular, for any $G$-invariant function $v: U \to \AA^1[d+1]$, we have:
\[\Crit_{U/B}(v)/\!/G \simeq \Crit_{(U/G)/B}(v/G) \in \Symp_{B,d}^0,\]
by Proposition \ref{Prop:Quot_of_TwCot}(1) for $w\simeq 0$.

\section{Local structure theorems}\label{Sec:LST}
In this section, we prove our main application (Theorem \ref{Thm_B:LST}): the {\em local structure theorems} for symplectic fibrations. 
We provide two versions:
\begin{enumerate}
\item {\'e}tale local structure theorem for $1$-stacks 
(\S\ref{ss:SPT});
\item smooth local structure theorem for higher stacks (\S\ref{ss:SSC}).
\end{enumerate}

Throughout this section, all derived Artin stacks are assumed to be {\em of finite type}, that is, the classical truncations are of finite type over $\C$.
In \S\ref{ss:SPT}, all derived $1$-Artin stacks are assumed to be {\em quasi-separated}, that is, the diagonals are quasi-compact.
The {\em stabilizers} of derived $1$-Artin stacks are the stabilizers of the classical truncations.

\subsection{Symplectic pushforward towers}\label{ss:SPT}

In this subsection, we provide an inductive local description of symplectic fibrations via towers of symplectic pushforwards.
The {\'e}tale local structure theorem follows by analyzing the last terms in the towers.



\begin{theorem}[Symplectic pushforward towers]\label{Thm:SPT}
Let $g:M \to B$ be a $d$-shifted $w$-locked symplectic fibration for $w\in \sA^0(B,d+2)$.
Assume that $M$ is a derived $1$-Artin stack with affine stabilizers, $B$ is a derived algebraic space, 
and $d\leq0$.
Given a point $m\in M(\C)$ with linearly reductive stabilizer,
there exist a sequence 
\begin{equation}\label{Eq:SPT}
\xymatrix@C-.7pc{
M_{(-d+1)} \ar[r] \ar[d]^{}& M_{(-d)} \ar[r]\ar[d]^{} & M_{(-d-1)} \ar[r]\ar[d]^{} & \cdots \ar[r] & M_{\left(\lceil \frac {-d-1}{2} \rceil +1\right)} \ar[r] \ar[d]^{} & M_{\left(\lceil \frac {-d-1}{2} \rceil \right)} \ar[d]^{}  \\
M_{(-2)} & M_{(-1)} \ar[l] &  M_{(0)} \ar[l] & \cdots \ar[l] & M_{\left(\lfloor \frac {-d-1}{2} \rfloor -1 \right)} \ar[l] & M_{\left(\lfloor \frac {-d-1}{2} \rfloor \right) }\ar[l] 
}
\end{equation}
of pointed derived $1$-Artin stacks $M_{(\bullet)}$ of finite presentation over $B$ 
such that 
\begin{enumerate}
\item [(C1)] $\TT_{M_{(k)}/M_{(k-1)}}$ are of tor-amplitude $[k,k]$, and $M_{(-2)}:=B$,
\item [(C2)] the vertical arrows $M_{(k)} \to M_{(-d-1-k)}$
are locked symplectic fibrations whose underlying functions are $w|_{M_{(-d-1-k)}}$, 
\item [(C3)] the square are symplectic pushforward squares (Remark \ref{Rem:bexLagCorFib}), that is,
\[M_{(k)} \cong \left(M_{(-d-k)} \to M_{(-d-1-k)}\right)_*\left(M_{(k-1)}\right),\]
\end{enumerate}
and 
a pointed {\'e}tale morphism 
\begin{equation}\label{Eq:SPT2}
M_{(-d+1)} \to (M,m)
\end{equation}
that preserves the $w$-locked symplectic forms and the stabilizers at the base points.
\end{theorem}

Consequently, symplectic fibrations are {\'e}tale-locally the symplectic pushforwards of symplectic fibrations that appear in the last vertical arrows $M_{\left(\lceil \frac {-d-1}{2} \rceil \right)} \to M_{\left(\lfloor \frac {-d-1}{2} \rfloor \right) }$ 
in \eqref{Eq:SPT}.
We will observe that those last terms are locally the zero loci of sections of orthogonal/symplectic bundles or are the identity maps, depending on the parity of $d$.
This will imply that symplectic fibrations are locally the symplectic zero loci (in \S\ref{ss:SZ}) or the twisted cotangent bundles (in \S\ref{ss:TwCot}).

\begin{corollary}[{\'E}tale local structure]\label{Cor:ELS}
Let $g:M \to B$ be a $d$-shifted $w$-locked symplectic fibration for $w\in \sA^0(B,d+2)$.
Assume that $M$ is a derived $1$-Artin stack with affine stabilizers, $B$ is a  derived algebraic space, and $d<0$. 
Let $m\in M(\C)$ be a point with linearly reductive stabilizer.
\begin{enumerate}
\item If $d\equiv 2 \in \Z/4$ (resp. $d\equiv 0 \in \Z/4$), then there exist
\begin{itemize}
\item a derived $1$-Artin stack $U$ of finite presentation over $B$ with affine stabilizers such that $\LL_{U/B}$ is of tor-amplitude $\geq \frac d2+1$,
\item a point $u\in U(\C)$ whose stabilizer group is linearly reductive,
\item an orthogonal (resp. symplectic) bundle $E$ over $U$, 
\item a section $s:\O_U \to E[\tfrac d2 +1]$ with  $s^2 \simeq w|_{U}$ and $s(u)\simeq 0$,
\item and 
a pointed {\'e}tale morphism
\begin{equation}\label{Eq:ELSeven}
\left(\SZ_{U/B}\left(E\left[\tfrac d2+1\right],s\right),u\right) \to \left(M,m\right)
\end{equation}
that preserves the $w$-locked symplectic forms and the stabilizers at $u$.
\end{itemize}
\item If $d$ is odd, then there exist
\begin{itemize}
\item a derived $1$-Artin stack $V$ of finite presentation over $B$ with affine stabilizers such that $\LL_{V/B}$ is of tor-amplitude $\geq \frac d2$,
\item a point $v\in V(\C)$ whose stabilizer group is linearly reductive,
\item a $w$-locked $1$-form $\alpha \in \sA^{1,\lc}(V/B,d+1)^w$ with $\alpha^1(v)\simeq 0$,\footnote{Here $\alpha^1(v) : \C \to \LL_{V/B}|_v[d+1]$ is the pullback of $\alpha^1: \O_V \to \LL_{V/B}[d+1]$ by $v: \Spec(\C) \to V$.}
\item 
and a pointed {\'e}tale morphism 
\begin{equation}\label{Eq:ELSodd}
\left(\T^*_{V/B,\alpha}[d],v\right) \to \left(M,m\right)
\end{equation}
that preserves the $w$-locked symplectic forms and the stabilizers at $v$.
\end{itemize}
\end{enumerate}
\end{corollary}

Corollary \ref{Cor:ELS} is equivalent to Theorem \ref{Thm_B:LST} stated in the introduction.

\begin{remark}[Symplectic quotient presentations]
Note that derived Artin stacks are {\'e}tale locally the quotient stacks of derived affine schemes by linearly reductive groups near points with linearly reductive stabilizers by \cite[Thm.~1.1]{AHR} (and \cite[Thm.~1.13]{AHHR}).
In the situation of Corollary \ref{Cor:ELS}, if we write $U \simeq U'/G$ and $V \simeq V'/G$, then the {\'e}tale symplectic charts in \eqref{Eq:ELSeven} and \eqref{Eq:ELSodd} can be presented as:
\[ \SZ_{U/B}\left(E,s\right) \simeq \SZ_{U'/B}\left(E',s'\right)/\!/G , \quad
\T^*_{V/B,\alpha}[d] \simeq \T^*_{V'/B,\alpha'}[d]/\!/G,\]
for induced $G$-equivariant bundle $E'$, section $s'$, $1$-form $\alpha'$,
since symplectic quotients are compatible with symplectic zero loci and twisted cotangents (Proposition \ref{Prop:Quot_of_TwCot}). 
\end{remark}

We can apply Corollary \ref{Cor:ELS} to stacks with good moduli spaces.

\begin{remark}[Good moduli]\label{Rem:goodmoduli}
In the situation of Corollary \ref{Cor:ELS}.
if $M_\cl$ has affine diagonal and good moduli space \cite[Def.~4.1]{Alp}, 
we can cover $M$ by the {\'e}tale charts \eqref{Eq:ELSeven} or \eqref{Eq:ELSodd}.
This follows from \cite[Prop.~12.14]{Alp} and \cite[Prop.~3.2, Prop.~4.13]{AHR}.
\end{remark}


Since the symplectic pushforward towers (Theorem \ref{Thm:SPT}) still exist for $d=0$, we also have an {\'e}tale local structure theorem for $0$-shifted symplectic stacks.

\begin{remark}[$0$-shifted case]\label{Rem:0-symp}
In the situation of Theorem \ref{Thm:SPT}, if $d=0$,
then there exist a derived affine scheme $U$ smooth over $B$ with an action of a linearly reductive group $G$, a $G$-invariant $w$-locked symplectic form, a $G$-fixed point $u$ that lies in the zero locus of the moment map, and a pointed {\'e}tale symplecto-morphism
\[(U/\!/G,u) \to (M,m)\]
that preserves the stabilizers at $u$.
This is a locked version of an analogous result for $0$-shifted non-degenerate $2$-forms (without closing structures) in \cite[Thm.~4.2.3]{Hal}.

Unlike the negatively-shifted cases, many $0$-shifted symplectic forms are not exact (nor locked) even in the absolute case.
Nevertheless, since the restrictions of the $0$-symplectic forms to the residual gerbes of closed points are always exact,
if we can lift the results in this paper to formal derived stacks with perfect cotangent complexes, we will obtain a formal local structure theorem for $0$-shifted symplectic Artin stacks with good moduli.
The author plan to investigate this in a future work.
\end{remark}


For schemes, the local structure theorem can be strengthen as follows:

\begin{remark}[Schemes]\label{Rem:ZLS}
In the situation of Corollary \ref{Cor:ELS}, assume that $M$, $B$ are derived schemes. 
Then $U$, $V$  can be arranged to be derived schemes and \eqref{Eq:ELSeven}, \eqref{Eq:ELSodd} can be arranged to be Zariski open embeddings.
Moreover, if $d\equiv 2\in \Z/4$ (resp. $d\equiv 0\in \Z/4$),
the orthogonal (resp. symplectic) bundle $E$ in Corollary \ref{Cor:ELS}(1) has a maximal isotropic (resp. Lagrangian) subbundle $F\subseteq E$, {\'e}tale (resp. Zariski) locally.
Hence the functoriality of symplectic zero loci (Proposition \ref{Prop:ChangeofSymmCplx}) gives us one additional term in the symplectic pushforward tower \eqref{Eq:SPT},
\[\xymatrix{
M_{\left(\lceil \frac {-d-1}{2} \rceil \right)} \ar@{}[r]|-{\simeq} \ar[d] &  \SZ_{U/U}(E[\frac d2 +1],s) \ar[r] \ar[d]^{}  & \Zero_{V/V}^\symp(F^\perp/F[\frac d2+1],s_2) \ar[d]\\
 M_{\left(\lfloor \frac {-d-1}{2} \rfloor \right) }  \ar@{}[r]|-{\simeq} & U & V:=\Zero(s_1), \ar[l]
}\]
where $s_1 :\O_U \to E\to F\dual$ and $s_2 : \O_{V} \to F^{\perp} \to F^\perp/F$ are the induced sections.
Hence if $d\equiv 2\in \Z/4$ and $\rank(\TT_{M/B})$ is even (resp. $d\equiv 0\in \Z/4$), then the result in Corollary \ref{Cor:ELS}(2) holds by Proposition \ref{Prop:ChangeofSymmCplx} (see the formula \eqref{Eq:SZbecometwcot}).
\end{remark}

Corollary \ref{Cor:ELS} recovers the Lagrangian neighborhood theorem \cite{JS} since we have an equivalence between Lagrangians and symplectic fibrations (Corollary \ref{Cor:LagFact}).

\begin{remark}[Lagrangian neighborhoods]\label{Rem:LNT}
Let $L \to M$ be a $d$-shifted Lagrangian of derived schemes (over $\Spec(\C)$).
Assume that $d<0$.
Then, {\'e}tale-locally, we have
\[(L \to M)\simeq \begin{cases}
\SZ_{U/B}(E[\tfrac {d+1}{2}],s) \to \Crit_B(w)	& \text{if $d$ is odd}\\
\Crit_{V/B}(v)  \to \Crit_B(w) & \text{if $d$ is even},
\end{cases}\]
where $U,V,B$ are derived affine schemes such that $\LL_B$ is of tor-amplitude $\geq \frac d2$,
$\LL_{U/B}$ is of tor-amplitude $\geq \frac {d+1}{2}$,
$\LL_{V/B}$ is of tor-amplitude $\geq \frac {d-1}{2}$,
$w: B \to \AA^1[d+1]$, $v:V \to \AA^1[d]$ are shifted functions,
$E$ is an orthogonal (resp. symplectic) bundle over $U$ if $d\equiv 3 \in \Z/4$ (resp. $d\equiv 1\in \Z/4$),
$s :\O_U \to E[\tfrac {d+1}{2}]$ is an isotropic section,
and the Lagrangians are given as the Lagrangian factors (in Corollary \ref{Cor:LagFact}).

Firstly, we may write $M\simeq \Crit_B(w)$.
Indeed, the $d$-symplectic form on $M$ is exact since $d<0$ (by \cite[Prop.~5.6(a)]{BBJ}).
Hence we can apply Corollary \ref{Cor:ELS} to $M$.
When $d$ is odd, $M$ is locally a critical locus;
when $d$ is even, $\dim(\TT_{M})$ is even since $M$ has a Lagrangian, 
and hence $M$ is locally a critical locus (by Remark \ref{Rem:ZLS}).

Secondly, the Lagrangian $L$ can be arrange to be exact. 
Indeed, the obstruction to exactness of the Lagrangian is measured by an element $o \in \sA^{\DR}(L,d+1)$.
If $d\leq -2$, then $\sA^{\DR}(L,d+1)$ is contractible and the obstruction vanishes.
If $d=-1$, then the obstruction $o \in \sA^{\DR}(L,d+1)\simeq \C$ is a constant function (if we assume that $L$ is connected) and we can replace $\Crit_B(w)$ by $\Crit_B(w-o)$.
Then $L \to \Crit_B(w) \to B$ is a $w$-locked symplectic fibration and we can apply Corollary \ref{Cor:ELS}.
\end{remark}


We now prove Theorem \ref{Thm:SPT}.
The proof can be divided into three parts.

\subsubsection*{Part I: Zero locus towers}

Firstly, we form a sequence of derived stacks $M_{(\bullet)}$ via the zero loci of sections of shifted vector bundles, which in particular satisfies the condition (C1).
Based on the {\'e}tale local structure theorem for $1$-Artin stacks \cite{AHR,AHHR}, this is an equivariant version of the local structure theorem of derived affine schemes of finite presentation in \cite[Thm.~7.4.3.18]{LurHA} and \cite[Thm.~4.1]{BBJ}.

\begin{lemma}[Zero locus towers]\label{Lem:ZLT}
Let $M$ be a finitely presented derived $1$-Artin stack with affine stabilizers over a derived algebraic space $B$.
Given a point $m\in M(\C)$ with linearly reductive stabilizer $G$,\footnote{Unlike in Theorem \ref{Thm_B:LST}, here $G$ is used as a group over $\Spec(\C)$, not a group scheme over $B$.}
there exist a sequence 
\[ M_{(-d+1)} \hookrightarrow M_{(-d)} \hookrightarrow  \cdots \hookrightarrow M_{(1)} \hookrightarrow M_{(0)} \xrightarrow{} M_{(-1)} \to M_{(-2)},\]
of pointed derived Artin stacks such that
\begin{itemize}
\item $M_{(-2)}:=B$, 
 $M_{(-1)}:=BG\times B$,
\item $M_{(0)} \to  M_{(-1)}$ is a smooth affine morphism of dimension $h^0(\TT_{M/B}|_{m})$,
\item 
$M_{(k+1)} :=\Zero(s_{(k)})$ for each $k\geq 0$,  where $E_{(k)}$ is a vector bundle on $M_{(k)}$ of rank $ h^{k+1}(\TT_{M/B}|_{m})$ and $s_{(k)} :\O_{M_{(k)}} \to E_{(k)}[-k]$ is a section, 
\end{itemize}
and a pointed {\'e}tale  morphism 
$M_{(-d+1)}\to (M,m)$
that preserves the stabilizers at the base points.

\end{lemma}

Since the proof of Lemma \ref{Lem:ZLT} is completely analogous to \cite[Thm.~4.1]{BBJ} using the reductivity of $G$, 
we postpone the proof to the end of this subsection.

\subsubsection*{Part II: Lifting locked $2$-forms}

Secondly, we lift locked $2$-forms on $M_{(k)} \to M_{(-d-1-k)}$ to $M_{(k-1)} \to M_{(-d-k)}$, inductively.
In particular, we will have presymplectic versions of (C2) and (C3).
This follows from the following general lifting criteria. 

\smallskip

Denote by $\I_{M/B}^p:=\Fil^p\DR(M/B) \in \QCoh_B$ for any $g:M\to B$ in $\dSt$.

\begin{lemma}[Lifting locked $2$-forms]\label{Lem:LiftingForms}
Consider a commutative triangle
\[\xymatrix{
M \ar[r]^{i} \ar[rd] & N \ar[d]^h \\
& U
}\]
of derived stacks where $i$, $h$ are affine and $U$ is cohomologically affine (footnote \ref{footnote:cohaffine}).
Assume that $\I_{N/U}$ is of amplitude $\leq a$ and $\I_{M/N}$ is of amplitude $\leq b$.
Then we have
\[\max(a+b,2b) \leq d+1 \implies i^*:\sA^{2,\lc }(N/U,d) \xrightarrow{}  \sA^{2,\lc}(M/U,d)\text{ is surjective}.
\]
\end{lemma}
\begin{proof} 
We first observe that there exists a canonical equivalence
\begin{equation}\label{Eq:LiftForms}
\cof(\I_{N/U}^2 \to \I_{M/U}^2) \simeq h_*\fib (\I_{M/N}^2 \to \LL_{N/U} \otimes_N \I_{M/N}) \textin \QCoh_U.
\end{equation}
This can be shown by the two canonical cofiber sequences $\I_{N/U} \to \I_{M/U} \to h_*\I_{M/N}$ and $i^*\LL_{N/U} \to \LL_{M/U} \to \LL_{M/N}$ since $\LL_{-/-}\simeq \cof(\I^2_{-/-} \to \I_{-/-})[1]$.

To show the desired surjectivity, it suffices to show that  \eqref{Eq:LiftForms} is of amplitude $\leq d+1$ since $U$ is cohomologically affine.
\begin{itemize}
\item Since $\I_{M/N}$ is of amplitude $\leq b$ and $i:M \to N$ is affine, $\I_{M/N}^2$ is of amplitude $\leq 2b\leq d+1$, by \cite[Cor.~7.4.3.6]{LurHA}.
\item Since $\I_{N/U}$ is of amplitude $\leq a$ and $h:N \to U$ is affine, $\LL_{N/U}$ is of amplitude $\leq a-1$, 
and thus $ \LL_{N/U} \otimes_N \I_{M/N}$ is of amplitude $\leq a+b-1 \leq d$.
\end{itemize}
Since $h$ is cohomologically affine, \eqref{Eq:LiftForms} is of amplitude $\leq d+1$.
\end{proof}

\subsubsection*{Part III: Lifting non-degneracy}

Finally, we obtain the non-degeneracy of the lifted $2$-forms from the {\em minimal dimension condition} of $M_{(\bullet)}$ at $m$, that is,
\begin{equation}\label{footnote:minimaldimension}
\dim H^k (\TT_{M/B}|_m) =
\begin{cases}
\dim(G)	 & \text{if } k =-1,\\
\dim(M_{0}/M_{(-1)})	 & \text{if } k =0, \\
\rank(E_{(k-1)})	 & \text{if } k \geq 1 .
\end{cases}
\end{equation}
Equivalently, the maps $\TT_{M_{(a)}/M_{(c)}}|_m \to \TT_{M_{(b)}/M_{(c)}}|_m$ have sections for all $a\geq b \geq c$.
There is also a general result on an equivalence of non-degeneracy along ``isotropic correspondence fibrations'' (cf. Definition \ref{Def:LagCorrFib}).

\begin{lemma}[Lifting non-degeneracy]\label{Lem:LiftingNondeg}
Consider a commutative diagram
\[\xymatrix{
M \ar[r]^r \ar[d]^g & N \ar[d]^h \\
B & U \ar[l]_{p}
}\]
of derived Artin stacks with finitely presented morphisms.
Let $m\in M(\C)$ with
\begin{enumerate}
\item [(A1)] $\TT_{M/B}|_m \to \TT_{N/B}|_{m}$ and $\TT_{N/B}|_{m} \to \TT_{U/B}|_m$ admit sections;
\item [(A2)] $\TT_{U/B}|_m$ and $(\TT_{U/B}|_m \oplus \TT_{N/U}|_m)\dual[d]$ have no common amplitudes.
\end{enumerate}
Given $\theta_M \in \sA^{2,\cl }(M/B,d)$ and $\theta_N \in \sA^{2,\cl }(N/U,d)$ with $\eta:(\theta_M)_{/U} \cong \theta_N|_M$, we have:
\begin{align*}
 &g: M \to B \text{ is symplectic (in an open neighborhood of $m\in M$) } 
\\ 
&\iff  
\begin{cases}
h:N \to U \text{ is symplectic (in an open neighborhood of $r(m)\in N$) and}
\\
r:M \to N \text{ is a Lagrangian correspondence fibration over $p:U \to B$}	\\
\text{(in open neighborhoods of $m \in M$ and $r(m) \in N$).}
\end{cases}
\end{align*}
\end{lemma}

\begin{proof} 
Note that the non-degeneracy of $2$-forms is an open condition since the supports of perfect complexes are closed.

Consider the commutative diagram induced by the adjoints of $2$-forms,
\[\xymatrix{
\TT_{M/U} \ar[r]^{} \ar[d] \ar@{}[rrd]|{\eta^{\#}}& \TT_{M/B} \ar[r]^-{\theta_M^{\#}} & \LL_{M/B}[d] \ar[d]\\
\TT_{N/U}|_M \ar[r]^-{\theta_N|_M^{\#}}  &\LL_{N/U}|_M[d] 
 \ar[r] & \LL_{M/U}[d].
}\]	
By the assumption (A1), after choosing the sections, we have 
decompositions
\[\TT_{M/B}|_m \simeq \TT_{M/U}|_m \oplus \TT_{U/B}|_m \simeq \TT_{N/U}|_m \oplus \TT_{M/N}|_m \oplus \TT_{U/B}|_m.\]
By the assumption (A2), the map  $\theta_M|_{m}^{\#}$
can be represented by a matrix of the form
\[\theta_M|_m^{\#} \cong \begin{pmatrix}
\theta_N|_m^{\#} & 0 & 0 \\
0 & 0 & \gamma\\
0 & \gamma\dual[d] & 0
\end{pmatrix}\]
for some map $\gamma : \TT_{U/B}|_m \to \LL_{M/N}[d]|_m$,
under the above decomposition. 
Then
\begin{align*}
\text{$\theta_M|_m^{\#}$ is an equivalence} 
&\iff \text{ $\theta_N|_m^{\#}$, $\gamma$ are equivalences}\\
&\iff \text{ $\theta_N|_m^{\#}$ is an equivalence }\& \text{ $\eta|_m^{\#}$ is cartesian}.
\end{align*}
Therefore, we have the desired equivalence of the non-degeneracy.
\end{proof}

Now Theorem \ref{Thm:SPT} follows immediately from the above lemmas (Lemma \ref{Lem:ZLT}, Lemma \ref{Lem:LiftingForms}, Lemma \ref{Lem:LiftingNondeg}).
We just have to check the amplitudes.

\begin{proof}[Proof of Theorem \ref{Thm:SPT}]

It suffices to check that the commutative diagram
\[\xymatrix{
M_{(k)} \ar[r] \ar[d] & M_{(k-1)} \ar[d] \\
M_{(-d-1-k)} & M_{(-d-k)} \ar[l]
}\]	
satisfies the assumptions in Lemma \ref{Lem:LiftingForms} and Lemma \ref{Lem:LiftingNondeg} for $k \geq \lceil \frac {-d+1}{2} \rceil+1$.

Firstly, $\I_{M_{(k+1)}/M_{(k)}} \simeq \I_{M_{(k)}/E_{(k)}[-k]}|_{M_{(k)}}$ is of amplitude $\leq -k$. Hence,
\begin{itemize}
\item $\I_{M_{(k-1)}/M_{(-d-k)}}$ is of amplitude $\leq d+k$,
\item $\I_{M_{(k)}/M_{(k-1)}}$ is of amplitude $\leq -k+1$.
\end{itemize}
Since $(d+k) + (-k+1) = d+1$ and $2(-k+1) \leq d-1$, we can apply Lemma \ref{Lem:LiftingForms}. 

Secondly, the minimal dimension condition \eqref{footnote:minimaldimension} gives us the splittings in (A1) of Lemma \ref{Lem:LiftingNondeg}. Moreover, we have
\begin{itemize}
\item $\TT_{M_{(-d-k)}/M_{(-d-1-k)}}|_m$ is of amplitude $[d+k,d+k]$,
\item $\TT_{M_{(k-1)}/M_{(-d-k)}}|_m$ is of amplitude $[1-k,d+k-1]$.
\end{itemize}
Since $-k<-k+1<k+d-1<k+d$, we have (A2) and can apply Lemma \ref{Lem:LiftingNondeg}.

As observed in Remark \ref{Rem:bexLagCorFib}, we have {\'e}tale symplecto-morphisms
\[M_{(k)} \xrightarrow{} \left(M_{(-d-k)} \to M_{(-d-1-k)}\right)_*\left(M_{(k-1)}\right).\]
Since these maps are closed embeddings by the moment map description (Proposition \ref{Prop:UnivMoment}), they are equivalences.
\end{proof}

The {\'e}tale local structure theorem follows from the lifting lemma (Lemma \ref{Lem:LiftingLockedformsZeroLoci}).

\begin{proof}[Proof of Corollary \ref{Cor:ELS}]
Form a symplectic pushforward tower $M_{(\bullet)}$ (Theorem \ref{Thm:SPT}).
As in Lemma \ref{Lem:ZLT}, we may further assume that $M_{(-1)} \to M_{(-2)}$ is the classifying stack of a linearly reductive group, $M_{(0)} \to M_{(-1)}$ is smooth affine, and the maps $M_{(k+1)} \to M_{(k)}$ for $k\geq 0$ are the zero loci of sections $s_{(k)}: M_{(k)} \to E_{(k)}[-k]$ where $E_{(k)}$ are vector bundles over $M_{(k)}$. 
In particular, all $M_{(k)}$ are cohomologically affine.

(1) If $d=2k$ for an integer $k<0$, let $U:=M_{\left(-k-1\right) }$, $E:=E_{\left(-k-1\right) }$, $s:=s_{\left(-k-1\right) }$.
The $w|_U$-locked symplectic form on the last vertical arrow $M_{\left(-k\right)} \to  U$ in \eqref{Eq:SPT} lifts to a $(d+2)$-shifted symmetric form on $E$ via the lifting lemma (Lemma \ref{Lem:LiftingLockedformsZeroLoci}) since $d \leq -2(-k-1 )-2 = d$.
After shrinking $U$, we may assume that the symmetric form is non-degenerate and thus $E$ is an orthogonal (resp. symplectic) bundle when $k$ is odd (resp. even).
Consequently, $M_{(-k)} \simeq \SZ_{U/U}(E[-k+1],s) \in \Symp_{U,d}^{w|_U}$ and
\[ M_{(-d+1)} \simeq \SZ_{U/B}(E[-k+1],s) \textin \Symp_{B,d}^w,\]
by the definition of symplectic zero loci (Definition \ref{Def:SympZero}).


(2) If $d=2k+1$ for an integer $k<0$, let $V:= M_{\left(-k\right)}$.
The $w|_V$-locked symplectic form on the last vertical arrow $\id_V :V \to V$ in \eqref{Eq:SPT} corresponds to a locked $1$-form $\alpha \in \sA^{1,\lc}(V/B,d+1)$ via the canonical equivalence \eqref{Eq:LSFidentity}. Hence we have
\[ M_{(-d+1)} \simeq \T^*_{V/B,\alpha} \textin \Symp_{B,d}^w,\]
by the definition of twisted cotangent bundles (Definition \ref{Def:TwCot}).
\end{proof}


We finally prove Lemma \ref{Lem:ZLT} to complete the proof of Theorem \ref{Thm:SPT}.

\begin{proof}[Proof of Lemma \ref{Lem:ZLT}]
We may assume that $B$ is a derived affine scheme.
By  \cite[Thm.~1.13]{AHHR}, $M$ has a {\em quotient stack presentation}, i.e.,
there is a derived affine scheme $L$ with a $G$-action, a $G$-fixed point $l \in L(\C)$, and an {\'e}tale morphism
\[N:=(L/G,l) \to (M,m),\]
that preserves the stabilizers at $l$.
We will inductively construct $M_{(\bullet)}$ that factors
\[N:=L/G \longrightarrow M_{(-1)}:= B\times BG.\]
Note that $ H^{-1}(\TT_{N/B}|_l) \simeq H^{-1}(\TT_{M_{(-1)}/B}|_{l})$ since $N \to BG$ has a section. 

%

\medskip
\noindent{\em Step 1: Closed embedding into smooth stack.}
We claim that there is a derived affine scheme $U$, smooth of dimension $h^0(\TT_{M/B}|_m)$ over $B$, with a $G$-action, and a $G$-equivariant closed embedding
$L \hookrightarrow U$ over $B$.
Then we will have a closed embedding
\[N:=L/G \hookrightarrow M_{(0)}:= U/G\quad \text{over $M_{(-1)}:= B\times BG$}\]
such that $ H^{0}(\TT_{N/B}|_l)\simeq H^{0}(\TT_{M_{(0)}/B}|_{l})$ (and $ H^{-1}(\TT_{N/B}|_l)\simeq H^{-1}(\TT_{M_{(0)}/B}|_{l})$).

We first consider the existence of $U$.
Indeed, there exists a $G$-equivariant closed embedding $L_\cl \hookrightarrow W $ into a smooth affine scheme $W$ with a $G$-action, by the lemma in \cite[p.~25]{MFK}.
This map $L_\cl \hookrightarrow W $ 
can be lifted to a $G$-equivariant closed embedding $L \hookrightarrow W$ 
by the infinitesimal lifting property of the smooth stack $W/G$ 
since $L$ is the colimit of its Postnikov truncations (in the category of derived Artin stacks) by \cite[Prop.~5.4.5]{LurDAG} 
and $L_\cl/G$ is cohomologically affine. 
Then we consider the induced closed embedding $L \hookrightarrow U:=W\times B$.

To find a minimal $U$, we will cut out $U$ by a $G$-invariant section $s$ of a $G$-equivariant vector bundle $E$ on $U$.
Indeed, consider the right exact sequence of $G$-representations
\[\xymatrix{
H^{-1}(\LL_{L/U}|_{l}) \ar[r]& H^0(\LL_{U/B}|_{l}) \ar[r] & H^0(\LL_{L/B}|_{l}) \ar[r] & 0.
}\]
Let $K:=\Ker( H^0(\LL_{U/B}|_{l}) \to H^0(\LL_{L/B}|_{l}))$.
The inclusion $K \hookrightarrow H^0(\LL_{U/B}|_{l})$ can be lifted to a $G$-equivariant map $a:K \to H^{-1}(\LL_{L/U}|_{l})$ since $G$ is reductive.
Note that
\begin{equation}\label{Eq:ZLT1}
\I_{L/U} \to \LL_{L/U}[-1] \to \LL_{L/U}|_{l}[-1] \to  H^{-1}(\LL_{L/U}|_{l}) \textin \QCoh_U^G \simeq \QCoh_{U/G}
\end{equation}
has a connective fiber since $\I_{l/L}$ and $\I_{L/U}$ are connective (and thus $\I_{L/U} \to \LL_{L/U}[-1]$ is also connective by \cite[Cor.~7.4.3.6]{LurHA}).\footnote{By abuse of notation, we are using the same letters to denote the pushforwards of quasi-coherent sheaves along the closed embeddings $l: \Spec(\C) \to L$ and $L \hookrightarrow U$.}
Hence $ K \otimes \O_U \to K \xrightarrow{a} H^{-1}(\LL_{L/U}|_{l})$ can be lifted to a map $b:  K \otimes \O_U \to \I_{L/U}$.
Let $E:=K\dual\otimes \O_U$ and $s:\O_U \to E$ be the dual of $E\dual \xrightarrow{b} \I_{L/U}\xrightarrow{} \O_U$.
Then $L \hookrightarrow U$ factors through the zero locus $\Zero(s)$ which is smooth of minimal dimension near $l$ 
since
$ H^{-1}(\LL_{\Zero(s)/B}|_{l})\simeq 0$ and $H^{0}(\LL_{\Zero(s)/B}|_{l}) \simeq  H^0(\LL_{L/B}|_{l})$. 
Then the claim follows after replacing $U$ with an open neighborhood of $l$ in $\Zero(s)$. (By \cite[Cor.~1.2]{MFK}, a $G$-invariant open neighborhood can be chosen to be affine after shrinking.)


\medskip
\noindent{\em Step 2. Classically equivalent embedding into quasi-smooth stack.}
We then claimed that there is a $G$-invariant section $s_{(0)}$ of a $G$-equivariant vector bundle $E_{(0)}\to U$ of rank $h^1(\TT_{M/B}|_m)$ and a $G$-equivariant closed embedding $L \hookrightarrow V:=\Zero(s_{(0)})$ over $U$ whose classical truncation is an isomorphism.
Then we will have a closed embedding
\[N:=L/G \hookrightarrow M_{(1)}:= V/G\quad \text{over $M_{(0)}:= U/G$}\]
such that $N_\cl \simeq (M_{(1)})_\cl$ and
 $ H^{i}(\TT_{N/B}|_l)\simeq H^{i}(\TT_{M_{(1)}/B}|_{l})$ for $i\leq1$.

The construction is similar to the last paragraph in Step 1. 
Since \eqref{Eq:ZLT1} is connective, the map $E\dual_{(0)}:=H^{-1}(\LL_{L/U}|_l)\otimes \O_U \to H^{-1}(\LL_{L/U}|_l)$ can be lifted to a map $E_{(0)}\dual \to \I_{L/U}$, which gives $s_{(0)}\dual:E_{(0)}\dual \to \I_{L/U} \to \O_U$ such that $L \hookrightarrow U$ factors through $V:=\Zero(s)$.
Since $E_{(0)}|_{U_\cl}\dual \to E_{(0)}|_{M_\cl}\dual \to \I_{M_\cl/U_\cl}/\I_{M_\cl/U_\cl}^2$ is surjective, by Nakayama lemma $E_{(0)}|_{U_\cl}\dual \to \I_{M_\cl/U_\cl}$ is also surjective after shrinking, and hence $L_\cl \simeq V_\cl$.

\medskip
\noindent{\em Step 3: Inductive construction.}
We finally construct $M_{(k+1)}$ from $M_{(k)}$, inductively, for $k\geq 1$.
More precisely, we will find a vector bundle $E_{(k)}$ of rank $h^{k+1}(\TT_{M/B}|_m)$ on $M_{(k)}$, a section $s_{(k)} :\O_{M(k)} \to E_{(k)}[-k]$, and a closed embedding
\[N \hookrightarrow M_{(k+1)}:=\Zero(s_{(k)})\quad \text{over $M_{(k)}$}\]
such that $N_\cl \simeq (M_{(k+1)})_\cl$ and $ H^{i}(\TT_{N/B}|_l)\simeq H^{i}(\TT_{M_{(k+1)}/B}|_{l})$ for $i\leq k+1$. Then 
$N \to M_{(-d+1)}$ is {\'e}tale at $l$, and hence is an equivalence after shrinking.

The construction is similar to Step 2. We first show that the composition
\[\I_{N/M_{(k)}}[-k] \to \LL_{N/M_{(k)}}[-k-1] \to \LL_{N/M_{(k)}}|_{n}[-k-1] \to  H^{-k-1}(\LL_{N/M_{(k)}}|_{n})\]
is connective, where $n:BG \hookrightarrow N:=L/G$ is the closed embedding given by $l\in L(\C)$.
Since $\LL_{N/M_{(k)}}$ is $(k+1)$-connective and $N_\cl \simeq (M_{(k)})_\cl$ by the induction hypothesis, $\I_{N/M_{(k)}}$ is $k$-connective by \cite[Cor.~7.4.3.2]{LurHA}, and 
thus $\I_{L/U} \to \LL_{L/U}[-1]$ is $2k$-connective by \cite[Cor.~7.4.3.6]{LurHA}).
Moreover, $\I_{n/N}$ is also connective and hence the claim follows.
Let $E_{(k)}:= (M_{(k)} \to BG)^* (H^{-k-1}(\LL_{N/M_{(k)}}|_{n}))\dual$.
Then the map $E_{(k)}\dual \to H^{-k-1}(\LL_{N/M_{(k)}}|_{n})$ can be lifted to a map $E_{(k)}\dual \to \I_{N/M_{(k)}}[-k]$. Hence the composition $E_{(k)}\dual \to \I_{N/M_{(k)}}[-k] \to \O_{M(k)}[-k]$ induces the desired section.
\end{proof}

\subsection{Smooth symplectic charts}\label{ss:SSC}

In this subsection, we show that symplectic fibrations of (higher) stacks admit smooth symplectic charts induced by symplectic fibrations of schemes.

We say that $C:V\dashrightarrow M$ is a ($w$-locked) {\em smooth symplectic cover} of a ($w$-locked) symplectic fibration $g:M \to B$ 
if it is a ($w$-locked) Lagrangian correspondence
\[ \xymatrix{
& C \ar@{_{(}->}[ld]_s \ar@{->>}[rd]^t & \\
V && M,
}\]
such that $t:C \to M$ is smooth surjective and $s_\cl:C_\cl\to V_\cl$ is an equivalence.

\begin{proposition}[Smooth symplectic charts]\label{Prop:SttoSch}
Let $g:M \to B$ be a $w$-locked symplectic fibration of derived Artin stacks for $w\in \sA^{0}(B,d+2)$.
Assume that $d<0$. 
Then there exist a smooth morphism $p:U \to B$, a $w|_U$-locked symplectic fibration of derived schemes $h:N \to U$, and a $w$-locked smooth symplectic cover
\[C: p_*N \dashrightarrow M \textin \Symp^{w}_{B,d}.\]
\end{proposition}

Immediately, we obtain the smooth local structure theorem for symplectic fibrations of stacks since the symplectic pushforwards are functorial.

\begin{corollary}[Smooth local structure]\label{Cor:SLS}
Let $g:M \to B$ be a $w$-locked symplectic fibration of derived Artin stacks for $w\in \sA^0(B,d+2)$. 
Assume that $d<0$.
\begin{enumerate}
\item If $d\equiv 2 \in \Z/4$ (resp. $d\equiv 0 \in \Z/4$), 
then there exist 
\begin{itemize}
\item a finitely presented morphism $p:U \to B$ from a derived scheme $U$ such that $\LL_{U/B}$ is of tor-amplitude $\geq \frac d2 +1$,
\item an orthogonal (resp. symplectic) bundle $E$ on $U$, 
\item a section $s: \O_U \to E [\tfrac d2+1]$ with $s^2 \cong w|_U$, 
\item and a $w$-locked smooth symplectic cover
\[C:\SZ_{U/B}(E,s) \dashrightarrow M.\]
\end{itemize}

\item If $d\not\equiv2 \in \Z/4$ or ($d\equiv 2 \in \Z/4$ and $\rank(\TT_{M/B})$ is even), then there exist 
\begin{itemize}
\item  a finitely presented morphism $q:V \to B$ from a derived scheme $V$ such that $\LL_{V/B}$ is of tor-amplitude $\geq \frac d2$, 
\item a locked $1$-form $\alpha \in \sA^{1,\lc }(V/B,d+1)^w$, 
\item and a $w$-locked smooth symplectic cover
\[C: \T^*_{V/B,\alpha}[d] \dashrightarrow M.\]
\end{itemize}
\end{enumerate}
\end{corollary}


We now prove Proposition \ref{Prop:SttoSch}.
We first reduce the base stack $B$ to a scheme.

\begin{lemma}[Base reduction]\label{Lem:RedBase}

Let $g:M \to B$ be a $w$-locked symplectic fibration for $w\in \sA^{0}(B,d+2)$ and $d<0$.
Let $p: U \to B$ be a smooth morphism.
Given a $w|_U$-locked smooth symplectic cover $C : N \dashrightarrow p^*M$,
the adjoint\footnote{The symplectic pushforward $p_*$ is also a left adjoint of the pullback $p^*$. 
}
\[C^\# : p_*(N) \dashrightarrow M \textin \Symp^w_{B,d}\]
is also a $w$-locked smooth symplectic cover.\end{lemma}

\begin{proof}
Note that we have a canonical commutative diagram (see \S\ref{ss:preSP})
\[\xymatrix{
C^\#\simeq C \ar[r]_-{s^\#} \ar[d]^{t} \ar@/^.4cm/[rr]^{s} \ar@/_.6cm/[dd]_{t^\#} & p_*N \ar[r] \ar[d] \cart & N \ar[d] \\
p^*M \ar[r] \ar[d] \cart & U \ar[r]^-{0} \ar[d]^{p} & \T^*_{U/B}[d+1] \\
M \ar[r] & B.
}\]	
Since $t:C \to M$ and $p:U \to B$ are smooth, so is $t^\# : C^\# \to M$.
On the other hand, the zero section $0 : U \to \T^*_{U/B}[d+1]$ is a closed embedding since $p: U \to B$ is smooth and $d<0$.
Since the classical truncations of $s : C \to N$ and the diagonal of $0 : U \to \T^*_{U/B}[d+1]$ are isomorphisms, so is $s^\#:C^\# \to p_*N$.
\end{proof}

Consequently, we may assume that the base $B$ is a derived affine scheme.
Then our strategy to reduce $M$ is as follows:
Choose a smooth surjection $L \to M$ from a derived affine scheme $L$ and form the zero locus tower $L_{(\bullet)}$ for $L$ (in Lemma \ref{Lem:ZLT}),
\[L := L_{(-e)} \hookrightarrow L_{(-e-1)}\hookrightarrow \cdots \hookrightarrow L_{(-d)} \hookrightarrow \cdots 
L_{(1)} \hookrightarrow L_{(0)} \xrightarrow{\mathrm{sm}} L_{(-1)}:=B.\]
Then we can pullback the locked symplectic form on $M$ to $L$, and lift to $L_{(-d)}$ by the lifting lemma (Lemma \ref{Lem:LiftingForms}).
Hence if we choose $L \to M$ to be of {\em minimal dimension}, then we will get the non-degeneracy of the lifted locked $2$-form on $L_{(-d)}$.

\begin{lemma}[Charts of minimal dimension]\label{Lem:MinChart}
Let $M$ be a finitely presented derived Artin stack over a derived affine scheme $B$.
For any point $m\in M(\C)$,
there exists a pointed smooth morphism
\[t:(L,l) \to (M,m)\]
such that $L$ is a derived affine scheme and $H^i(\TT_{L/B}|_{l}) \simeq H^i(\TT_{M/B}|_m)$  for all $i\geq0$.
%
\end{lemma}

\begin{proof}
Choose a pointed smooth morphism $(L,l) \to (M,m)$ from a derived affine scheme $L$ and $l\in L(\C)$.
Then we have an exact sequence of vector spaces
\[\xymatrix{
0 \ar[r] & H^0(\LL_{M/B}|_m)) \ar[r] & H^0(\LL_{L/B}|_l) \ar[r]^-{a} & H^0(\LL_{L/M}|_l) \ar[r] & H^1(\LL_{M/B}|_m) \ar[r] & 0
}\]
and equivalences $H^i(\LL_{M/B}|_m)) \simeq H^i(\LL_{L/B}|_l)$ for $i>0$.
We will cut out $L$ by a section of a vector bundle to make the map $a:H^0(\LL_{L/B}|_l)\to H^0(\LL_{L/M}|_l)$ vanish.\footnote{The argument is similar to the third paragraph of Step 1 of the proof of Lemma \ref{Lem:ZLT}.}
Choose a subspace $K \subseteq H^0(\LL_{L/B}|_l)$ such that $H^0(\LL_{M/B}|_m)\to H^0(\LL_{L/B}|_l) \twoheadrightarrow H^0(\LL_{L/B}|_l)/K $
is an isomorphism.
Observe that the composition
\[\I_{l/L} \to \LL_{l/L}[-1] \to \LL_{L/B}|_l \to H^0(\LL_{L/B}|_l) \textin \QCoh_L\]
is connective, since $l : \Spec(\C) \to L$, $\Spec(\C) \xrightarrow{l} L \to B$ are closed embeddings and $L \to B$ is affine. Hence $E\dual:=K\otimes \O_L \to K\hookrightarrow H^0(\LL_{L/B}|_l)$ can be lifted to a map $E\dual\to \I_{l/L}$,
which is equivalent to a section $s:\O_L \to E$ such that the point $l:\Spec(\C) \to L$ lies in $\Zero(s)$.
Then 
$H^{-1}(\LL_{\Zero(s)/B}|_l)=0$ since $K\hookrightarrow H^0(\LL_{L/B}|_l)$
and the map $H^{0}(\LL_{\Zero(s)/B}|_l) \to H^0(\LL_{\Zero(s)/M}|_l)$ is zero since  $H^0(\LL_{M/B}|_m)\simeq H^0(\LL_{L/B}|_l)/K$.
Hence $\Zero(s) \to B$ is smooth of minimal dimension at $l$.
Replace $L$ with $\Zero(s)$.
\end{proof}

\begin{lemma}[Lifting non-degeneracy II]\label{Lem:LiftingNondeg2}
Consider a correspondence
\[\xymatrix{
& L\ar[ld]_s \ar[rd]^t & \\
N && M
}\]
of finitely presented derived Artin stacks over a derived stack $B$.
Let $l\in L(\C)$ with
\begin{enumerate}
\item [(A1)] $\TT_{L/B}|_{l} \to \TT_{M/B}|_{t(l)}$ has a retract, $\TT_{L/B}|_l \to \TT_{N/B}|_{s(l)}$ has a section,
\item [(A2)] $\TT_{L/M}|_l[1]$ and $(\TT_{L/M}|_{l}[1]\oplus \TT_{N/B}|_{s(l)})\dual[d]$ have no common amplitudes.
\end{enumerate}
Given $\theta_M \in \sA^{2,\cl }(M/B,d)$ and $\theta_N \in \sA^{2,\cl }(N/B,d)$ with $\eta:\theta_N|_{L} \cong \theta_M|_L$, we have:
\begin{align*}
 &M \text{ is symplectic (in an open neighborhood of $m\in M$) } 
\\ 
&\iff  
\begin{cases}
$N$ \text{ is symplectic (in an open neighborhood of  $s(l)\in N$)}
\\
L : N \dashrightarrow M \text{ is a Lagrangian correspondence}\\
\text{(in an open neighborhood of  $l \in L$)}
\end{cases}
\end{align*}
\end{lemma}

\begin{proof}
We omit the proof since it is analogous to the proof of Lemma \ref{Lem:LiftingNondeg}.
\end{proof}

We now have all the ingredients to prove Proposition \ref{Prop:SttoSch} (and Corollary \ref{Cor:SLS}).

\begin{proof}[Proof of Proposition \ref{Prop:SttoSch}]
By Lemma \ref{Lem:RedBase}, we may assume that the $B$ is a derived scheme 
(since $\sA^{1,\lc }(V/B,d+1)^w \xrightarrow{\simeq}\sA^{1,\lc }(V/U,d+1)^{w|_U}$ for any $V \to U \to B$).
Choose a minimal chart $L \to M$ as in Lemma \ref{Lem:MinChart} and a minimal zero locus tower $L_{(\bullet)}$ as in Lemma \ref{Lem:ZLT}.
Then we can apply Lemma \ref{Lem:LiftingNondeg2} for $L_{(-d)} \leftarrow L \rightarrow M$ since the assumptions (A1), (A2) follow from the minimal dimension conditions. 
\end{proof}

\begin{proof}[Proof of Corollary \ref{Cor:SLS}]
From the smooth symplectic covers in Proposition \ref{Prop:SttoSch} and the functoriality of symplectic pushforwards (Lemma \ref{Lem:SP_Functoriality}), the smooth local structure theorem for stacks (Corollary \ref{Cor:SLS}) can be reduced to the Zariski/{\'e}tale local structure theorem for scheme (Remark \ref{Rem:ZLS}).
\end{proof}

\section{Virtual Lagrangian cycles}\label{Sec:VLC}

In this section, we study an application to enumerative geometry (Theorem \ref{Thm_C:VLC}): 
the {\em virtual Lagrangian cycles} for $(-2)$-symplectic fibrations.
We show that they are {\em uniquely} determined by natural functorial properties (\S\ref{ss:VLC}).
The key ingredient is the {\em symplectic deformation} of a symplectic fibration to the normal bundle (\S\ref{ss:SD}).

\subsection{Symplectic deformations}\label{ss:SD}

In this subsection, we construct symplectic deformations of locked symplectic fibrations to the zero sections of normal bundles (Corollary \ref{Cor:SD}).
More generally, we construct deformations of locked forms to the zero sections (Proposition \ref{Prop:DefLF}) using the double deformation spaces (Lemma \ref{Lem:DD}).

Note that a finitely presented morphism of derived stacks $g:M \to B$ has a canonical deformation to the zero section of the normal bundle $0:M \to \T_{M/B}[1]$;
we have a canonical map $G:M\times \AA^1 \to \D_{M/B}$ whose general/special fiber is:
\[(G:M\times \AA^1 \to \D_{M/B})\times_{\AA^1}\{\zeta\}\simeq 
\begin{cases}
g:M \to B  & \text{if }\zeta\neq0, \\
0:M \to \T_{M/B}[1] &\text{if } \zeta=0.
\end{cases}
\]
Moreover, any locked form on a general fiber induces a locked form on the special fiber; we have a canonical specialization map
\[\Sp : \sA^{p,\lc}(M/B,d) \xrightarrow{(-)^p} \sA^{p}(M/B,d)\simeq \scS^{p}(\T_{M/B}[1],p+d) \xrightarrow{\scZ}\sA^{p,\lc}(M/\T_{M/B}[1],d),\]
where the map $\scZ$ is given as in Proposition \ref{Prop:LFZero}.
We show that this specialization map can be given ``continuously'' as follows:

\begin{proposition}[Deformations of locked forms]\label{Prop:DefLF}
Let $g:M \to B$ be a finitely presented 
morphism of derived stacks.
Then there exists a canonical map
\[\Def : \sA^{p,\lc}(M/B,d) \to \sA^{p,\lc}(M\times\AA^1/\D_{M/B},d) \]
whose general/special fiber is:
\[(\bullet \xrightarrow{\zeta} \AA^1)^* \circ \mathrm{Def} \simeq\begin{cases}
\zeta^{-p }\cdot(-) : \sA^{p,\lc}(M/B,d) \to  \sA^{p,\lc}(M/B,d) & \text{if }\zeta\neq0, \\
\Sp  :  \sA^{p,\lc}(M/B,d) \to  \sA^{p,\lc}(M/\T_{M/B}[1],d)	 & \text{if }\zeta=0.
\end{cases}\]
\end{proposition}

As a special case, symplectic deformations can be obtained.

\begin{corollary}[Symplectic deformations]\label{Cor:SD}
Let $g:M \to B$ be a $d$-shifted locked symplectic fibration with $\theta\in \sA^{2,\lc}(M/B,d)$.
If $g$ has quasi-affine diagonal, then
\[\left(G: M \times \AA^1 \to \D_{M/B},\, \Def(\theta) \in \sA^{2,\lc}(M\times\AA^1/\D_{M/B},d)\right)\]
is also a $d$-shifted locked symplectic fibration.
\end{corollary}

\begin{proof}
The geometricity of the deformation space $\D_{M/B}$ follows from \cite[Thm.~5.1.1]{HP}.\footnote{We refer to \cite{HKR} for the geometricity of $\D_{M/B}$ without the quasi-affine diagonal assumption.}
The non-degeneracy of the deformed $2$-form $\Def(\theta)$ follows from Proposition \ref{Prop:DefLF} (and Proposition \ref{Prop:LFZero}) since it can be checked fiberwise.
\end{proof}

Since the locked forms can be viewed as functions on the deformation spaces (Proposition \ref{Prop:DRviaDefSp}),
Proposition \ref{Prop:DefLF} can be lifted to a geometric statement involving the {\em double deformation space},
\[\DD_{M/B}:=\D_{M\times\AA^1/\D_{M/B}}.\]
We observe that the double deformation space $\DD_{M/B}$ is the {\em doubling}
of the ordinary deformation space $\D_{M/B}$ under the {\em multiplication map},\footnote{Note that the pushfoward along $B\times\AA^1/\GG_m\times\AA^1/\GG_m \to B\times\AA^1/\GG_m$ is equivalent to the diagonal functor $(-)^{\diag}:\QCAlg_B^{\fil,\fil} \to \QCAlg_B^{\fil}$ where $\Fil^pC^\diag:=\Fil^p_I\Fil^p_{II}C$, and the pullback is equivalent to the doubling $(-)^{\doub}:\QCAlg_B^{\fil} \to \QCAlg_B^{\fil,\fil}$ where $\Fil^p_I\Fil^q_{II}C^{\doub}:=\Fil^{\max(p,q)}C$.} 
\[\mu :B\times \AA^1 \times \AA^1 \to B\times \AA^1 : (b,x,y)\mapsto (b,xy).\]

\begin{lemma}[Double deformation spaces]\label{Lem:DD}
Given a geometric morphism of derived stacks $g:M \to B$, we have a canonical fiber square of derived stacks
\begin{equation}\label{Eq:DD}
\xymatrix{
 \DD_{M/B} \ar@{.>}[r]^{\scD} \ar[d] \cart & \D_{M/B} \ar[d] \\
B\times\AA^1\times\AA^1 \ar[r]^-{\mu} & B\times\AA^1.
}\end{equation}
Moreover, this diagram is $\GG_m\times\GG_m$-equivariant with the canonical actions in the left two objects and the induced actions on the right two objects via the multiplication map $\mu:\GG_m\times\GG_m \xrightarrow{}\GG_m$.
\end{lemma}

\begin{proof}
Note that the projection $p:\DD_{M/B}:=\D_{M\times\AA^1/\D_{M/B}} \to \D_{M/B} \times\AA^1$ fits into a fiber square (analogous to \eqref{Eq:D_U/E} in Lemma \ref{Lem:D_U/E}),
\begin{equation}\label{Eq:3}
\xymatrix{
\uMap_{\D_{M/B}\times \AA^1}(\D_{M/B}\times 0,M \times \AA^2) \ar[r]^-p \ar[d] \cart & \uMap_{B\times\AA^2}(B\times 0\times \AA^1,M \times \AA^2) \ar[d] \\
\uMap_{B\times\AA^2}(B\times \AA^1 \times 0,M \times \AA^2) \ar[r] & \uMap_{B\times\AA^2}(B\times \AA^1\times 0, \D_{M/B}\times \AA^1),
}\end{equation}
where
$\uMap_{B\times\AA^2}(B\times \AA^1\times 0, \D_{M/B}\times \AA^1)  \simeq \uMap_{B\times\AA^2}(B\times 0, M\times \AA^2)$.
Since 
\[\xymatrix{
B\times (\AA^1\times 0\cup 0\times\AA^1) & B\times0\times\AA^1 \ar[l]\\
B\times\AA^1\times 0 \ar[u] & B \times 0 \ar[l] \ar[u]
}\]
is a {\em pushout square} in the category of derived Artin stacks \cite[Thm.~5.6.4]{LurDAG}, we have canonical equivalences
\[\DD_{M/B} \simeq \uMap_{B\times\AA^2}(B\times (\AA^1\times 0\cup 0\times\AA^1),M\times \AA^2) \simeq \mu^*\uMap_{B\times\AA^1}(B\times0,M\times\AA^1).\]
It completed the proof.
\end{proof}


\begin{proof}[Proof of Proposition \ref{Prop:DefLF}]
Consider the ``second-coordinate'' $\GG_m$-action on $\DD_{M/B}$, i.e., $\GG_m$-action induced by $\uMap_{\D_{M/B}\times\AA^1/\GG_m}(\D_{M/B}\times B\GG_m, M \times \AA^1 \times \AA^1/\GG_m).$
Let
\[\Def:=\scD^* : \sA^{p,\lc}(M/B,d) \to \sA^{p,\lc}(M\times\AA^1/\D_{M/B},d)\]
be the pullback of weight $p$ functions along the map $\scD:\DD_{M/B}\to\D_{M/B}$ in \eqref{Eq:DD}.
Here locked forms are identified to equivariant functions via Proposition \ref{Prop:DRviaDefSp}.
Since the special fiber of the above fiber square \eqref{Eq:3} is the fiber square \eqref{Eq:D_U/E} in Proposition \ref{Prop:LFZero} for $\T_{M/B}[1]$, we have $0^* \circ \Def \simeq \Sp$ as desired.
\end{proof}


\begin{remark}[Functoriality]\label{Rem:SDFunct}
If $L : M \dashrightarrow N$ is a $w$-locked Lagrangian correspondence of $w$-locked symplectic fibrations $M\to B$ and $N\to B$, then
\[\xymatrix@R-.5pc{
& L \times\AA^1 \ar[ld] \ar[rd] \ar[d] &\\
M \times \AA^1 \ar[d] & \D_{L/B} \ar[ld] \ar[rd]& N \times \AA^1 \ar[d] \\
\D_{M/B} && \D_{N/B}
}\]
is a {\em relative version} of a Lagrangian correspondence, i.e. the induced correspondence
\[L\times \AA^1 : (M\times\AA^1)\times_{\D_{M/B}}\D_{L/B} \dashrightarrow (N\times\AA^1)\times_{\D_{N/B}}\D_{L/B}\]
is a Lagrangian correspondence of symplectic fibrations over $\D_{L/B}$.
\end{remark}

\subsection{Characteristic properties}\label{ss:VLC}

In this subsection, we show that the virtual Lagrangian cycles for $(-2)$-symplectic fibrations are the unique bivariant classes over the exact loci satisfying natural functorial properties.
The symplectic deformations (in \S\ref{ss:SD}) allows us to reduce the statement to the symplectic zero loci (in \S\ref{ss:SZ}).

To construct virtual Lagrangian cycles from $(-2)$-symplectic fibrations, we additionally need {\em orientations}.\footnote{We follow the sign conventions in \cite{KM,OT}; especially, for a perfect complex $E$ on $M$, we use the pairing $\det(E)\otimes \det(E\dual) \simeq \O_M$ in \cite[Eq.~(8), (57)]{OT} for the choice of $\det(E\dual)\simeq \det(E)\dual$.}

\begin{itemize}
\item An {\em orientation} of a $(-2)$-shifted locked symplectic fibration $g:M \to B$ is an equivalence $\ori_M:\O_M \simeq \det(\TT_{M/B}[1])$ such that
\[\qquad\ori^2_M:=(\O_M \xrightarrow[\ori_M]{\simeq}\det(\TT_{M/B}[1]) \xrightarrow[\theta_M]{\simeq} \det(\LL_{M/B}[-1])\xrightarrow[\ori_M\dual]{\simeq}\O_M) \simeq (-1)^{m(m-1)/2}\] 
where $m:=\rank(\TT_{M/B})$. 
This is an analog of \cite[Def.~2.1]{OT}.
\item We say that a Lagrangian correspondence $L:M \dashrightarrow N$ of oriented $(-2)$-shifted locked symplectic fibrations over $B$ is {\em oriented} if
\[\qquad\ori_M\cdot\ori_N:=\left(
\begin{matrix}
\O_L \xrightarrow[\ori_M\otimes \ori_N]{\simeq}\det(\TT_{M/B}[1]|_L)\otimes \det(\TT_{N/B}[1]|_L) \\
\qquad\qquad\xrightarrow[\eqref{Eq:LagCorr}]{\simeq} \det(\TT_{L/B}[1])\otimes\det(\LL_{L/B}[-1])\xrightarrow[\OT]{\simeq}\O_L
\end{matrix}
\right) \simeq (-\sqrt{-1})^l,\]
where $l:=\rank(\TT_{L/B})$.
This is an analog of \cite[Eq.~(18)]{OT}.
\end{itemize}
Note that the symplectic pushforwards and pullbacks of oriented symplectic fibrations (via Theorem \ref{Thm:SP}) carry induced orientations.
Indeed, there are obvious choices for the pullbacks;
for pushforwards $p_*N\to B$ of $N \to U$, we use the orientations that make the counit Lagrangian correspondences $p^*p_*N \dashrightarrow N$ oriented.


We use the Chow groups in \cite{Kre};
for any derived $1$-Artin stack $M$, denote by $A_*(M):=A_*(M_\cl,\Q)$ the Chow group of its classical truncation $M_\cl$ with rational coefficients.
Thus we consider the following technical assumptions.\footnote{If we use the motivic Borel-Moore homology in \cite{Khan} as our intersection theory and the representability result in \cite{HKR}, then the assumptions (A2) and (A3) can be removed (using the arguments in \cite[Ap.~B]{BKP}).
It is desirable to remove the assumption (A1) as well, but the author does not know how to do so.}

\begin{assumption}\label{Assumption:VLC}
Let $g:M \to B$ be a morphism of derived stacks satisfying:
\begin{enumerate}
\item [(A1)] $M_\cl$ is the quotient stack of a quasi-projective scheme by a linear action of a linear algebraic group;
\item [(A2)] $B_\cl$ is an $1$-Artin stack of finite type with affine stabilizers;
\item [(A3)] $g_\cl:M_\cl \to B_\cl$ is of Deligne-Mumford type (i.e. $\LL_{M/B}$ is connective).
\end{enumerate}
\end{assumption}

\begin{theorem}[Virtual Lagrangian cycles]\label{Thm:VLC}

Let $g:M \to B$ be an oriented locked $(-2)$-symplectic fibration satisfying Assumption \ref{Assumption:VLC} and $w:B \to \AA^1$ be the underlying function.
Then there exists a canonical map
\[[M/B]^\lag : A_*(\Zero(w)) \to A_{*+\frac12\rank(\TT_{M/B})}(M),\] 
satisfying the following properties:
\begin{enumerate}
\item [(P1)] (Base change I: bivariance)
Consider a pullback square
\[\xymatrix{
p^*M \ar[r]^-{p_M} \ar[d] \cart & M \ar[d]^{g}\\ 
U \ar[r]^-{p} & B
}\]
and denote by $p_Z:  \Zero(w|_U) \to \Zero(w)$ the restriction of $p:U \to B$.
\begin{enumerate}
\item If $p_\cl$ is projective, then we have
\[ [M/B]^\lag \circ (p_Z)_* = (p_M)_* \circ [p^*M/U]^\lag.\]
\item If $p$ is quasi-smooth (and $p_\cl$ is quasi-projective), then we have
\[p_M^! \circ [M/B]^\lag =[p^*M/U]^\lag \circ p_Z^!, \]
where $(-)^!$ denotes the quasi-smooth (virtual) pullbacks \cite{BF,Man}.
\end{enumerate}
\item [(P2)]  (Base change II) Consider a pushforward square (Remark \ref{Rem:bexLagCorFib})
\[\xymatrix{
p_*N  \ar[r]^-{r} \ar[d]  & N\ar[d]^{h}\\ 
B & \ar[l]_-{p} U
}\]
for an oriented $w|_U$-locked symplectic fibration $h:N \to U$ satisfying Assumption \ref{Assumption:VLC}.
If $p$ is smooth (and $p_\cl$ is quasi-projective), then $r_\cl$ is an equivalence and we have
\[[p_*N/B]^\lag =[N/U]^\lag \circ p_Z^!.\]
\item [(P3)] (Functoriality) Consider an oriented Lagrangian correspondence
\[\xymatrix{
& C\ar[ld]_g \ar[rd]^f \\
M && L
}\]
of oriented locked $(-2)$-symplectic fibrations $M,L \to B$ satisfying Assumption \ref{Assumption:VLC}.
If $f$ is quasi-smooth and $g_\cl$ is an isomorphism, then we have
\[[M/B]^\lag = f^! \circ [L/B]^\lag.\]
\end{enumerate}
Moreover, the maps $[M/B]^\lag$ are uniquely determined by the above properties.
\end{theorem}

We note that the base change property (P2) follows from the other properties.

\begin{remark}[Base change II is superfluous]\label{Rem:P2}
In the situation of (P2), 
we have the counit Lagrangian correspondence
\[\xymatrix{
& p_*N\ar[ld]_s \ar[rd]^t \\
p^*p_*N && N,
}\]
such that $s:p_*N \to p^*p_*N$ is quasi-smooth and  $t_\cl:(p_*N)_\cl\to N_\cl$ is an isomorphism.
Hence the functoriality (P3) gives us
\[[N/U]^\lag= (p_*N \to p^*p_*N)^! \circ [p^*p_*N/U]^\lag.\]
By the bivariance $(P2)$, the compositions with the smooth pullback $p_Z^!$ give us
\begin{align*}
[N/U]^\lag \circ p_Z^!
&= (p_*N \to p^*p_*N)^! \circ [p^*p_*N/U]^\lag \circ  p_Z^! \\
&=(p_*N \to p^*p_*N)^! \circ (p^*p_*N \xrightarrow{} p_*N)^!\circ [p_*N/B]^\lag = [p_*N/B]^\lag.
\end{align*}
\end{remark}

It is necessary to consider the {\em exact loci} $\Zero(w)$,\footnote{Recall from \S\ref{ss:DR} that the underlying function $w: B \to \AA^1$ of a locked $(-2)$-symplectic fibration $g:M \to B$ is the obstruction to exactness. Hence giving exact structures on the base changes $p^*M \to U$ along $p:U \to B$ is equivalent to giving factorizations $U \to \Zero(w)$ of $p:U \to B$.} instead of $B$, in Theorem \ref{Thm:VLC}.

\begin{remark}[Counterexample without exactness]\label{Rem:ExactNecessary}
Consider a special orthogonal bundle $E=F \oplus F\dual$ over a scheme $U$ and the symplectic zero section (\S\ref{ss:SZ})
\[(0_E^\symp : U \to E) \in \Symp_{E,-2}^{\q_E}.\]
If Theorem \ref{Thm:VLC} holds with $\Zero(w)$ replaced by $B$, then we should have a map
$[U/E]^\lag : A_*(E) \to A_{*-r}(U)$ satisfying the following equation:
\begin{equation}\label{Eq:4}
e(E) \circ [U/E]^\lag \circ \pi_E^! = e(F)  : A_*(U) \to A_{*-\rank(F)}(U),
\end{equation}
where $e(E)$ is the top Chern of $E$ and $\pi_E^!$ is the smooth pullback of the projection $\pi_E:E \to U$.
Indeed, we have a fiber square and a Lagrangian correspondence
\[\xymatrix{
E[-1] \ar[r] \ar[d] \cart & U \ar[d]^{0_E} \\
U \ar[r]^{0_E} & E,
}\qquad\qquad
\xymatrix{
& F[-1] \ar[ld] \ar[rd] & \\
E[-1] && U.
}
\]
By the bivariance (P1)(b), we have $e(E) \circ [U/E]^\lag = [E[-1]^\lag/U] \circ 0_E^!$;
by the functoriality (P3), we have $[E[-1]^\lag/U] = e(F)$.
The compositions with $\pi_E^!$ give us the claimed formula \eqref{Eq:4}.
However, there is a simple counterexample to this formula \eqref{Eq:4}:
if $U=\PP^1$ and $F=\O_{\PP^1}(1)$,
then $e(E)=0$, but $e(F)\neq 0$.
\end{remark}

We discuss how Theorem \ref{Thm:VLC} is related to its classical shadow in \cite{Park}.

\begin{remark}[Symmetric obstruction theories]\label{Rem:SOT}
Given a $(-2)$-symplectic fibration $g:M \to B$, the canonical map
\[\LL_{M/B}|_{M_\cl} \to \LL_{M_\cl/B_\cl} \to  \LL_{M_\cl/B_\cl} ^{\geq-1}\]
has $2$-connective cofiber by \cite[Cor.~7.4.3.2]{LurHA}.
Equivalently, $H^i$ are bijective for $i\geq 0$ and is surjective for $i=-1$.
The above map is called a {\em symmetric obstruction theory} in \cite[Def.~1.9]{Park} (when $M$ and $B$ are derived schemes).

If the symplectic form on $g:M \to B$ admits an exact structure, then the intrinsic normal cone $\fC_{M/B}:=\fC_{M_\cl/B_\cl}$ in \cite{BF,AP} is {\em isotropic} (cf. \cite[Def.~1.10]{Park}), that is,
\[(\fC_{M/B}  \hookrightarrow \T_{M/B}[1] \xrightarrow{\q} \AA^1) \simeq 0,\]
where $\q$ is the quadratic function induced by the underlying $2$-form.
This follows from Proposition \ref{Prop:DRviaDefSp} since $\fC_{M/B}$ is the flat limit inside the classical deformation space \cite{Ful,Kre,AP}.
This isotropic condition is needed for the construction of $[M/B]^\lag : A_*(B) \to A_*(M)$ in \cite{Park}.


A Lagrangian correspondence $C:M \dashrightarrow L$ gives rise to a commutative diagram
\[\xymatrix{
\TT_{C/B}[2] \ar[r] \ar[d] \cart & \TT_{L/B}[2]|_C \simeq \LL_{L/B}|_C \ar[r] \ar[d] & \LL_{C/M} \ar@{=}[d]\\
\TT_{M/B}[2]|_C \simeq \LL_{M/B}|_C \ar[r] & \LL_{C/B} \ar[r] & \LL_{C/M}.
}\]
When $C \to L$ is quasi-smooth and $C_\cl \simeq L_\cl$, the above morphism of cofiber sequences gives us the compatibility condition in \cite[Def.~2.1]{Park}, needed for the functoriality.
\end{remark}

Theorem \ref{Thm:VLC} is a relative version of the virtual cycles for Donaldson-Thomas theory of Calabi-Yau $4$-folds introduced in \cite{BJ,OT}.\footnote{We will follow the construction of virtual Lagrangian cycles in \cite{Park} which is obviously equivalent to \cite{OT} in the absolute setting. However, the comparison of \cite{OT} and \cite{BJ} is not obvious; it is shown in \cite{OT2} through a beautiful but complicated proof. Thus it is desirable to prove the comparison of \cite{OT} with \cite{BJ} (and also \cite{Pri2}) via the uniqueness in Theorem \ref{Thm:VLC}.}

\begin{remark}[Deformation invariance issue]\label{Rem:BJ/OT}
Already in \cite{BJ,OT}, relative versions of virtual Lagrangian cycles are presented;
in the enumerative geometry perspective, the deformation invariance of the virtual Lagrangian cycles are stated.
However, there are some mistakes.
As we observed in Remark \ref{Rem:ExactNecessary}, without assuming the exactness of symplectic forms, it is not possible to define functorial bivariant classes over the entire bases of $(-2)$-symplectic fibrations.

More specifically, \cite{BJ,OT} used an imprecise form of relative Darboux theorem.
By Corollary \ref{Cor:ELS}, locked\footnote{All $(-2)$-symplectic fibrations are locked, formally locally on the bases by Proposition \ref{Prop:Locking}.} $(-2)$-symplectic fibrations are locally the (symplectic) zero loci of sections of orthogonal bundles that are {\em not necessarily isotropic}.
However, in \cite[Thm.~3.22]{BJ} and \cite[pp.~35-36]{OT}, these sections are assumed to be isotropic, which is equivalent to assume that the symplectic forms are exact.
Unlike the absolute case (i.e. $B=\Spec(\C)$), there are many non-exact $(-2)$-symplectic forms in the relative setting (e.g. zero sections of orthogonal bundles in \S\ref{ss:SZ}).

Theorem \ref{Thm:VLC} (and Remark \ref{Rem:ExactNecessary}) says that the virtual Lagrangian cycles are deformation invariant along the exact loci rather than along the entire bases.
\end{remark}


Before proving Theorem \ref{Thm:VLC}, we first consider the {\em local model}.
Let $E$ be an orthogonal bundle over a classical Artin stack $U$.
The zero locus $\Zero(s)$ of a section $s: U \to E$ has a canonical $(-2)$-symplectic form over $U$ (Definition \ref{Def:SympZero}); we have
\[\SZ(E,s):=\SZ_{U/U}(E,s) \in \Symp_{U,-2}^{s^2}.\]
Moreover, given an isotropic subbundle $K \subseteq E$ with $s \cdot K=0$, 
we have induced sections $s_1\in \Gamma(U,K^\perp)$ and $s_2 \in \Gamma(U,K^\perp/K)$ with a Lagrangian correspondence
\begin{equation}\label{Eq:5.1-LagCorrSZ}
\Zero(K^\perp,s_1) : \SZ(E,s) \dashrightarrow \SZ(K^\perp/K,s_2),
\end{equation}
such that the classical truncation of $\Zero(K^\perp,s_1) \to \SZ(E,s)$ is an equivalence and $\Zero(K^\perp,s_1) \to \SZ(K^\perp/K,s_2)$ is quasi-smooth.
Indeed, we have a maximal isotropic correspondence (in the sense of \eqref{Eq:MaxIsoCorr} in \S\ref{ss:SZ})
\[K^\perp : E \dashrightarrow K^\perp/K,\]
and the claimed Lagrangian correspondence \eqref{Eq:5.1-LagCorrSZ} can by obtained by pulling back the Lagrangian correspondence \eqref{Eq:LagCorrSympZeroSections} in Lemma \ref{Lem:FunctZeroSection} by the section $s_1:U \to K^{\perp}$.

\begin{proposition}[Local model]\label{Prop:VLC-Local}
Let $U$ be 
the quotient stack of a quasi-projective scheme by a linear action of a linear algebraic group,
$E$ be a special orthogonal bundle over $U$,
and $s\in \Gamma(U,E)$ be a section.
Then there exists a canonical map
\[\sqrt{e}(E,s) : A_*(\Zero(s^2)) \to A_{*-\frac12\rank{E}}(\Zero(s))\]
satisfying the following properties:
\begin{enumerate}
\item (Bivariance) $\sqrt{e}(E,s) $ commutes with projective pushforwards and (quasi-projective) lci pullbacks (along base changes of $U$).
\item (Reduction formula) If $K \subseteq E $ is an isotropic subbundle with $s\cdot K=0$, then
\[\sqrt{e}(E,s) = (\Zero(K^{\perp},s_1) \to \Zero(K^{\perp}/K,s_2))^!\circ \sqrt{e}(K^\perp/K,s_2). 
\]
\end{enumerate}
Moreover, the maps $\sqrt{e}(E,s)$ are uniquely determined by the above properties.
\end{proposition}

The map $\sqrt{e}(E,s)$ is constructed in \cite{OT} and the bivariance and reduction formula are shown in \cite[Lem.~4.4, Lem.~4.5]{KP}.
Although the uniqueness follows from the construction in \cite[Def.~4.1]{KP},
we provide a proof here for reader's convenience.

\begin{proof}[Proof of the uniqueness part of Proposition \ref{Prop:VLC-Local}]

We may assume that $U$ is a quasi-projective scheme, using Totaro's approximations of classifying stacks \cite{Tot,EG2}.
Indeed, write $U=[P/G]$ for a quasi-projective scheme $P$ with a linear action of a linear algebraic group $G$.
For each integer $k$, we can find a $G$-representation $V$ with a $G$-invariant open subscheme $W \subseteq V$ whose complement has codimension $\geq k$ and $W/G$ is quasi-projective.
By the homotopy property of Chow groups \cite[Cor.~2.5.7]{Kre}, we can replace the stack $U$ with the scheme $(P\times W)/G$.

It suffices to consider the case when $s=0$.
Indeed, we may assume that $s$ is {\em isotropic}, i.e. $s^2=0$, after replacing $U$ with $\Zero(s^2)$.
Let $V$ be the (classical) blow up of $U$ along $\Zero(s)_\cl$ and $D\subseteq V$ be the exceptional divisor.
By the blowup sequence 
\[\xymatrix{A_*(D) \ar[r] & A_*(V)\oplus A_*(\Zero(s)) \ar[r] & A_*(U) \ar[r] & 0},\]
in \cite[Ex.~1.8.1]{Ful}, it suffices to consider $V$ (and $\Zero(s)$).
Note that $\O_V(D)$ is an isotropic subbundle of $E|_V$ which contains $s|_V$ (see \cite[Lem.~2.1]{KP}).
By the reduction formula, we can replace $E|_V$ with $\O_V(D)^{\perp}/\O_V(D)$ and assume that $s=0$.

It suffices to consider the case when $E$ has a {\em maximal} isotropic subbundle, i.e. an isotropic subbundle of rank $r={\rank(E)}/{2}$.
Indeed, let $F:=(T\to U) \mapsto \{F_1 \subseteq F_2 \subseteq \cdots F_{r-1}\subseteq E|_T\}$ be the maximal isotropic flag variety, where the pullback $E|_F$ has a maximal isotropic subbundle $K$ (see \cite[\S6]{EG}).
The projection  $F \to U$ is a composition of smooth quadric bundles; in particular, it is proper surjective and we have a proper codescent sequence \cite[Prop.~1.3(2)]{Kimura},
\[\xymatrix{
A_*(F\times_U F) \ar[r] & A_*(F) \ar[r] & A_*(U) \ar[r]& 0.
}\]
Replace $U$ with $F$ and assume that $E$ has a maximal isotropic subbundle $K$.

Finally, the reduction formula implies:
\[\sqrt{e}(E,s) = (K[-1]\to U)^! : A_*(U) \to A_{*-r}(U),\] 
where $(K[-1]\to U)^!=e(K)$ is the top Chern class of $K$.
\end{proof}

We now prove Theorem \ref{Thm:VLC}.
We first consider the {\em uniqueness} part.
Assume that we are given an assignment
\[\begin{Bmatrix}
\text{oriented $w$-locked $(-2)$-symplectic fibrations} \\
\text{$g:M \to B$ satisfying Assumption \ref{Assumption:VLC}}
\end{Bmatrix} \to \left\{[M/B]^\lag : A_*(\Zero(w)) \to A_*(M)\right\}\]
satisfying the properties (P1), (P2), (P3) in Theorem \ref{Thm:VLC}.

\subsubsection*{Step 1: Deformations to normal bundles}

We use symplectic deformations (Corollary \ref{Cor:SD}) to replace general symplectic fibrations with the zero sections of symmetric complexes.
Observe that a $\GG_m$-equivariant function $W: \D_{M/B}\to \AA^1$ (equivalently, a $(-p)$-shifted locked $p$-form) gives rise to a {\em localized} specialization map
\[\sp^{\loc}_{M/B} : A_*(\Zero_B(W_1) ) \to A_*(\Zero_{\T_{M/B}[1]}(W_0)), \]
where $W_0 :\T_{M/B}[1]$ (resp. $W_1: B \to \AA^1$ ) is the restriction of $W$ to the special (resp. general) fiber.
The construction is completely analogous to the ordinary specialization map in \cite[\S5.2]{Ful}; 
define $\sp^{\loc}_{M/B}$ as the unique map that fits into 
\[\xymatrix{
A_{*+1}(\Zero_{\T_{M/B}[1]}(W_0)) \ar[r] \ar[rd]_{c_1(\O)=0} & A_{*+1}(\Zero_{\D_{M/B}}(W)) 
\ar[r] \ar[d]^{0^!} & A_{*+1}(\Zero_B(W_1) \times \GG_m) \ar[r] \ar[ld] & 0  \\
& A_*(\Zero_{\T_{M/B}[1]}(W_0))   & \ar@{-->}[l]_-{\sp^\loc_{M/B}} A_*(\Zero_B(W_1)) \ar[u]_{\pr_1^*} ,
}\]
where the upper row is the excision sequence \cite[Prop.~2.3.6]{Kre}
and $0^!$ is the Gysin map for the divisor $\Zero_{\T_{M/B}[1]}(W_0) \subseteq \Zero_{\D_{M/B}}(W)$.\footnote{The deformation space $\D_{M/B}$ is a derived $1$-Artin stack by \cite[Thm.~5.1.1]{HP}. Indeed, since $M$ is global quotient and $B$ is $1$-Artin, the diagonal of $M$ is affine and the diagonal of $B$ has separated diagonal, and hence the diagonal of $g:M \to B$ separated. Since $g$ is of Deligne-Mumford type, its diagonal is quasi-finite and hence quasi-affine by the Zariski main theorem \cite[Thm.~6.15]{Knu}.}


\begin{lemma}[Deformations to normal bundles]\label{Lem:VLC-SympDef}
For an oriented $w$-locked $(-2)$-symplectic fibrations $g:M \to B$ satisfying Assumption \ref{Assumption:VLC}, we have
\begin{equation}\label{Eq:VLC-1}
[M/B]^\lag = [M/\T_{M/B}[1]]^\lag \circ \sp^\loc_{M/B} : A_*(\Zero(w)) \to A_*(M).
\end{equation}
\end{lemma}

\begin{proof}
By Corollary \ref{Cor:SD}, the map $M\times\AA^1 \to \D_{M/B}$ has a canonical locked symplectic form
and hence we have a map
\[[M\times\AA^1/\D_{M/B}]^\lag : A_*(\Zero_{ \D_{M/B}}(W)) \to A_*(M\times \AA^1),\]
where $W:\D_{M/B} \to \AA^1$ is the weight $2$ function induced by the locked symplectic form on $g:M \to B$. By the bivariance (P1), we have 
\[[M/B]^\lag \circ 1^! = [M/\T_{M/B}[1]]^\lag \circ 0^! : A_*(\Zero_{ \D_{M/B}}(W)) \to A_*(M)\]
where the two Gysin maps $0^!=1^!:A_*(M\times \AA^1)\to A_*(M)$ are the same by the homotopy property of Chow groups \cite[Cor.~2.5.7]{Kre}. 
Since $1^! : A_{*+1}(\Zero_{\D_{M/B}}(W)) \to A_*(\Zero_B(W_1))$ is surjective and $0^! = \sp_{M/B}^\loc\circ 1^!$, we have the desired equality \eqref{Eq:VLC-1}.
\end{proof}

Consequently, it suffices to prove the uniqueness for the symplectic zero sections $0_E^\symp :M \to \E$ of symmetric complexes $E$ of tor-amplitude $[-1,1]$.

\subsubsection*{Step 2: Reductions to orthogonal bundles}

We use the functoriality of symplectic zero sections (Lemma \ref{Lem:FunctZeroSection}) to replace symmetric complexes with orthogonal bundles.
Note that a symmetric complex $E$ of tor-amplitude $[-1,1]$ 
on a classical $1$-Artin stack $M$ with the resolution property\footnote{The quotient stacks of quasi-projective schemes by linear actions of linear algebraic groups have the {\em resolution property} by \cite[Lem.~2.6]{Tho}. More precisely, all perfect complexes are equivalent to bounded chain complexes of vector bundles.}
admits a {\em symmetric resolution} (cf.~\cite[Prop.~1.3]{Park}), i.e.,
we can find an orthogonal bundle $F$, a vector bundle $D$, and a maximal isotropic correspondence (in the sense of \eqref{Eq:MaxIsoCorr} in \S\ref{ss:SZ})
\[D: E \dashrightarrow F\]
such that $a:D \to E$ is smooth surjective and $b:D \to F$ is a closed embedding.
See \cite[Prop.~4.1]{OT} for the existence of a symmetric resolution.

\begin{lemma}[Reduction to orthogonal bundles]\label{Lem:VLC-Reduction}
Under the above notations, 
\begin{equation}\label{Eq:VLC-2}
[0_E^\symp]^\lag = [0_F^\symp]^\lag \circ (QD \xrightarrow{Qb} QF)_* \circ (QD \xrightarrow{Qa} QE)^* : A_* (QE) \to A_*(M),
\end{equation}
where $QE$, $QD$, $QF$ are the zero loci of the quadratic functions on $E$, $D$, $F$. 
\end{lemma}
\begin{proof}
By Lemma \ref{Lem:FunctZeroSection}, we can form a commutative diagram
\[\xymatrix{
0_E^\symp  \simeq a_*b^* 0_F^\symp \ar[r] \ar[d]  & b^*0_F^\symp \ar[r] \ar[d] & 0_F^\symp \ar[d]^{} 
\\ 
E & \ar[l]_a D  \ar[r]^b & F
}\]
where the left square is a symplectic pushforward square (Remark \ref{Rem:bexLagCorFib}) and the right square is a pullback square.
Note that the classical truncations of the upper arrows are equivalences.
By the base change (P2) and the bivariance (P1), we have
\[[0_E^\symp]^\lag = [b^*0_F^\symp]^\lag \circ Qa^* = [0_F^\symp]^\lag \circ Qb_* \circ Qa^*,\]
as desired. 
\end{proof}

Consequently, it suffices to prove the uniqueness for the symplectic zero sections $0_F^\symp:M \to F$ of orthogonal bundles $F$.
Observe that the symplectic zero section can be regarded as the symplectic zero loci of the tautological section $\tau \in \Gamma(F,F|_F)$,
\[0_F^\symp \simeq \SZ_{F/F}(F|_{F},\tau) \simeq \Symp_{F,-2}^{\q_F}.\]
By the local model (Proposition \ref{Prop:VLC-Local}), we have 
\begin{equation}\label{Eq:VLC-3}
[0_F^\symp]^\lag= \sqrt{e}(F|_F,\tau) : A_*(QF) \to A_*(M),
\end{equation}
where $\sqrt{e}(F|_F,\tau):=\sqrt{e}(F|_{F_\cl},\tau|_{F_\cl})$.

Conversely, the virtual Lagrangian cycles can be constructed by combining the formulas \eqref{Eq:VLC-1}, \eqref{Eq:VLC-2}, \eqref{Eq:VLC-3}:
\begin{equation}\label{Eq:VLC-Construction}
[M/B]^\lag:=\sqrt{e}(F|_F,\tau) \circ (QD \to QF)_* \circ (QD \to QE)^* \circ \sp^\loc_{M/B}.
\end{equation}

\begin{proof}[Proof of Theorem \ref{Thm:VLC}]
Define the virtual Lagrangian cycles as \eqref{Eq:VLC-Construction};
it is independent of the choice of $D,F$ by (the arguments in) \cite[\S4.2]{OT}. 
The bivariance (P1) is shown in \cite[Prop.~1.15, Rem.~2.5]{Park} 
(only the lci pullbacks are considered in \cite{Park}, but the same argument works for the quasi-smooth pullbacks), 
the functoriality (P3) is shown in \cite[Thm.~2.2, Thm.~A.4]{Park}, and the base change (P2) follows from the functoriality (P3) as explained in Remark \ref{Rem:P2}.
The uniqueness follows from Lemma \ref{Lem:VLC-SympDef}, Lemma \ref{Lem:VLC-Reduction}, Proposition \ref{Prop:VLC-Local}, as explained above.
\end{proof}

\section{Moduli of perfect complexes}\label{Sec:Perf}

In this section, we present our main example (Theorem \ref{Thm_D:Perf}): {\em moduli of perfect complexes} for families of Calabi-Yau $4$-folds.
We first provide a general criteria for locking closed forms via {\em underlying formal functions} (\S\ref{ss:Obs}).
We then apply it to {\em mapping stacks} and show that the canonical symplectic forms are locked on components where ``topological types'' are fixed (\S\ref{ss:Map}).



\subsection{Locking closed forms}\label{ss:Obs}

In this subsection, we explain how to lift closed forms to locked forms by reducing it to formal neighborhoods of the bases.

Throughout this subsection, our base stack $B$ is a classical Artin stack of finite type over $\C$,
and all derived Artin stacks are assumed to be of finite type over $B$.

\begin{proposition}[Underlying formal functions]\label{Prop:Locking}
Let $M$ be a derived Artin stack over $B$.
For integers $d \leq -p$, we have a canonical fiber square
\[\xymatrix{
\sA^{p,\lc}(M/B,d) \ar[r]^-{[-]} \ar[d]^{\widehat{(-)}} \cart & \sA^0(B,p+d) \ar[d]\\
\sA^{p,\cl}(M/B,d) \ar@{.>}[r]^-{\langle-\rangle} & \prod_{b\in B(\C)}\prod_{\pi_0(M_b)} \sA^0(\widehat{B}_b,p+d)
}\]	
for some dotted arrow $\langle-\rangle$,
where $\widehat{B}_b$ is the formal completion of $B$ at $b$, and $\pi_0(M_b)$ is the set of connected components of the fiber $M_b$ of $M \to B$ over $b$.\footnote{Here the formal completion is defined as: $\widehat{B}_b:=\varinjlim_{(A,0) \to (B,b)} A$, where $A$ are local Artinian schemes. The functor $\pi_0 :  \dSt \to \Set$ is the left adjoint of the functor of constant derived stacks.}
\end{proposition}

We can interpret Proposition \ref{Prop:Locking} as follows:
\begin{itemize}
\item If $d<-p$, then $d$-shifted closed $p$-forms have unique locking structures.
\item If $d=-p$, then $d$-shifted closed $p$-forms have locking structures if and only if the underlying formal functions converge to global functions on the bases.
Moreover, such locking structures are unique when they exist.
\end{itemize}

Recall from \S\ref{ss:DR} that locked forms are equivalent to closed forms whose associated de Rham forms come from functions on the bases; see the fiber square \eqref{Eq:ex/cl/DR}.
Thus Proposition \ref{Prop:Locking} can be deduced by analyzing non-positively shifted de Rham forms.

\begin{lemma}[$0$-shifted de Rham Forms]\label{Lem:A^DR}
Let $M$ be a derived Artin stack over $B$.
\begin{enumerate}
\item The space $\sA^{\DR}(M/B,0)$ is discrete.
\item The restriction map 
\begin{equation}\label{Eq:InfDetection}
\sA^{\DR}(M/B,0) \hookrightarrow \varprojlim_{A \to B} \sA^{\DR}(M_A/A,0) \text{ is injective},
\end{equation} 
where $A \to B$ are maps from local Artinian schemes $A$.
\item If $B$ is a local Artinian scheme, then we have a canonical equivalence
\begin{equation}\label{Eq:LocArt}
\sA^{\DR}(M/A,0)\xleftarrow{\simeq} \prod_{\pi_0(M)} \sA^0(A,0).
\end{equation}
\end{enumerate}
\end{lemma}

Assuming Lemma \ref{Lem:A^DR}, Proposition \ref{Prop:Locking} follows immediately.
\begin{proof}[Proof of Proposition \ref{Prop:Locking}]
Define the underlying formal functions $\langle-\rangle$ as:
\[
\sA^{p,\cl}(M/B,-p) 
\to \sA^{\DR}(M/B,0) 
\hookrightarrow \varprojlim_{A\to B} \sA^{\DR}(M_A/A,0) 
\xleftarrow{\simeq} \varprojlim_{A\to B}\prod_{\pi_0(M_A)} \sA^0(A,0).\]
Then the fiber square in Proposition \ref{Prop:Locking} follows from the fiber square in \eqref{Eq:ex/cl/DR}.
\end{proof}



We finally prove Lemma \ref{Lem:A^DR} to complete the proof of Proposition \ref{Prop:Locking}.
Lemma \ref{Lem:A^DR} is a consequence of comparison theorem of derived de Rham cohomology and classical de Rham cohomology \cite{Har}, shown in \cite{Bhatt}.\footnote{Under the Hochschild-Kostant-Rosenberg isomorphism, such comparison can also be shown by the results on periodic cyclic homology in \cite{FT,Goo1,Goo2}.} 
Without recalling the definition of classical de Rham cohomology, we can rephrase the comparison theorem by dividing it into the following three pieces:
\begin{itemize}
\item (Descent) Given a commutative triangle of derived affine schemes
\begin{equation}\label{Eq:Fact}
\xymatrix{
& U \ar[d]\\
M \ar[ru] \ar[r] & B,}
\end{equation}
we have a canonical equivalence
\begin{equation}\label{Eq:Descent}
\hDR(M/B) \xrightarrow{\simeq} \Tot \left(\hDR(M/\Cech_*(U/B)\right) \textin \QCAlg^\fil_B.
\end{equation}
This follows from \cite[Cor.~2.7]{Bhatt} 
(see also \cite[Lem.~B.1.1]{CPTVV}).
\item (Nil-invariance) Given a morphism of derived Artin stacks $M \to B$, we have
\begin{equation}\label{Eq:Nilinvariance}
\Fil^0\hDR(M/B) \xrightarrow{\simeq} \Fil^0\hDR(M_\red/B) \textin \QCAlg_B. 
\end{equation}
Indeed, by applying the descent \eqref{Eq:Descent} twice, we may assume that $M \hookrightarrow B$ is a closed embedding of derived affine schemes such that $\I_{M/B}$ is $1$-connective.
Then $\Fil^0\hDR(M/B) \xrightarrow{\simeq} \Fil^0\hDR(M_\red/B) \simeq \O_B$ by \cite[Thm.~8.8]{Qui70}.


\item (Comparison) Given a closed embedding of classical schemes $M \hookrightarrow B$, 
\begin{equation}\label{Eq:Carlsson}
F^0\hDR(M/B) \simeq \widehat{\O}_{B,M}:=\varprojlim_{p\to \infty} {\O_B}/{\I_{M/B}^p}.\footnote{Unlike in \S\ref{ss:SPT}, here the ideal $\I^p_{M/B}\subseteq \O_B$ is considered classically.}
\end{equation}
This is \cite[Thm.~4.4]{Carl}, based on the descent \eqref{Eq:Descent} and the nil-invariance \eqref{Eq:Nilinvariance}. 
\end{itemize}
The completed derived de Rham complexes $\Fil^0\hDR(M/B)$ are uniquely determined by the smooth descent and the above three properties.

\begin{proof}[Proof of Lemma \ref{Lem:A^DR}]
(1) 
Note that the limits of discrete spaces are discrete.
Using the smooth descent of $\hDR$ and the nil-invariance \eqref{Eq:Nilinvariance}, we may assume that $M$ and $B$ are classical affine schemes.
By the descent \eqref{Eq:Descent}, we may assume that $M \hookrightarrow B$ is a closed embedding.
Then the statement follows from the comparison \eqref{Eq:Carlsson}.

(2) Using the smooth descent of $\DR$ and the nil-invariance \eqref{Eq:Nilinvariance}, we may assume that $M$ and $B$ are classical affine schemes.
Consider a factorization \eqref{Eq:Fact} with a closed embedding $M \hookrightarrow U$ and a smooth morphism $U \to B$.
For any morphism $A \to U$ from a local Artinian scheme $A$, we can form a fiber diagram
\[\xymatrix{
M\times_U A \ar[r] \ar[d] \cart & M \times_B A \ar[r] \ar[d] \cart & A \ar[d]\\
M \ar[r] & M\times_B U \ar[r] \ar[d] \cart & U \ar[d] \\
& M \ar[r] & B.
}\]
Hence replacing $B$ with $U$, we may assume that $M \hookrightarrow B$ is a closed embedding.

By the Krull intersection theorem, the canonical map
\[\varprojlim_{p \to \infty} \frac{\O_B}{\I^p_{M/B}} \to \prod_{b\in B(\C)}\varprojlim_{p \to \infty}\varprojlim_{k \to \infty}\frac{\O_B}{\I^p_{M/B} + \I_{b/B}^k} \]
is injective.
Thus the restriction map in \eqref{Eq:InfDetection} is also injective by the comparison \eqref{Eq:Carlsson}.

(3) We first construct the canonical map in \eqref{Eq:LocArt}.
Indeed, there is a canonical map \[\sA^0(A,0) \to \sA^{\DR}_A[0] \textin \dPSt_A\]
of derived prestacks, where $\sA^0(A,0)$ is regarded as a constant prestack.
Then the induced map of $M$-points of the stackification of the above map is the claimed map.
Since we have a functorial map, we may assume that $M$ is a connected affine scheme.

Note that we have a canonical equivalence of filtered algebras
\[\hDR(M/A) \xleftarrow{\simeq} \hDR(M\times A/A) \xleftarrow{\simeq} \hDR(M) \otimes_C \Gamma(A,\O_A).\]
Indeed, the first restriction map is an equivalence by the nil-invariance \eqref{Eq:Nilinvariance},
and the second Kunneth map is an equivalence since the affine morphism $M \to \Spec(\C)$ has the base change property \cite[Prop.~3.10]{BFN} and $\Gamma(A,\O_A)$ is a finite-dimensional $\C$-algebra.
Consequently, it suffices to prove the statement for $A=\Spec(\C)$.

Then the derived de Rham cohomology is just the singular cohomology
\[H^*\Fil^0\hDR(M) \simeq H_{\mathrm{Sing}}^*(M^\mathrm{an},\C),\]
of the underlying analytic space $M^{\mathrm{an}}$, by the \cite[Thm.~4.10]{Bhatt} and \cite[IV.~1.1]{Har}.
Since $M$ is connected, $M^\mathrm{an}$ is also connected, and hence $H_{\mathrm{Sing}}^0(M^\mathrm{an},\C)\simeq \C$.
\end{proof}

\begin{remark}[Absolute case]\label{Rem:Absolute}
If $B=\Spec(\C)$ and $d<0$, then a $d$-shifted closed $p$-form on a derived Artin stack $M$ is exact, as stated in \cite[Cor.~5.3]{Toen} with a sketch proof.
Moreover, 
if $d\leq -p$, then there exists a unique exact structure;
if $d=-p+1$, then there exists a canonical exact structure; 
this is shown in \cite[Prop.~3.2]{KPS}.
\end{remark}

\subsection{Mapping stacks}\label{ss:Map}


In this subsection, we construct locked symplectic forms on mapping stacks with Calabi-Yau sources and symplectic targets.

Let $f:X \to B$ be an $n$-dimensional {\em Calabi-Yau} morphism of derived stacks.
More precisely, $f$ is of finite presentation, universally $f_*:\QCoh_X \to \QCoh_B$ preserves colimits and perfect complexes, and the {\em Calabi-Yau structure} is an equivalence
\[\Omega_{X/B} : \O_X \xrightarrow{\simeq} f^\dag\O_B[-n],\]
where $f^\dag$ is the right adjoint of $f_*$.\footnote{This is the relative version of $\O$-oriented $\O$-compact derived stack in \cite[Def.~2.1, 2.4]{PTVV}.}
For a smooth projective morphism of schemes, Calabi-Yau structures are equivalent to trivializations of the canonical line bundle.
%
%

The main object in this subsection is the mapping stack
\[M:=\uMap_B(X,Y) \to B,\]
for a derived stack $Y$ 
over $B$.
Recall \cite{PTVV} that closed (e.g. symplectic) forms on the target $Y$ give rise to closed (resp. symplectic) forms on the mapping stack $M$.
An analogous construction also works for locked forms;
we have an {\em integration map}
\[\int_{M\times_BX/M,\Omega} (-) : \DR(M\times_B X/B) \to \DR(M/B)[-n] \textin \QCoh_B^\fil,\]
defined as the composition:
\begin{multline*}
\DR(M\times_B X/B) \xleftarrow[\simeq]{\textrm{Kunneth}} \DR(M/B)\otimes \DR(X/B) \xrightarrow{\id \otimes \Gr^0}\DR(M/B)\otimes f_*(\O_X)\\	
\xrightarrow{\Omega} \DR(M/B) \otimes f_* f^\dag \O_B[-n] \xrightarrow{f_*\dashv f^\dag} \DR(M/B)[-n],
\end{multline*}
where the Kunneth formula can be shown by the base change and the projection formula for $f_*$ in \cite[Prop.~3.10, Rem.~3.11]{BFN}.\footnote{It suffices to consider the Kunneth formula of the associated graded algebras since $(\Gr,\Fil^{-\infty})$ is conservative and $\Fil^{-\infty}\DR(-/B)\simeq \O_B$. We may assume that $M$ is affine by descent since $f_*\Lambda^p\LL_{X/B}$ is perfect. Then the Kunneth map is an equivalence by the arguments in \cite[Lem.~A.7]{BKP}.}
Consequently, we have 
\[\int_{X_M/M,\Omega}^{\star} \ev^*(-):\sA^{p,\star }(Y,d) \to \sA^{p,\star}(M,d-n) \for \star \in \{\lc,\cl,\DR,\emptyset\},\]
where $\ev : X_M:=M\times_B X \to Y$ is the evaluation map.


Even when a $d$-shifted closed $p$-form on the target 
is {\em not} locked, if $d=n-p$, 
we can still find a canonical locking structure on the induced closed form on the mapping stack, 
formally locally on the base (by Proposition \ref{Prop:Locking}).
We compute its underlying function to determine when it glues globally.

\begin{proposition}[Mapping stacks -- Local] \label{Prop:Map}
Let $f:X \to B$ be an $n$-dimensional Calabi-Yau morphism of derived stacks with $\Omega:\O_X \to f^{\dag}\O_B[n]$.
Let $Y$ be a derived stack over $B$ with a closed form $\alpha \in \sA^{p,\cl}(Y/B,n-p)$.
Let $M:=\uMap_B(X,Y)$.
If $B$ is local Artinian with closed point $b\in B(\C)$,
then the underlying (formal) function
\begin{equation}\label{Eq:6.2.1}
\left\langle\int^{\cl}_{X_M/M,\Omega} \ev^* (\alpha) \right\rangle \in \prod_{\pi_0(M)} \sA^0(B,0)
\end{equation}
is equivalent to the locally constant function
\[ (m:X_b \to Y)\in M(\C) \mapsto  \int^{\DR}_{X/B,\Omega}\widetilde{m^*[\alpha]} \in \sA^{\DR}(B/B,0)\simeq \sA^0(B,0),\]
where $X_b:=X\times_{B} \{b\}$ and $\widetilde{(-)}:= ((-)|_{X_b})^{-1}: \sA^{\DR}(X_b/B,n) \simeq  \sA^{\DR}(X/B,n)$.
\end{proposition}

\begin{proof} 
%
It suffices to compute the restriction of the function \eqref{Eq:6.2.1} to each point $m \in M(\C)$.
Since $\int_{X_M/M,\Omega}(-) : \DR(X_M/B) \to \DR(M/B)$ is functorial on $M$, we have:
\begin{align*}
\left[\int_{X_M/M,\Omega}^{\cl}{\ev^*(\alpha)}\right] \mapsto  \int_{X_0/0,\Omega}^{\DR} m^*[\alpha] &\under \sA^{\DR}(M/B,0) \xrightarrow{m^*} \sA^{\DR}(\{b\}/B,0),\\
\int^{\DR}_{X/B,\Omega}\widetilde{m^*[\alpha]} \mapsto \int_{X_b/\{b\},\Omega}^{\DR} m^*[\alpha] &\under \sA^{\DR}(B/B,0) \xrightarrow[\simeq]{b^*} \sA^{\DR}(\{b\}/B,0).
\end{align*}
This completes the proof.
\end{proof}

We now specialize the situation to a smooth projective morphism of classical schemes $f:X \to B$.
The space of {\em horizontal} de Rham cohomology classes is:
\[\xymatrix{
H_{\DR}^*(X/B)^\nabla \ar[r] \ar@{^{(}->}[d] \cart & \varprojlim_{A \to B}  H^*_{\DR}(X_A) \ar@{^{(}->}[d] \\
H_{\DR}^*(X/B) \ar[r] & \varprojlim_{A \to B} H^*_{\DR}(X_A/A),
}\]
where $H^k_{\DR}(-/-):=\pi_0\sA^{\DR}(-/-,k)$ and the limits are taken over maps $A \to B$ from local Artinian schemes $A$.\footnote{This is equivalent to the definition of horizontal sections in \cite[Rem.~3.9]{Blo}.}
We may view a horizontal class $v\in H_{\DR}^*(X/B)^\nabla$ as a ``locally constant'' family of de Rham cohomology classes:
\[b \in B(\C) \mapsto v_b:=v|_{X_b/\{b\}} \in H^*_{\DR}(X_b:=X\times_B\{b\}).\]


\begin{corollary}[Mapping stacks -- Global]\label{Cor:Map}
Let $f:X \to B$ be a smooth projective morphism of classical schemes with a $n$-dimensional Calabi-Yau structure $\Omega$ and
$(h:Y \to B,\theta_{/B})$ be an $(n-2)$-symplectic fibration for $\theta \in \sA^{2,\cl}(Y,n-2)$.
Given a horizontal class $v\in H_{\DR}^{n}(X/B)^\nabla$, consider the open substack 
\[\uMap_B(X,Y)_v:= \{ (X_b \xrightarrow{m} Y_b, b\in B(\C)) : m^*[\theta_b] = v_b \in H_{\DR}^{n}(X_b)\} \subseteq \uMap_B(X,Y),\]
where $Y_b:=Y\times_B\{b\}$ and $\theta_b:=\theta|_{Y_b}$.
Then we have a locked symplectic fibration
\[\uMap_B(X,Y)_v \in \Symp_{B,-2}^{W} \where W:=\int_{X/B,\Omega}^{\DR} v \in H^0_{\DR}(B/B)\simeq \sA^0(B,0).\]
\end{corollary}

We first clarify the statement of Corollary \ref{Cor:Map}.
\begin{itemize}
\item The mapping stack $\uMap_B(X,Y)$ is a derived Artin stack of finite presentation by Lurie's representability theorem \cite{LurDAG}, see \cite[Cor.~3.3]{Toen}.
\item We have $\uMap_B(X,Y)_v\subseteq \uMap_B(X,Y)$ as an open substack since $f:X \to B$ is topologically locally trivial by the Ehresmann theorem.
\item The discrete space $\sA^0(B,0)$ is identified to the set $\pi_0\sA^0(B,0)=H^0(B,\O_B)$.
\end{itemize}

\begin{proof}[Proof of Corollary \ref{Cor:Map}]
By Proposition \ref{Prop:Locking}, we may assume that the base $B$ is local Artinian.
By Proposition \ref{Prop:Map}, it suffices to prove the equality
\[\widetilde{m^*[\theta_{/B}]} = v  \textin \pi_0\sA^{\DR}(X/B,n) = H^n_{\DR}(X/B),\]
for any $(m:X_b \to Y)\in \uMap_B(X,Y)_v(\C)$, where $b\in B(\C)$ is the closed point. 
These two de Rham classes are both horizontal, i.e. elements of
\[\Image\left(H^n_{\DR}(X) \xrightarrow{(-)_{/B}} H^n_\DR(X/B)\right) = H^n_\DR(X/B)^\nabla \textin H^*_\DR(X/B).\]
Since $(-)|_{X_b}:H^n_\DR(X/B)^\nabla \xrightarrow{\simeq} H^n_{\DR}(X_b)$ is an equivalence, the equality
\[\widetilde{m^*[\theta_{/B}]}|_{X_b} = m_b^*[\theta_b] = v_b=v|_{X_b} \textin H^n_{\DR}(X_b)\]
completes the proof.
\end{proof}

We finally consider the moduli of perfect complexes in Theorem \ref{Thm_D:Perf}.

\begin{example}[Moduli of perfect complexes]\label{Ex:Perf}
Let $f:X \to B$ be a smooth projective Calabi-Yau morphism of classical schemes of dimension $4$.
Let $v_*\in H^{2*}_{\DR}(X/B)^{\nabla}$. 
Let $\Perf$ be the stack of perfect complexes \cite{ToVa} which is $2$-shifted symplectic \cite[Thm.~2.12]{PTVV}.\footnote{The stack $\Perf$ is not geometric, but is {\em locally geometric} (i.e. is the union of open substacks that are geometric). The notion of shifted symplectic forms extends to locally geometric stacks.}
Let $\Perf(X/B):=\Map_B(X,\Perf_B)$ 
and
\[\Perf(X/B,v):=\{\text{perfect complexes $E$ on $X_b$ such that $\ch_*(E)=(v_*)_b\in H^{2*}_{\DR}(X_b)$}\}\]
be the open substack with fixed topological Chern character.
Then we have
\[\Perf(X/B,v) \in \Symp_{B,-2}^{W}, \quad\text{where}\quad W:=\int_{X/B} v_2\cup \Omega.\]
Moreover, the underlying function $W$ measures the $(0,4)$-Hodge piece of $v_2$, that is,
\[W=0 \iff v_2 \in \Fil^1H_{\DR}^4(X/B).\]

If $f: X \to B$ is smooth projective Calabi-Yau morphism of dimension $n \geq 5$, then $\Perf(X/B) \to B$ is a $0$-locked $(2-n)$-symplectic fibration by Proposition \ref{Prop:Locking}.
\end{example}

In Example \ref{Ex:Perf}, 
if the base $B$ is {\em reduced} (and $\Perf(X/B,v)$ is non-empty), 
then a horizontal class $v_p\in H_{\DR}^{2p}(X/B)^\nabla$ is a {\em Hodge class} (i.e. lie in $\Fil^p H_{\DR}^{2p}(X/B)$) by the global invariance cycle theorem \cite[Thm.~4.1.1]{Del71} (see \cite[Prop.~11.3.5]{CS}).
In particular, the symplectic form is exact.
Consequently, we have the deformation invariance of the Donaldson-Thomas invariants for Calabi-Yau $4$-folds \cite{CL,BJ,OT}.

\begin{remark}[Deformation invariance]\label{Rem:DefInv}
Let $f:X \to B$ be a smooth projective Calabi-Yau morphism of classical schemes of dimension $4$ and $v_*\in H^{2*}_{\DR}(X/B)^{\nabla}$.
Consider an open substack
\[M\subseteq \Perf(X/B,v)_{\O_X}:=\Perf(X/B,v)\times_{\det,\Pic(X/B),\O_X}B\] 
where $\det:\Perf(X/B) \to \Pic(X/B)$ is the determinant map in \cite{STV}.
Assume that $M$ is a Deligne-Mumford stack (and is a quotient stack of a quasi-projective scheme by a linear action of a linear algebraic group).
Given an orientation of $M \to B$, Theorem \ref{Thm:VLC} gives us a bivariant class
\[[M/B]^\lag : A_*(B) \to A_*(M),\]
since $v_2$ is a Hodge class over $B_\red$ by the invariant cycle theorem \cite{Del71} and $A_*(B)=A_*(B_\red)$.
When $M \to B$ is proper, for any bivariant class $\psi \in A^*(M)$, the function
\[b \in B(\C) \mapsto \int_{[M_b]^\lag} \psi|_{M_b} \in \C \quad\text{is locally constant.}\]

\end{remark}

However, if we only fix the Hilbert polynomials, instead of the Chern characters, then we can have non-exact symplectic fibrations over reduced bases.

\begin{remark}[Non-exact symplectic moduli spaces]\label{Rem:nonexactmoduli}
Consider the open subscheme $B\subseteq |\O_{\PP^5}(6)|$ consists of smooth sextic hypersurfaces in $\PP^5$.
Let $X \to B$ be the universal family.
Consider the Hilbert scheme of planes
\[I(X/B):=\{(S\subseteq X_b,b\in B) : P_S(t) = (t+2)(t+1)/2\}.\]
By \cite[Thm.~1.4]{BKP}, we have an open embedding
$I(X/B)\hookrightarrow (\Perf(X/B)_{\O_X})_{\cl}.$
Hence there is a $(-2)$-shifted symplectic fibration
\[\qquad (RI(X/B) \to B) \in \Symp_{B,-2}\quad\text{such that}\quad RI(X/B)_\cl =I(X/B).\]
In this case, the above $(-2)$-symplectic form is {\em not} exact.
Indeed, in a contractible analytic neighborhood $D\subseteq B$ of $b \in B$, the Chern character $\ch_2(\O_S) \in H^4_{\DR}(X_b)$ of $S\in I(X/B)$ lifts to a horizontal section $\ch_2 \in H^4_{\DR}(X_D/D)^\nabla \simeq H^4_{\DR}(X_b)$
and the function of $(0,4)$-Hodge pieces 
\[W: d \in D \mapsto (\ch_2)_d^{(0,4)} \in H^4_{\DR}(X_d)/F^1_{\Hd}H^4_{\DR}(X_d) \simeq \C\]
is non-zero since $\dim \TT_{\Crit_D(W),b} = \dim \TT_{D,b}-19$ by \cite[Cor.~4.28]{BKP}.
The underlying formal function of the symplectic fibration $RI(X/B) \to B$ (in Proposition \ref{Prop:Locking}) is the restriction of the above function $W$ to the formal neighborhood of $b\in B$ (by Corollary \ref{Cor:Map}) which is non-zero.
Therefore, the symplectic form is not exact.
\end{remark}

\end{document}